\documentclass[11pt]{amsart}

\usepackage{amsmath,amsthm,amscd,amsfonts,wasysym,amssymb,epic,eepic,bbm,tikz-cd,mathtools,mathdots,old-arrows, comment}
\usepackage[pagebackref,colorlinks=true,linkcolor=blue,citecolor=blue]{hyperref}
\usepackage[noabbrev]{cleveref}
\usepackage{MnSymbol}

\allowdisplaybreaks
\newcommand\numberthis{\addtocounter{equation}{1}\tag{\theequation}}
\theoremstyle{plain}
\newtheorem{theorem}{Theorem}[section]
\newtheorem{theorem*}{Theorem}
\newtheorem{lemma}[theorem]{Lemma}
\newtheorem{prop}[theorem]{Proposition}

\newtheorem{defn}{Definition}[section]

\theoremstyle{remark}
\newtheorem*{remark}{Remark}

\newcounter{remarkscounter}

\numberwithin{equation}{section}
\newcommand{\quash}[1]{}

\theoremstyle{definition}

\renewcommand{\bar}{\overline}
\numberwithin{equation}{subsection}

\usepackage{tikz}
  \usetikzlibrary{arrows.meta}
  \usepackage{circuitikz}
  \usepackage{menukeys}
  \usetikzlibrary{3d}
  \usetikzlibrary{calc,shadows,shapes.misc,shapes.symbols}
  \usetikzlibrary{decorations.text}
  \usetikzlibrary{decorations.pathmorphing} % noisy shapes
  \usetikzlibrary{fit}
  \usetikzlibrary{backgrounds}	% drawing the background after the foreground
\usepackage{tikz-cd}

\def\centerarc[#1](#2)(#3:#4:#5);%
    {
    \draw[#1]([shift=(#3:#5)]#2) arc (#3:#4:#5);
    }
\usepackage{quiver}
\allowdisplaybreaks

\title{Eisenstein Cohomology and Critical Values of Certain $L$-Functions: \textit{The Case} $G_2$}
\author{Farid HosseiniJafari}
\date{}

\begin{document}
\maketitle

\begin{abstract}
We establish results on the rationality of ratios of successive critical values of Langlands-Shahidi $L$-functions, as they appear in the constant terms of the Eisenstein series associated with the exceptional group of type $G_2$ over a totally imaginary number field. Furthermore, we prove the rationality of the critical values for each $L$-function in the products, such as the symmetric cube $L$-functions. Our method generalizes the Harder-Raghuram method \cite{HR} to cases where multiple $L$-functions appear in the constant term and involve an exceptional group. Finally, our results on the automorphic version of Deligne's conjecture align with its motivic counterpart, as demonstrated in the recent work of Deligne and Raghuram \cite{deligne2024motives}.
\end{abstract}

\vspace{25pt}
\section{Introduction} 
This paper aims to establish a connection between the cohomology of arithmetic groups and the special values of Langlands-Shahidi $L$-functions, extending the Harder-Raghuram method (\cite{HR, HarderB, Harder2010, H2, RagTI,Raghurampadic,RagNH}) to cases involving multiple $L$-functions appearing in constant terms of Eisenstein series. Specifically, we focus on the exceptional group $G_2$, which played a key role in developing the Langlands-Shahidi method \cite{Lang,SH89,Shahidi1991LanglandsCO,Shahidi1981OnCL}. Consequently, we prove rationality results for the ratios of successive critical values of Langlands-Shahidi $L$-functions associated with $G_2$ over a totally imaginary field.

 Harder outlined an ambitious program to establish a bilateral relation between the cohomology of arithmetic groups and the special values of Langlands' automorphic $L$-functions \cite{HarderB}, building on his pioneering work on $GL_2$ \cite{H2}. In his collaboration with Raghuram \cite{HR}, they explored one direction of this relationship by proving the rationality of the ratio of successive critical points of Rankin-Selberg $L$-functions, providing a cohomological interpretation of some aspects of the Langlands-Shahidi method through the study of the Eisenstein cohomology of the ambient group $GL_{N}$ over a totally real number field. They proposed that their method could be extended to establish rationality results for Langlands-Shahidi $L$-functions attached to any reductive group. However, it was not clear how their method could be adapted to cases involving multiple $L$-functions or to the case of exceptional groups. We complete the program for the exceptional group of type $G_2$, by studying the Eisenstein cohomology of these groups and providing a cohomological interpretation of the Langlands-Shahidi method in the case of $G_2$, where multiple $L$-functions appear in the constant term of Eisenstein series attached to $G_2$. We now proceed to state the main results of this paper.

Let $F$ be a totally imaginary number field that contains a \textit{CM}-subfield. Denote the rings of adeles of $F$ by $\mathbb{A}_{F}$, and for simplicity when $F=\mathbb{Q}$ denote it by $\mathbb{A}$. Let $E$ be a ``sufficiently'' large finite Galois extension of $\mathbb{Q}$ that contains a copy of $F$, and fix an embedding $\iota: E \to \mathbb{C}$. 

Let $Res_{F/\mathbb{Q}}(GL_2)$ be a rational group defined by the Weil restriction of scalars of a split group of type $GL_2/F$. Let $^{\iota}\!\sigma = {^{\iota}\sigma_f} \times {^{\iota}\sigma_\infty}$ be a cohomological cuspidal representation of $Res_{F/\mathbb{Q}}(GL_2/F)(\mathbb{A}) = GL_2(\mathbb{A}_F)$, formed as explained in \ref{sec:cohrep}. Denote the central character of ${^{\iota}\sigma}$ by $\omega_{{^{\iota}\sigma}}$, and the adjoint cube representation of $GL_2$ by $Ad^3({^{\iota}\sigma})$ \cite{SH89}.

Following Deligne \cite[Prop-Def 2.3]{Deligne1979ValeursDF}, define the critical points of the (completed) Langlands-Shahidi $L$-functions as the (half-)integers at which the infinite part of the $L$-function and its counterpart on the other side of the functional equation are both finite. The main theorems of this paper are as follows:
\begin{theorem*}[Theorem \ref{thm:mainbeta}]
    Assume that $^{\iota}\sigma$ is non-monomial. For any two successive critical points $m$ and $m+1$ of both $L$-functions $L(s,Ad^{3}({^{\iota}\sigma}))$ and $L(2s,\omega_{{^{\iota}\sigma}})$, we have:
    \begin{enumerate}
        \item The vanishing of $L(m+1,Ad^{3}({^{\iota}\sigma}))L(2m+1,\omega_{{^{\iota}\sigma}})$ is independent of the choice of the embedding of $E$ in $\mathbb{C}$.
        \item Suppose $L(m+1,Ad^{3}(^{\iota}\sigma)) \neq 0$; then
         \[
             |\delta_{F/\mathbb{Q}}|^{5/2} \frac{L(m,Ad^3(^{\iota}\sigma))L(2m,\omega_{^{\iota}\sigma})}{L(m+1,Ad^3(^{\iota}\sigma))L(2m+1,\omega_{^{\iota}\sigma})}\in \iota(E) \subset \bar{\mathbb{Q}},
            \]
        where $\delta_{F/\mathbb{Q}}$ is the absolute discriminant of $F$.
        \item  For any $\gamma \in Gal(\bar{\mathbb{Q}}/\mathbb{Q})$:
   \begin{multline*}
       \gamma\left(|\delta_{F/\mathbb{Q}}|^{5/2} \frac{L(m,Ad^3(^{\iota}\sigma))L(2m,\omega_{^{\iota}\sigma})}{L(m+1,Ad^3(^{\iota}\sigma))L(2m+1,\omega_{^{\iota}\sigma})}\right) = \\ \epsilon(\gamma, ^{\iota}w)\epsilon(\gamma, ^{\iota}w')|\delta_{F/\mathbb{Q}}|^{5/2} \frac{L(m,Ad^3(^{\gamma\circ\iota}\sigma))L(2m,\omega_{^{\gamma\circ\iota}\sigma})}{L(m+1,Ad^3(^{\gamma\circ\iota}\sigma))L(2m+1,\omega_{^{\gamma\circ\iota}\sigma})},
   \end{multline*}
   where $\epsilon(\gamma, ^{\iota}w)$ and $\epsilon(\gamma, ^{\iota}w')$  are signs as defined in Definition \ref{def:sign}.
    \end{enumerate}
\end{theorem*}

\begin{theorem*}[Theorem \ref{thm:mainalpha}]\label{thm1}
   Suppose $m$ and $m+1$ are the critical points of the $L$-functions $L(s, \;^{\iota}\sigma)$, $L(2s,\omega_{^{\iota}\sigma})$, and $L(3s, \;^{\iota}\sigma\otimes \omega_{^{\iota}\sigma})$. Then, we have the following: 
     \begin{enumerate}
         \item The vanishing of $L(m+1, \;^{\iota}\sigma)L(2m+1,\omega_{^{\iota}\sigma})L(3m+1, \;^{\iota}\sigma\otimes \omega_{^{\iota}\sigma})$ is independent of the choice of embedding from $E$ in $\mathbb{C}$. 
          \item Suppose $L(m+1,\;^{\iota}\sigma)\neq0$; then:
         \[
             |\delta_{F/\mathbb{Q}}|^{5/2} \frac{L(m, \;^{\iota}\sigma)L(2m,\omega_{^{\iota}\sigma})L(3m, \;^{\iota}\sigma\otimes \omega_{^{\iota}\sigma})}{L(m+1, \;^{\iota}\sigma)L(2m+1,\omega_{^{\iota}\sigma})L(3m+1, \;^{\iota}\sigma\otimes \omega_{^{\iota}\sigma})}\in \iota(E) \subset \bar{\mathbb{Q}}.
            \]
        \item  For any $\gamma \in Gal(\bar{\mathbb{Q}}/\mathbb{Q})$:
   \begin{multline*}
       \gamma\left(|\delta_{F/\mathbb{Q}}|^{5/2} \frac{L(m, \;^{\iota}\sigma)L(2m,\omega_{^{\iota}\sigma})L(3m, \;^{\iota}\sigma\otimes \omega_{^{\iota}\sigma})}{L(m+1, \;^{\iota}\sigma)L(2m+1,\omega_{^{\iota}\sigma})L(3m+1, \;^{\iota}\sigma\otimes \omega_{^{\iota}\sigma})}\right) = \\ \epsilon(\gamma, ^{\iota}w)\epsilon(\gamma, ^{\iota}w')|\delta_{F/\mathbb{Q}}|^{5/2} \frac{L(m, \;^{\gamma\circ\iota}\sigma)L(2m,\omega_{^{\gamma\circ\iota}\sigma})L(3m, \;^{\gamma\circ\iota}\sigma\otimes \omega_{^{\gamma\circ\iota}\sigma})}{L(m+1, \;^{\gamma\circ\iota}\sigma)L(2m+1,\omega_{^{\gamma\circ\iota}\sigma})L(3m+1, \;^{\gamma\circ\iota}\sigma\otimes \omega_{^{\gamma\circ\iota}\sigma})},
   \end{multline*}
   where $\epsilon(\gamma, ^{\iota}w)$ and $\epsilon(\gamma, ^{\iota}w')$  are signs as defined in Definition \ref{def:sign}.
     \end{enumerate}
\end{theorem*}
We need to interpret the $L$-functions above as Langlands-Shahidi $L$-functions associated with an ambient group. The rank-one Eisenstein cohomology of this ambient group then provides the source of rationality for the ratios of the critical values of these $L$-functions. The Harder-Raghuram method has been successfully applied to various cases (\cite{HR,Ra21A,RagTI,BHA-Rag,KR}), with a unifying feature of these works being their focus on classical groups that involve a single $L$-function contributing to the constant term of the Eisenstein series.  In this paper, we establish the first instance of the Harder-Raghuram method in which the ambient group is an exceptional group where multiple $L$-functions appear in the constant term of the Eisenstein series. A key step in the Harder-Raghuram method is to provide a cohomological interpretation of aspects of the Langlands-Shahidi method. The case $G_2$ plays a key role in Shahidi’s early work \cite{Shahidi1978FunctionalES,Shahidi1981OnCL,SH89,ShBook} in developing the Langlands-Shahidi method. We expect our work to play a similar role, paving the way for studying the rationality of Langlands-Shahidi $L$-functions in numerous cases, where we deal with multiple $L$-functions and/or exceptional groups.

\subsection*{Outline of the paper.} 

Let $G = Res_{F/\mathbb{Q}}(G_2/F)$, and $S^G$ be a (adelic) locally symmetric space with level structure $K_f$ attached to $G$ as defined in Section \ref{sec:lss}. The open compact subgroup $K_f$ is assumed to be small enough and neat. Let $E$ be a sufficiently large Galois extension of $\mathbb{Q}$ containing a copy of $F$. Let $\mathcal{M}_{\lambda}$ be a sheaf attached to an irreducible representation of $G \times E$ as usual, then one can define the sheaf cohomology group $H^{\bullet}(S^{G},\mathcal{M}_{\lambda})$ 
which is a Hecke module. Let $\bar{S}^G = S^{G} \cup \partial S_G$ be the Borel-Serre compactification, then the cohomology of the compactified space and the space itself ``coincide'' due to the homotopy equivalency of two spaces. We have the following long exact sequence:
\begin{equation}
    \cdots \to H^{\bullet}_{c}(S^{G},\mathcal{M}_{\lambda}) \xrightarrow{\mathfrak{i}^{\bullet}} H^{\bullet}(S^{G},\mathcal{M}_{\lambda}) \xrightarrow{\mathfrak{r}^{\bullet}} H^{\bullet}(\partial S^{G},\mathcal{M}_{\lambda}) \to \cdots.
\end{equation}

Following \cite{HarderB} and \cite{SchK}, appealing to Kostant's theorem \cite{kos} one can describe the boundary cohomology in terms of algebraic (unnormalized) parabolic induction of the cohomology groups of the Levi factor $M$ of conjugacy classes of parabolic subgroups $P$ with sheaves originating from an irreducible representation of $M \times E$ with highest weight $\mu$ obtained by twisted action of Kostant representative as described in Section \ref{boundcomp} and \ref{sec:kos}. We can compatibly define the arithmetic cohomology of the Levi factors and derive the same long exact sequence as above. The image of $\mathfrak{i}^{\bullet}$, denoted by $H^{\bullet}_{!}(S^{M},\mathcal{M}_{\mu})$, is known as inner cohomology. It admits an isotypical decomposition for ``sufficiently large" $E$, \cite[Section 2.3.5]{HR}. In Section \ref{ch:SVL}, we define the strongly inner cohomology group \(H^{\bullet}_{!!}(S^{M}, \mathcal{M}_{\mu}) \subset H^{\bullet}_{!}(S^{M}, \mathcal{M}_{\mu})\). After a base change using \(\iota: E \to \mathbb{C}\), the strongly inner cohomology precisely captures the cuspidal cohomology group. In Section \ref{sec:cohrep}, for a Hecke summand ${^{\iota} \sigma_f} \in H_{!!}(S^{M},\mathcal{M}_{\mu})\otimes_{E,\iota} \mathbb{C}$, we explicitly define an infinitesimal representation ${^{\iota}\sigma_{\infty}}$ in terms of the highest weight ${^{\iota}\mu}$ such that ${^{\iota}\sigma} = {^{\iota}\sigma_f} \otimes {^{\iota}\sigma_{\infty}}$ is a cuspidal cohomological representation.

In Section \ref{ch:LS}, we review the Langlands-Shahidi method to obtain the ratio of the product of automorphic $L$-functions appearing in the constant term of the Eisenstein series. To prove the rationality of the ratios of these $L$-functions, we develop a cohomological framework that aligns with the cohomological Langlands-Shahidi method (\cite[Section 9.2.1]{HarderB}, \cite[Figure 6.18]{HR}). The general concept of these constructions is depicted in Figure \ref{fig:proof}.

In Figure \ref{fig:proof}, the outer diagram illustrates the Langlands-Shahidi method for a standard maximal parabolic subgroup $P=MU$, where $\tilde{\sigma}$ is the contragredient and $Ind$ denotes the normalized parabolic induction evaluated at the point of evaluation $k$, see Definition \ref{def:poe}. At the point of evaluation, the normalized and algebraic parabolic inductions coincide and we can relate the cohomological and the automorphic frameworks (the dashed lines). Moreover, the map $M$ is the standard intertwining map, $Eis_{P}$ is the Eisenstein summation, and $\mathcal{F}_{P}$ is the constant term along $P$, see Equation \ref{def:constant} and its following remark.

The inner diagram illustrates the cohomological counterpart of the Langlands-Shahidi machinery. We use the notations introduced in Section \ref{sec:MD}, here
\[
I_{b}^{S}(\sigma_f)_{w_{cl}} := ^{a}\!\operatorname{Ind}_{P(\mathbb{A}_f)}^{G(\mathbb{A}_f)}(H^{b_2^{F}}(\underline{S}^{M}, \mathcal{M}_{w_{cl}.\lambda}(\sigma_f))^{K_f},
\]
where the subscript $b$ indicates that we are considering the bottom-degree situation. The term $w_{cl}$ denotes a specific Kostant representation, as described earlier, ensuring that the cohomology of the Levi factor is non-trivially contributes to boundary cohomology at bottom-degree $q_b$. The Kostant representative $w_{cl}$ is determined by one of the key steps in this paper, the combinatorial lemma (Lemma \ref{lem:combalpha} and \ref{lem:combbeta}).

\begin{figure}[h!]
    \centering
    \includegraphics[width=1\linewidth]{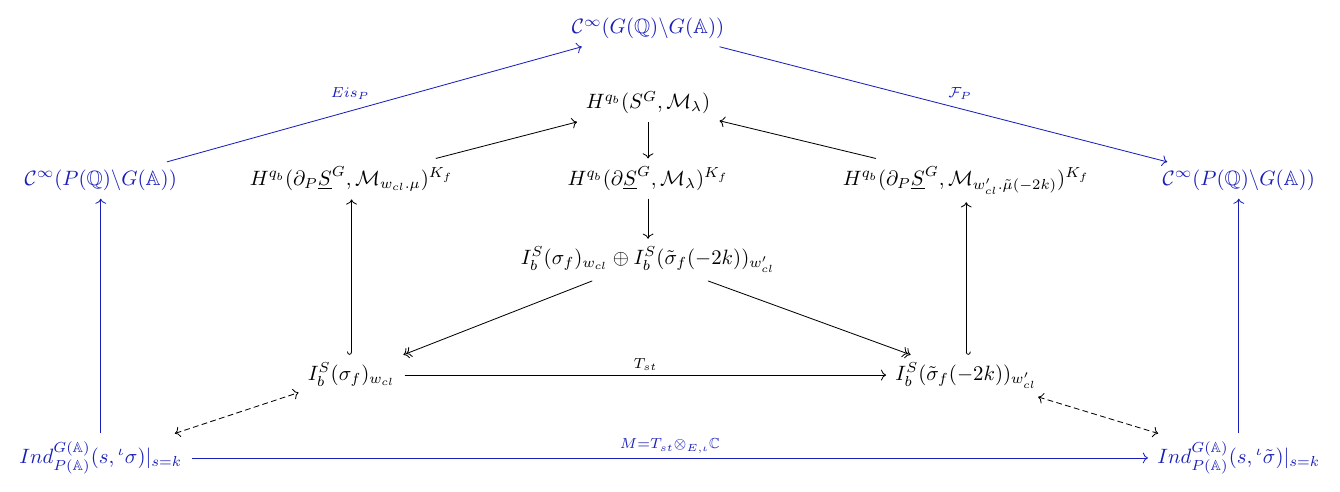}
    \caption{Cohomological Langlands-Shahidi method via Eisenstein cohomology}
    \label{fig:proof}
\end{figure}

The Hecke summand $\sigma_f$ can be naturally decomposed into its local components, allowing the induced module to be expressed as a tensor product of these components. In the cohomological setting, the intertwining operator $T_{st}$ is represented as a tensor product of local intertwining maps, each of which is $E$-linear, justifies the notation used in the Figure \ref{fig:proof}. At the point of evaluation and after a base change via $\iota: E \to \mathbb{C}$, we prove that the local intertwining operators coincide with their automorphic counterparts, contributing non-trivially to $L$-values in a manner consistent with the automorphic setting. The proof for non-archimedean places is presented in Section \ref{sec:Nonint}, following the approach outlined in \cite{Raghurampadic} for the general case. The archimedean case is addressed in Section \ref{sec:arch} and is more intricate. The proof begins by showing that, at the transcendental level, the induced modules evaluated at $k$ are irreducible and the intertwining maps are isomorphisms (Propositions \ref{prop:indirrbeta} and \ref{prop:indirrbeta}). Using cocycle decomposition in the automorphic setting and Delorme's lemma in the cohomological setting, the calculation is reduced to rank-one operators. Referring to the computations in \cite[Section 4.1]{RagTI} for $GL_2$, we demonstrate that the local intertwining operators at archimedean places contribute non-trivially to the ratios of archimedean local $L$-factors in our case, as established in Propositions \ref{archinterthm} and \ref{archinterthma}.

The next step is to state and prove a version of the Manin-Drinfeld principle \cite[Theorem 5.12]{HR} for our case. As outlined in Theorem \ref{thm:mnnb}, this principle asserts that the direct sum 
\begin{equation}\label{summand}
I_{b}^{S}(\sigma_f)_{w_{cl}} \oplus I_{b}^{S}(\sigma_f (-2k))_{w'_{cl}},
\end{equation} 
splits off as an isotypical Hecke summand within the Hecke module $H^{q_b}(\partial \underline{S}^{G}, \mathcal{M}_{\lambda})^{K_f}$. The dimensions of the individual summands are each $\mathsf{k}$, resulting in a total dimension of $2\mathsf{k}$ for the direct sum. 

We first establish a multiplicity-one result specific to our case to prove Theorem \ref{thm:mnnb}. In particular, we show that the algebraically induced representations discussed earlier are not equivalent almost everywhere to any representations induced from other parabolic subgroups. Furthermore, we demonstrate that $^{a}\!\operatorname{Ind}_{P(\mathbb{A}_{f})}^{G(\mathbb{A}_{f})}(\sigma_{f})$ is not equivalent almost everywhere to any other parabolically induced representation along other conjugacy classes of parabolic subgroups. The proof is an extension of the approach in \cite[Propositions 2.3.1 and 2.3.3]{mundy} for $G_2$ over a totally real number field to the case of a totally imaginary number field.

The final component is our main theorem on the rank-one Eisenstein cohomology for the group $G$, presented as Theorem \ref{thm:eis}. This theorem establishes that the image of Eisenstein cohomology forms a $\mathsf{k}$-dimensional $E$-subspace within the isotypic component described in \ref{summand}, which has a total dimension of $2\mathsf{k}$. The rationality results are then obtained by following the approach outlined in \cite{HR}. In our construction, ``\textit{the Eisenstein cohomology is analogous to a line within a two-dimensional plane, where the slope corresponds to the ratio of the products $L$-values}'' discussed earlier.

In Section \ref{sec:critical}, we explicitly calculate the critical integer for the $L$-factors and determine the set of critical integers for the product of $L$-factors. We show that the set of critical values forms an interval of half-integers. To apply the rationality results outlined for the evaluation point to all the critical points, we use Tate twists and track the combinatorial conditions that determine $w_{cl}$. This imposes a bound on the possible Tate twists that covers all the points in the critical set of the product of $L$-functions—``\textit{no more and no less!}''—as stated and proved in the combinatorial lemma (see Lemmas \ref{lem:combbeta} and \ref{lem:combalpha}). Thus, we can extend our main rationality results to all critical points of the products of $L$-functions by studying the rationality results for the evaluation point alone.

Furthermore, we observe that for a given representation, the critical set of the product of $L$-functions covers all critical integers of one and only one of the $L$-functions involved. By imposing a specific combinatorial condition on the representation ${^{\iota}\sigma}$, we can extend our main rationality results to any of the $L$-functions that appear in the product at all critical points, assuming that the rationality results for the other $L$ functions involved are known. This combinatorial condition is detailed in Section \ref{sec:ttr}. As an application, we can derive similar rationality results for $L(s, \mathrm{Sym}^3({^{\iota}\sigma}))$, where $\mathrm{Sym}^3$ denotes the symmetric cube transfer for $GL_2$. This is achieved by leveraging our results for $L(s, \mathrm{Ad}^3({^{\iota}\sigma}))$ and twisting the representation by the central character. Rationality results for $\mathrm{Sym}^3$ have previously been established in special cases by \cite{GarHar,KSSpecial,GrobRasym,Januszewski2011} and in the general case by Raghuram \cite{Racub}, who used functoriality to derive these results by fstudying the Eisenstein cohomology of $GL_n \times GL_{n-1}$. Although our approach is closely aligned with that of \cite{Racub}, our results are obtained independently of functoriality.

\vspace{5pt}
\noindent\textbf{Acknowledgements.} This article is a revised version of my Ph.D. thesis. I am grateful to my advisors Freydoon Shahidi and A. Raghuram for suggesting the problem and their guidance and encouragement. I also extend my gratitude to Gunter Harder for his profound work, which inspired this paper. I would also like to thank Jayce Getz for many useful discussions and numerous suggestions for the improvement of the manuscript.

\pdfbookmark{TABLE OF CONTENTS}{Contents}
\tableofcontents

\section{Preliminaries}\label{sec:Preliminaries}
\subsection{The base field}
Let $F$ be a number field of degree $d_F=[F:\mathbb{Q}]$. The set of archimedean places of $F$, denoted by $S_{\infty}$, corresponds to the embeddings of $F$ into $\mathbb{C}$, up to the action of $\text{Aut}(\mathbb{C}/\mathbb{R})$. Fix the notation for the pair of complex conjugate embeddings $\eta_v, \bar{\eta}_v \in Hom_{\mathbb{Q}}(F,\mathbb{C})$ associated with the complex place $v$.  In this paper, we are interested in totally imaginary number fields. In this case $d_F=n_2 =2\mathsf{r}$ where $\mathsf{r}$ is the number of complex places.

A totally imaginary number field $F$ contains a unique maximal totally real subfield $F_0$, it is known that $F_0$ has at most one totally imaginary quadratic extension within $F$ \cite{Weil79}. A totally imaginary quadratic extension of a totally real number field is called a CM-field. Following \cite{RagTI}, define two cases:
\begin{enumerate}
    \item[2.1.] $\mathbf{CM-case:}$  In this case, there is an intermediate CM-field extension $F_1$ of $F_0$ inside $F$.
    \item[2.2.] $\mathbf{TR-case:}$ In this case, there is no CM-subfield of $F$ containing $F_0$. We will denote $F_1 = F_0$ for future reference.
\end{enumerate}
For a number field $F$, the ring of adeles of $F$ is denoted by $\mathbb{A}_F$. If $F=\mathbb{Q}$, denote the rings of adeles of $\mathbb{Q}$ with $\mathbb{A}= \mathbb{A}_f \times \mathbb{A}_{\infty}$ where $\mathbb{A}_f$ is the ring of finite adeles. 

\subsection{The groups}
Let $G_0/F$ be a split exceptional group of type $\mathbf{G_2}$ defined over a totally imaginary number field $F$. The group $G_{0}$ is both adjoint and simply connected \cite{Zampera}. Fix a split Cartan subgroup $T_0$ in $G_0$. Let $B_0 = T_0 U_0$ be a fixed Borel subgroup of $G_0$ containing $T_0$. Let $\Delta = \{\alpha, \beta\}$ denote the set of simple roots, where $\alpha$ is the short root and $\beta$ is the long root of $T_0$. Let $\Phi$ and $\Phi^{+}$ be the sets of all roots and positive roots of $G_0$, respectively, corresponding to the chosen Borel subgroup. Then:
\[
    \Phi^{+} = \{\alpha,\; \beta,\; \alpha+\beta,\; 3\alpha+2\beta,\; 2\alpha+\beta,\; 3\alpha+\beta\}.
\]

\begin{figure}[htp]
    \centering
    \resizebox{0.4\textwidth}{0.4\textwidth}{
    \begin{tikzpicture}[
        % arrow heads for all lines (with narrower arrow head width)
        -{Straight Barb[bend,
           width=\the\dimexpr10\pgflinewidth\relax,
           length=\the\dimexpr12\pgflinewidth\relax]},
      ]
      \draw [opacity =0,dashed,fill=blue!50,fill opacity=0.2]{(0,0) -- (60:4.6) -- (0,4) -- cycle}  ;
        % straight arrows
        \foreach \i in {0, 1, ..., 5} {
          \draw[thick, blue] (0, 0) -- (\i*60:2);
          \draw[thick, blue] (0, 0) -- (30 + \i*60:3.5);
        }
        % arc arrow
        \centerarc[thick,<->, dashed,red](0,0)(0:5*30:1.5cm);
       
        % annotations
        \node[right] at (2, 0) {$\alpha$};
        \node[above left, inner sep=.2em] at (5*30:3.5) {$\beta$};
        \node[above] at (4*30:2) {$\alpha+\beta$};
        \node[above] at (2*30:2) {$\gamma_s = 2\alpha+\beta$};
        \node[above] at (3*30:3.5) {$\gamma_l = 3\alpha+2\beta$};
        \node[above right] at (30:3.5) {$3\alpha+\beta$};
        \node[right] at (15:1.5) {\color{red}$\Phi^{+}$};
    \end{tikzpicture}}
    \caption{Root Lattice of $\mathbf{G_2}$, Dominant chamber is shaded}
    \label{fig:G2}
\end{figure}
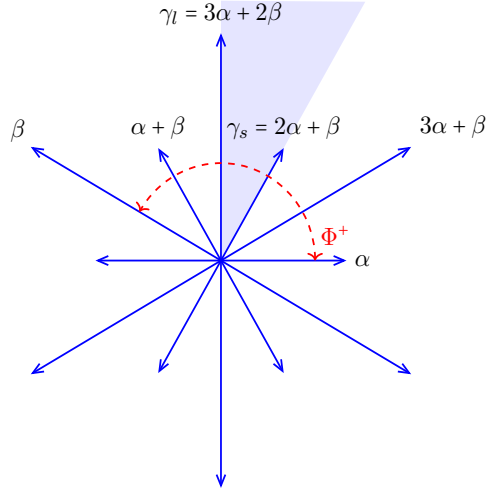%G2 Lattice
The fundamental weights are $\gamma_s = 2\alpha+\beta$ and $\gamma_l = 3\alpha+2\beta$. The half sum of positive roots is
\begin{equation}\label{eq:halfsum}
    \rho_{G} = \frac{1}{2}\sum_{\gamma \in \Phi^{+}} \gamma = 5\alpha+3\beta= \gamma_l+\gamma_s.
\end{equation}

% The dual group $G_{0}^{\vee}$ is again a group of type $\mathbf{G_2}$. The dual root of the simple short root of $G_{0}$ will be the simple long root of $G_{0}^{\vee}$, and the dual of a long root will be the short one, therefore $\Delta^{\vee}=\{\alpha^{\vee}, \beta^{\vee}\}$.

% Let $( \; ,\; )$ denote the standard inner product of the vector system generated by the simple roots and $\langle \theta_1 , \theta_2 \rangle = \frac{2(\theta_1,\theta_2)}{(\theta_2,\theta_2)}$,  then we have:
% \[
% \begin{pmatrix}
% (\alpha,\alpha) & (\alpha,\beta) \\
% (\beta,\alpha) & (\beta,\beta) 
% \end{pmatrix} = \begin{pmatrix}
% 2 & -3 \\
% -3 & 6
% \end{pmatrix}, \qquad \begin{pmatrix}
% \langle \alpha,\alpha \rangle & \langle \alpha,\beta\rangle \\
% \langle \beta,\alpha\rangle  & \langle \beta,\beta \rangle
% \end{pmatrix} = \begin{pmatrix}
% 2 & -1 \\
% -3 & 2
% \end{pmatrix},
% \]
% the matrix on the right-hand side above is the Cartan matrix of $G_0$.

The Weyl group of $G_{0}$ denoted by $W_{0}$, is isomorphic to the dihedral group $D_{6}$ with 12 elements. The Weyl group acts naturally on the root lattice of $G_{0}$. It is generated by the simple reflections $w_{\alpha} :=s_{\alpha}$ and $w_{\beta} :=s_{\beta}$ about the line perpendicular to $
\alpha$ and $\beta$, respectively. The reflection $s_{\theta}$ for any $\theta \in \Phi$ is defined as:
\[
    s_{\theta} (\vartheta ):=  \vartheta - \langle \vartheta,\theta^{\vee} \rangle \theta, \quad \forall \vartheta \in \Phi.
\]
For simplicity, denote $w_{\theta_2}w_{\theta_1}$ by $w_{\theta_1\theta_2}$ for any positive roots $\theta_1$ and $\theta_2$. Then:
\[
W=\{1,\;w_{\alpha}, \;w_{\beta}, \;w_{\alpha\beta}, \;w_{\beta\alpha}, \;w_{\alpha\beta\alpha}, \;w_{\beta\alpha\beta},\; w_{\alpha\beta\alpha\beta},\; w_{\beta\alpha\beta\alpha},\; w_{\alpha\beta\alpha\beta\alpha},\; w_{\beta\alpha\beta\alpha\beta},\; w_{G}\},
\]
where $w_{G}=w_{\alpha\beta\alpha\beta\alpha\beta}=w_{\beta\alpha\beta\alpha\beta\alpha}$ is the longest element of $W$ sending all the positive roots to the negative ones, and more precisely it sends any $\theta \in \Phi^{+}$ to $-\theta$. As $w_{\alpha}^2 = w_{\beta}^2 = 1$, the inverse of each Weyl element is simply the element with a subscript in the inverse order. Let $l(w)$ be the length of an element $w\in W$, then on has:
\[
    l(w) = \#\{\theta \in \Phi^{+} \;|\; w(\theta) \in \Phi^{-}\}, 
\]
this is equivalent to the number of elements in the subscript.  We have $0\leq l(w) \leq 6$, where elements with odd length represent the reflection of the single positive root, and elements with even length are the rotations of angles which are multiples of $\pi/3$.

The character group of $T_0$ denoted by $X^{*}(T_0)$, and the co-character group denoted  by $X_{*}(T_0)$ of $T_0$, can be described explicitly as:
\begin{equation}\label{eq:char}
X^{*}(T_0) \simeq \mathbb{Z}(2\alpha+\beta)\oplus \mathbb{Z}(\alpha+\beta), \quad X_{*}(T_0) \simeq \mathbb{Z}(2\alpha+\beta)^{\vee} \oplus \mathbb{Z}(\alpha+\beta)^{\vee}.
\end{equation}
Moreover, $T_0(F)$ can be identified with $F^{*}\times F^{*}$ as:
\begin{align*}
    t_0 : T_0(F) &\simeq F^{*}\times F^{*},\\
    t &\mapsto ((2\alpha+\beta)(t), (\alpha+\beta)(t)).
\end{align*}
More explicitly, for $x \in F^{\star}$ we have:
\begin{align*}
    \alpha(t_{0}^{-1}(a,b)) &=ab^{-1}, & \beta(t_{0}^{-1}(a,b)) &= a^{-1}b^{2},\\
    \alpha^{\vee}(x) &=t_{0}^{-1}(x,x^{-1}), & \beta^{\vee}(x) &= t_{0}^{-1}(1,x).
\end{align*}
Furthermore, under the map $t_0$ a character $\chi$ of $T_0(F)$ can be represented by $\chi_1\otimes \chi_2$ where $\chi_1$ and $\chi_2$ are characters of $F^{*}$.

The group $ G_0 $ has three proper standard parabolic $ F $-subgroups: the minimal standard parabolic $ F $-subgroup $ B_0 $, and two maximal standard parabolic $ F $-subgroups, $ P_{0,\alpha} $ and $ P_{0,\beta} $, corresponding to the roots $\alpha$ and $\beta$, respectively. For $\theta \in \Delta$, the Levi decomposition $ P_{0,\theta} = M_{0,\theta}U_{0,\theta} $ holds, where $ U_{0,\theta} $ is the unipotent radical, and $ M_{0,\theta} $ is the Levi factor containing $ T_0 $ and $\theta$ as its simple root. We define an equivalence relation between parabolic $ F $-subgroups of $ G_0 $, where two parabolic $ F $-subgroups are equivalent if their Levi factors are conjugated by an element in $ G_0(F) $. The equivalence class of a parabolic $ F $-subgroup $ P_0 $ is called an associate class and is denoted by $[P_0]$. All standard parabolic subgroups of $ G_0 $ are self-associate. Moreover, for $\theta \in \Delta$, the unipotent subgroup $ U_{0,\theta} $ is generated by the root spaces corresponding to the five roots in $ \Phi^{+} \setminus \{\theta\} $, making $ U_{0,\theta} $ a subgroup of dimension 5.

For any $\vartheta \in \Phi^{+}$, let $X_{\vartheta}:  F\to G_{0}(F)$ be the corresponding one-dimensional space and then by the application of Jacobson-Morozov theorem there is a map $\phi_{\vartheta}: SL_2(F) \to G_{0}(F)$ where for $x \in F^{*}$:
\[
    \phi_{\vartheta}\left(\begin{pmatrix}
1 & x \\
0 & 1 
\end{pmatrix}\right) = X_{\vartheta}(x), \qquad \phi_{\vartheta}\left(\begin{pmatrix}
1 & 0 \\
x & 1 
\end{pmatrix}\right) = X_{-\vartheta}(x), \qquad \phi_{\vartheta}\left(\begin{pmatrix}
x & 0 \\
0 & x^{-1} 
\end{pmatrix}\right) = \vartheta^{\vee}(x).
\]
For $\theta \in \Delta$, the Levi factor defined to be the centralizer of $\gamma^{\vee}(F^{*})$, where $\gamma$ is the fundamental root perpendicular to $\theta$, in $G_{0}(F)$. The group generated by $X_{\theta}(F^{*})$, $X_{-\theta}(F^{*})$, and $\gamma^{\vee}(F^{*})$ in $G_0(F)$ is the Levi factor $M_{0,\theta} \simeq GL_2(F)$ \cite{SH89}. We can extend the map $\phi_{\theta}: SL_2(F) \to M_{0,\theta}(F)$ to the map $\phi_{\theta}: GL_2(F) \to M_{0,\theta}(F)$, by defining:
\[
    \phi_{\theta}\left(\begin{pmatrix}
    x & 0 \\
    0 & 1 
    \end{pmatrix}\right) = \gamma^{\vee}(x).
\]
Then, the restriction of this map to the diagonal matrices will provide us with a parametrization of the maximal torus $\phi_{\theta}|_{F^{*}\times F^{*}}: F^{*}\times F^{*} \to T_{0}(F)$. Denote this map as $t_{\theta}$. Let us outline this parametrization for both short and long simple roots.

The map $t_{\alpha}: F^{*}\times F^{*} \simeq T_0(F)$ defined as above, is adjusted to the short root $\alpha$ as the simple root of $M_{0,\alpha}$ and for $t \in T_{0}(F)$ is given by:
\[
    t_{\alpha}(\text{diag}((2\alpha+\beta)(t), (\alpha+\beta)(t))) = t.
\]
It gives us the parametrization of the maximal torus inside of the $M_{0,\alpha}$ as
\begin{align*}
    \alpha(t_{\alpha}(a,b)) &=ab^{-1}, & \beta(t_{\alpha}(a,b)) &= a^{-1}b^{2}.
\end{align*}
    
Another parametrization $t_{\beta}:F^{*}\times F^{*}\to T_0(F)$ adjusted to the long root $\beta$ as the simple root of $M_{0,\beta}$ and for $t \in T_{0}(F)$ is given by:
\begin{equation}
    t_{\beta}(\text{diag}((\alpha+\beta)(t), (\alpha)(t))) = t.
\end{equation}
Again, it gives us the parametrization of the maximal torus inside of the $M_{0,\beta}$ as
\begin{align*}
    \alpha(t_{\beta}(a,b)) &=b,& \beta(t_{\beta}(a,b)) &= ab^{-1}.
\end{align*}

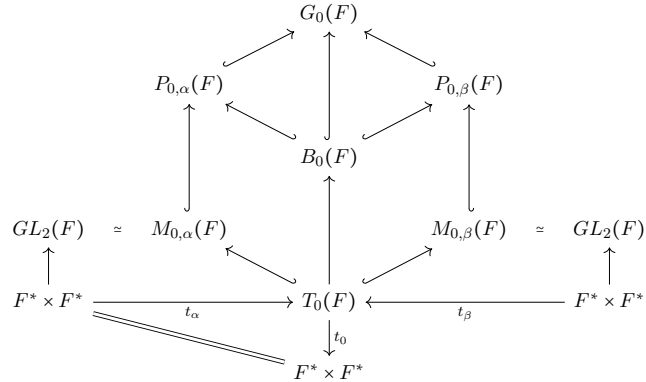
\begin{figure}[ht]
    \centering
    \[\adjustbox{scale=0.7,center}{\begin{tikzcd}
	&& {G_0(F)} \\
	& {P_{0,\alpha}(F)} && {P_{0,\beta}(F)} \\
	&& {B_{0}(F)} \\
	{GL_2(F)} & {M_{0,\alpha}(F)} && {M_{0,\beta}(F)} & {GL_2(F)} \\
	{F^* \times F^*} && {T_0(F)} && {F^* \times F^*} \\
	&& {F^* \times F^*}
	\arrow[hook, from=2-2, to=1-3]
	\arrow[hook', from=2-4, to=1-3]
	\arrow[hook, from=3-3, to=1-3]
	\arrow[hook, from=3-3, to=2-4]
	\arrow[hook', from=3-3, to=2-2]
	\arrow[hook, from=4-2, to=2-2]
	\arrow[hook', from=4-4, to=2-4]
	\arrow[hook, from=5-3, to=4-4]
	\arrow[hook', from=5-3, to=4-2]
	\arrow[from=5-3, to=3-3]
	\arrow["{t_{\beta}}"', tail reversed, no head, from=5-3, to=5-5]
	\arrow["{t_{\alpha}}", tail reversed, no head, from=5-3, to=5-1]
	\arrow["{t_0}"', tail reversed, no head, from=6-3, to=5-3]
	\arrow[from=5-5, to=4-5]
	\arrow[from=5-1, to=4-1]
	\arrow["\simeq"{description}, draw=none, from=4-1, to=4-2]
	\arrow["\simeq"{description}, draw=none, from=4-5, to=4-4]
	\arrow[Rightarrow, no head, from=5-1, to=6-3]
    \end{tikzcd}}\]
    \caption{Parametrizations of the maximal torus $T_0(F)$}
    \label{fig:parametrization}
\end{figure}%G2 coroot Lattice

\begin{remark}\label{choice of character}
    The maps $t_0^{-1}$ and $t_{\alpha}$ defined above are identical. They underscore the distinction that $t_0$ describes the maximal torus embedded in the Borel subgroup while $t_{\alpha}$ characterizes the torus embedded in the Levi factor $M_{0,\alpha}(F)$ through its parametrization in terms of $GL_2(F)$. The diagram in Figure \ref{fig:parametrization} helps to have a better understanding of these embeddings.
\end{remark}

The seven-dimensional representation of $\mathbf{G_2}$ attached to the short fundamental weight $\gamma_{s} = 2\alpha+\beta$ is recognized as the smallest representation of $\mathbf{G_2}$, wherein every representation of $\mathbf{G_2}$ can be found within its tensor algebra\cite{ful-Har}. This representation is called the standard representation of $\mathbf{G_2}$ and is denoted by $R_7$. We will follow \cite{mundy} to describe a matrix representation of the Levi factors under the standard representation $R_7$. Let $V_7$ be the representation space of $R_7$. The set of weight vectors of $V_7$ corresponding to weights of $R_7$ is $\{v_{-(\gamma_s)},\; v_{-(\alpha+\beta)},\; v_{-\alpha},\; v_{0},\; v_{\alpha},\; v_{\alpha+\beta},\; v_{\gamma_s}\}$, it forms an ordered basis of the space of $7\times 7$ matrices representing $\mathbf{G_2}$. 

Translating the weights of the standard representation by $w_{\alpha}$, determine the decomposition of the representation of $M_{0,\alpha}$. Considering the representative of the maximal torus introduced earlier we have:
\begin{equation}
    R_7 \circ t_{\alpha} = \operatorname{diag}(\tilde{\rho}_2, \operatorname{Ad}^2, \rho_2) \in SO(7,\mathbb{C}),
\end{equation}
the representations $\rho_2$ and $\tilde{\rho}_2$ are the standard representation of $GL_2(\mathbb{C})$  and its contragredient, respectively. Also $Ad^2$ is the adjoint square representation of $GL_2(\mathbb{C})$. Similarly, translations of weights by $w_{\beta}$ we have:
\begin{equation}
    R_7 \circ t_{\beta} = \operatorname{diag}(\operatorname{det}^{-1}, \tilde{\rho}_2, 1, \rho_2, \operatorname{det}) \in SO(7,\mathbb{C}).
\end{equation}

\

\subsection{Locally symmetric spaces}\label{sec:lss}
Let $G=Res_{F/\mathbb{Q}}(G_0)$ be the rational group of type $G_2$ obtained by Weil restriction. Since $G_{0}$ is a split group over a number field $F$, the group $G$ is a quasi-split group \cite{borel2012linear}, with $B = Res_{F/\mathbb{Q}}(B_{0})$ being the Borel subgroup of $G$. Similarly, we can define the restriction of scalars of all the subgroups of $G$ corresponding to those in $G_{0}$ and denote them by the same notation without the subscript $0$.

The adelic points of the group denoted by $G(\mathbb{A}) = G_0(\mathbb{A}_F)$, where its group at infinity $G(\mathbb{R})$ is a direct product of complex groups of type $\mathbf{G_2}$ over archimedean places. Fix a maximal compact subgroup of $G(\mathbb{R})$ in a regular way:
\[
K_{\infty}\simeq  \prod_{v\in S_c} \mathbf{G_2^{c}},
\]
where $\mathbf{G}^{c}_2$ is the compact form of a group of type $\mathbf{G_2}$, \cite[Chapter X.6]{Hel}. Note that $K_{\infty}$ is a connected group.

Let
\[
    S^{G} := S^{G}_{K_f} = G(\mathbb{Q})\backslash G(\mathbb{A})/ K_{\infty}K_f.
\]
be a locally symmetric space attached to $G$ and an open compact subgroup $K_f$ in $G(\mathbb{A}_f)$. We may assume $K_f$ to be neat, see \cite[Lemma 15.2.5]{GetzHahn2024}, ensuring that $S^{G}$ is a smooth manifold. When $K_f$ is clear from the context, we will omit it from the notation as above to simplify expressions.

The space $S^{G}$ is not compact as the $\mathbb{Q}$-rank of $G$ is greater than one \cite{BHC}. Let $\bar{S}^{G} \supset S^{G}$ be the Borel-Serre compactification of $S^{G}$, constructed by adding a boundary stratified as $G(\mathbb{Q})$-conjugacy classes of parabolic subgroups $P$ over the field of rational numbers, denoted by $\partial S^{G} = \cup_{P}\partial_{P} S^{G}$ \cite{BS}. Since we assumed $K_f$ to be neat, $\bar{S}^{G}$ is a compact finite-volume manifold with corners.

For a parabolic subgroup $P = M U$, we can define the locally symmetric spaces attached to its Levi quotient, $\pi_P: P \to M$,
\begin{align*}
    S^{M} &= M(\mathbb{Q})\backslash M(\mathbb{A})/ K_{\infty,M}K_{f,M},
\end{align*}
where $K_{\infty,M} = \pi_P(P\cap K_{\infty}) \subset M_{P}(\mathbb{R})$ is a maximal compact subgroup, and  $K_{f,M} \subset M_{P}(\mathbb{A}_f)$ is compact-open subset defined in \cite{HarderB} which is neat as it inherit the neatness from the $K_f$ through the quotient map \cite{SchwermerGC}. In maximal parabolic cases, since the Levi factor $M_{0, \theta}$ is isomorphic to $GL_2/F$, we can analyze the locally symmetric space of $M_{\theta}$ more explicitly using the one for $GL_2$. Consider the standard subgroups of $ GL_2/F $, the Borel subgroup $ B_{0,2} $ consisting of upper triangular matrices, the maximal torus $ T_{0,2} $ represented by the diagonal matrices subgroup, the center of $ GL_2/F $ denoted by $ Z_{0,2} $, and the maximal $\mathbb{Q}$-split subgroup of $ Z_{0,2} $ denoted by $ S_{0,2} $. The corresponding subgroups in the restriction of scalars of the Levi factor are denoted as follows:
\[M_{\theta} \supset B_{2,\theta}=T_{2,\theta}.U_{2,\theta} \supset Z_{2,\theta} \supset S_{2,\theta}.\]

 The group at infinity is: 
\[M_{\theta}(\mathbb{R})\simeq  \prod_{v\in S_c} GL_2(\mathbb{C}),\]
with the center $Z_{2,\theta} \simeq \prod_{v\in S_r} \mathbb{R}^{\times}.I_2 \times \prod_{v\in S_c} \mathbb{C}^{\times}.I_2$ where $I_2 = diag(1,1)$, and the maximal torus $S_{2,\theta}(\mathbb{R})$ sits diagonally in $Z_{2,\theta}(\mathbb{R})$. The maximal compact subgroup  of $M_{\theta}(\mathbb{R})$ is:
\[C_{2,\theta,\infty} \simeq \prod_{v\in S_c} U(2)\]

Moreover, define:
\[K_{2,\theta,\infty} = C_{2,\theta,\infty}.S_{2,\theta}(\mathbb{R})\]
and $K^{\circ}_{2,\theta,\infty} = C_{2,\theta,\infty}.S_{2,\theta}(\mathbb{R})^{\circ}$ is its connected component. The maximal compact subgroup $K_{2,\theta,\infty}$ is connected. 

To study the cohomology group discussed in this text, we need the dimensions of the spaces introduced above. We will list them here for future reference.  The dimension of the group $G(\mathbb{R})$ is:
\[
    \dim_{\mathbb{R}}(G(\mathbb{R})) = \mathsf{r}dim_{\mathbb{R}}(\mathbf{G_2}(\mathbb{C})) = 28\mathsf{r} = 14d_F.
\]
For $\theta\in \Delta$, the dimension of the unipotent subgroup is:
\[
    \dim_{\mathbb{R}}(U_{\theta}(\mathbb{R})) = d_F. dim_{\mathbb{R}}(U_{0,\theta}(\mathbb{R})) = 5d_F.
\]

Moreover, we have:
\begin{align*}
    \dim_{\mathbb{R}}(G(\mathbb{R})) &= 28\mathsf{r},\\
    \dim_{\mathbb{R}}(K_{\infty}) &= 14 \mathsf{r},  \\
    \dim_{\mathbb{R}}(S^{G}) &= \dim_{\mathbb{R}}(G)-\dim_{\mathbb{R}}(K_{\infty})= 14 \mathsf{r},  \\
    \dim_{\mathbb{R}}(\partial S^{G}) &= 14\mathsf{r}-1.
\end{align*}
For $\theta \in \Delta$, we can calculate the dimensions associated with the Levi factor $M_{\theta}$, as it is isomorphic to the group $Res_{F/\mathbb{Q}}(GL_2/F)$:
\begin{align*}
    \dim_{\mathbb{R}}(S^{M_{\theta}}) &=4\mathsf{r}-1.
\end{align*}

\subsection{Characters of the torus and its parametrizations}\label{sec:char}
Let $E/\mathbb{Q}$ denote a `large enough' Galois extension containing a copy of $F$, as outlined in \cite{RagTI}. Let $\iota:E \to \mathbb{C}$ be an embedding, define the map $\iota_{*}:\operatorname{Hom}(F,E) \to \operatorname{Hom}(F,\mathbb{C})$ by $\iota_{*}(\tau) = \iota \circ \tau =:\tau^{\iota}$, and denote the complex conjugate $\overline{\iota_{*}(\tau})$ by $\bar{\tau}^{\iota}$.

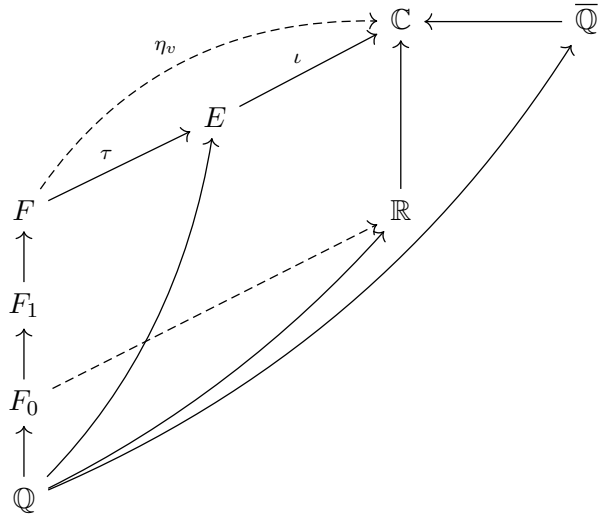
\begin{figure}[htp]
\centering 
\[\begin{tikzcd}
	&&&& {\mathbb{C}} && {\bar{\mathbb{Q}}} \\
	&& E \\
	F &&&& {\mathbb{R}} \\
	{F_1} \\
	{F_0} \\
	{\mathbb{Q}}
	\arrow[from=6-1, to=5-1]
	\arrow[from=5-1, to=4-1]
	\arrow[from=4-1, to=3-1]
	\arrow[curve={height=18pt}, from=6-1, to=2-3]
	\arrow[curve={height=12pt}, from=6-1, to=3-5]
	\arrow[from=3-5, to=1-5]
	\arrow["\iota", from=2-3, to=1-5]
	\arrow["\tau", from=3-1, to=2-3]
	\arrow[dashed, from=5-1, to=3-5]
	\arrow[curve={height=30pt}, from=6-1, to=1-7]
	\arrow[from=1-7, to=1-5]
	\arrow["{\eta_v}", curve={height=-30pt}, dashed, from=3-1, to=1-5]
    \end{tikzcd}\]
    \caption{Embedding of Fields in Totally Imaginary Case}
    \label{fig:fieldE}
\end{figure}

Now we can define the algebraic character module,
\[
    X^{*}(T\times E) := \operatorname{Hom}(T\times_{\mathbb{Q}} E, \mathbb{G}_m).
\]
There is a natural action of $\operatorname{Gal}(E/\mathbb{Q})$ on $X^{*}(T\times E)$. Moreover, since $T = Res_{F/\mathbb{Q}}(T_0)$ and $T_0$ is split over $F$, we have
\[
    X^{*}(T\times E) = \bigoplus_{\tau:F\to E}X^{*}(T_0\times_{F,\tau} E) = \bigoplus_{\tau:F\to E}X^{*}(T_0).
\]

Any $\lambda \in X^{*}_{\mathbb{Q}}(T\times E)= X^{*}(T\times E)\otimes \mathbb{Q}$, can be decomposed  to $\lambda = (\lambda^{\tau})_{\tau: F\to E}$ where $\lambda^{\tau}$ is a rational character of $T_0$. Moreover, the action of $\eta \in Gal(E/\mathbb{Q})$ on $\lambda$ can be seen as $^{\eta}\lambda =(^{\eta}\lambda^{\tau})_{\tau:F\to E} = (\lambda^{\eta^{-1}\circ \tau})_{\tau:F\to E}$. Then by (\ref{eq:char}) and the fact that $T_0$ is split over $F$, for any $\tau: F\to E$ we have:
\begin{equation}
    \lambda^{\tau}= m^{\tau} \gamma_l+ n^{\tau} \gamma_s = a^{\tau}(2\alpha+\beta) + b^{\tau}(\alpha+\beta) =: (a^{\tau},b^{\tau})_{0}.
\end{equation}    
The second equality comes from the parametrization $t_0: T_0 \to F^{*}\times F^{*}$. The relation between the two different representatives above can be described as
\begin{align}
    m^{\tau} &= b^{\tau}, \quad n^{\tau} = a^{\tau} - b^{\tau},\\
    a^{\tau} &= m^{\tau}+n^{\tau}, \quad b^{\tau} =m^{\tau}.
\end{align}

For any $\tau: F \to E$, the weight $\lambda^{\tau}$ is called an integral weight if
\[
a^{\tau}, b^{\tau} \in \mathbb{Z} \quad \iff \quad m^{\tau}, n^{\tau} \in \mathbb{Z}.
\]
The weight $\lambda = (\lambda^{\tau})_{\tau:F\to E}$ is an integral weight if $\lambda^{\tau}$ is an integral weight for all $\tau:F\to E$. 

For any $\tau : F\to E$, an integral weight $\lambda^{\tau}$ is called a dominant integral weight if 
\[
    \left\{
     \begin{array}{l}
        m^{\tau}\geq 0,\\
        n^{\tau}\geq 0;
     \end{array}
   \right. \quad \iff  \quad a^{\tau} \geq b^{\tau} \geq 0.
\]
The weight $\lambda \in X^{*}(T\times E)$ is a dominant integral weight if $\lambda^{\tau}$ is a dominant integral weight for all the embeddings $\tau: F\to E$. We will denote the set of all dominant integral weights by $X^{+} (T \times E)$.

As we mentioned, the description of the weight of the torus depends on the different parametrizations of $T_0$, in Figure \ref{fig:parametrization}. For the case $t_{\alpha}$, we have the same description as above and we will use the notation $(a^{\tau},b^{\tau})_{\alpha} := (a^{\tau},b^{\tau})_{0}$. But for the parametrization $t_{\beta}$ the weights $\lambda^{\tau}$ presents as
\begin{equation}
    \lambda^{\tau} = m^{\tau} \gamma_l + n^{\tau} \gamma_s = c^{\tau} (\alpha+\beta) + d^{\tau}\alpha =: (c^{\tau}, d^{\tau})_{\beta}.
\end{equation}
The relation between the different parametrizations above can be described as follows:
\begin{align}
    m^{\tau} &= c^{\tau}-d^{\tau}, \quad n^{\tau} = 2d^{\tau} - c^{\tau},\\
    c^{\tau} &= 2m^{\tau}+n^{\tau}, \quad d^{\tau} = m^{\tau}+n^{\tau}.
\end{align}
Therefore, $\lambda \in X^{+}(T\times E)$ if and only if for all $\tau : F\to E$, it holds that $c^{\tau}$ and $d^{\tau}$ are integers and $2d^{\tau} \geq c^{\tau} \geq d^{\tau}\geq 0$. When the parametrization of $T_0$ is evident from the context, we will simply use the notation $(a^{\tau},b^{\tau})$ to represent a weight of $T_{0}.$

\section{Cohomology of the arithmetic groups in $G_2$}
\subsection{Cohomology of arithmetic groups as a sheaf cohomology}\label{sec:fdr}
For a dominant integral weight $\lambda \in X^{+}(T\times E)$, define an absolute irreducible finite-dimensional representation $(\rho_{\lambda}, \mathcal{M}_{\lambda})$ of the group $G\times_{\mathbb{Q}} E$ with the highest weight $\lambda$. Keep in mind that $\mathcal{M}_\lambda$ carries the information about $E$, as $\lambda$ is defined over $E$. To highlight the dependence on the field $E$, we may denote it as $\mathcal{M}{\lambda,E}$. Here $\mathcal{M}_{\lambda} = \bigotimes_{\tau}\mathcal{M}_{\lambda^{\tau}}$ where $\mathcal{M}_{\lambda^{\tau}}/E$ is an irreducible finite-dimensional representation of $G_0\times_{\tau} E$ with the highest weight $\lambda ^{\tau}$. There is an action of $Gal(E/\mathbb{Q})$ on the highest weight module $\mathcal{M}_{\lambda}$ induced by its action on the highest weight. Now define the rationality field
\[ E(\lambda) := E^{\{\gamma \in  Gal(E/\mathbb{Q}) \; : \; ^{\gamma}\lambda = \lambda \}} \subset E,\]
where following \cite{HR}, it is generated by the values of $\lambda$ on $T(\mathbb{Q})$ suggesting the existence of the canonical rational field of definition $\mathbb{Q}(\lambda)$. Moreover, for any $\sigma \in Gal(E/\mathbb{Q})$ there is a natural linear map $\Phi_{\sigma}: \mathcal{M}_{\lambda} \to \;^{\sigma}\mathcal{M}_{\lambda} = \mathcal{M}_{^{\sigma}\lambda}$ which is an isomorphism. 

The group $G(\mathbb{Q}) = G_0(F)$ acts diagonally on the module $\mathcal{M}_\lambda$. Naturally, we can define a sheaf of $E$-vector spaces on $S^G$ associated with the $G(\mathbb{Q})$-module $\mathcal{M}_\lambda$, as described in \cite[Section 15.3]{GetzHahn2024}. For simplicity, we will also denote this sheaf by $\mathcal{M}_\lambda$, as the context makes it clear when it refers to the representation, for example, when it appears in relative Lie algebra cohomology. Furthermore, since $K_f$ is neat, the sheaf $\mathcal{M}_\lambda$ is locally constant and forms a local system.

% \begin{remark}
%    The vanishing property of the sheaf $\mathcal{M}_{\lambda}$ is explained in \cite[Section 2.2.8]{HR}, where an ``algebraicity'' condition on the highest weight ensures that the action of the center of the group on the highest weight module, via the central character of $\rho_{\lambda}$, is trivial. In our case, since the center of the group is trivial, the algebraicity condition is vacuously satisfied for any dominant weight. Therefore, we have the non-vanishing of sheaves in our context.
% \end{remark}

 We can extend the sheaf $\mathcal{M}_{\lambda}$ to a sheaf on $\bar{S}^{G}$, denoted by the same notation. The homotopy equivalence between $S^{G}$ and $\bar{S}^{G}$ implies that the canonical restriction map is an isomorphism between the sheaf cohomology groups $H^{\bullet}(\bar{S}^{G},\mathcal{M}_{\lambda})$ and $H^{\bullet}(S^{G},\mathcal{M}_{\lambda})$. To the manifold with corner $\bar{S}^{G}$ with the boundary $\partial S^{G}$ we can naturally assign the long exact sequence:
\begin{equation}\label{les}
    \cdots \to H^{\bullet}_{c}(S^{G},\mathcal{M}_{\lambda}) \xrightarrow{\mathfrak{i}^{\bullet}} H^{\bullet}(S^{G},\mathcal{M}_{\lambda}) \xrightarrow{\mathfrak{r}^{\bullet}} H^{\bullet}(\partial S^{G},\mathcal{M}_{\lambda}) \to \cdots,
\end{equation}
where, in the middle term, we replaced $\bar{S}^{G}$ with $S^{G}$ following the discussion above. The cohomology with compact support and the cohomology of the boundary denoted by $H^{\bullet}_{c}(S^{G},\mathcal{M}_{\lambda})$ and $H^{\bullet}(\partial S^{G},\mathcal{M}_{\lambda})$, respectively. Let $\mathcal{H}^{G}_{K_f} = \mathcal{C}^{\infty}_{c}(G(\mathbb{A}_f)/\!\!/ K_f)$ be the Hecke module with the measure normalized in the natural way. Therefore, the cohomology groups in the fundamental exact sequence (\ref{les}) are Hecke modules \cite[Section 2.3.2]{HR}.

Following \cite[Section 2.3.3]{HR}, let $E'$ be a larger Galois extension of $\mathbb{Q}$ containing $E$ and  $\iota_{1}: E \to E'$ be an injection. It induces an isomorphism on character groups $X^{*}(T\times E)$ to $X^{*}(T\times E)$ mapping $\lambda$ to $^{\iota_1}\lambda$. Moreover, any embedding $\tau : F \to E$ is in one-to-one correspondence with $\iota_1 \circ \tau: F \to E^{'}$, encompassing all embeddings from $F$ to $E'$. So, we have the unique form for any $^{\iota_1}\!\lambda \in X^{*}(T \times E^{'})$ which may described as:
\[
    ^{\iota_1}\!\lambda = (\lambda^{\iota_1\circ \tau})_{\tau : F\to E} = (\lambda^{\tau_1})_{\tau_1 : F\to E^{'}}, \qquad \tau_1 = \iota_1 \circ \tau.
\]
Again, we can define the highest module $\mathcal{M}_{{^{\iota_1}\lambda},E^{'}}$ with the highest weight $^{\iota_1}\!\lambda$ and attach the sheaf $\mathcal{M}_{{^{\iota_1}\!\lambda},E^{'}}$ as before. The identification between the character groups after the base change described above will induce a Hecke equivariant isomorphism between the sheaves $\mathcal{M}_{{^{\iota_1}\!\lambda},E^{'}} \simeq \mathcal{M}_{\lambda} \otimes_{E,\iota_{1}} E^{'}$. Finally, we can form the isomorphism between the corresponding cohomology groups after the base change:
\[
    \iota_1^{\bullet}: H^{\bullet}(S^{G},\mathcal{M}_{{^{\iota_1}\!\lambda},E^{'}}) \simeq H^{\bullet}(S^{G},\mathcal{M}_{\lambda} \otimes_{E,\iota_{1}} E^{'}).
\]
We can show that any cohomology group in the fundamental exact sequence (\ref{les}) behaves well under the field extension as described above. 

\begin{remark}
    Following \cite{HR}, we refer to the transcendental level when we consider the embedding $\iota: E \to \mathbb{C}$ and extend the base field to $\mathbb{C}$. We refer to the arithmetic level when the base field is a finite Galois extension $E/\mathbb{Q}$, as mentioned earlier in this section.
\end{remark}

We are interested in studying the image of the maps $\mathfrak{i}^{\bullet}$ and $\mathfrak{r}^{\bullet}$ in (\ref{les}). The image of the map $\mathfrak{i}^{\bullet}$ is called the inner cohomology and denoted by:
\[
    H_{!}^{\bullet}(S^{G},\mathcal{M}_{\lambda}) := \operatorname{image}(H^{\bullet}_{c}(S^{G},\mathcal{M}_{\lambda}) \xrightarrow{\mathfrak{i}^{\bullet}} H^{\bullet}(S^{G},\mathcal{M}_{\lambda}) ).
\]
And the image of the map $\mathfrak{r}^{\bullet}$ is called the Eisenstein cohomology and denoted by:
\[
    H_{Eis}^{\bullet}(\partial S^{G},\mathcal{M}_{\lambda}) := \operatorname{image}(H^{\bullet}(S^{G},\mathcal{M}_{\lambda}) \xrightarrow{\mathfrak{r}^{\bullet}} H^{\bullet}(\partial S^{G},\mathcal{M}_{\lambda}) ).
\]
We can define the cohomology group of the Levi factors as a Hecke module in the same way. In our case, Levi factors are isomorphic with Weil restriction of scalars of a group $GL_2$ over a number field $F$ where the structure of the cohomology of arithmetic groups assigned to $GL_2$ are well-understood by \cite{H2, RagTI}.

\subsection{Cohomology of boundary components} \label{boundcomp}
There is a finite union of boundaries over the  $G(\mathbb{Q})$ conjugacy classes of proper parabolic subgroups as follows:
\[
    \partial \bar{S}^G = \bigcup_P \partial_P \bar{S}^G,
\]
where the boundary component for each $P$ is
\[
    \partial_{P} S^{G} = P(\mathbb{Q})\backslash G(\mathbb{A})/ K_{\infty}K_{f}.
\]
Therefore, we can define the cohomology of a boundary component as
\[
    H^{\bullet}(\partial_{P} \bar{S}^G, \mathcal{M}_{\lambda}) = H^{\bullet}(P(\mathbb{Q}\backslash G(\mathbb{A})/K_{\infty}K_f, \mathcal{M}_{\lambda}).
\]
For $\xi$ running over the representatives in the of finite space $P(\mathbb{A}_f)\backslash G(\mathbb{A}_f)/K_{f}$, let us define $K_{f,U}(\xi) = U(\mathbb{A}_f)\cap \xi K_f \xi^{-1}$ then we have the fibration:
\[
     S^{P}_{K_{f,P}(\xi)} \to S^{M}_{K_{f,M}(\xi)}.
\]
Consider the quotient space $\Gamma_U\backslash U(\mathbb{R})$, where $\Gamma_{U}\subset U(\mathbb{Z})$ is a subgroup of finite index, depending on $K_{f,U}(\xi)$ as indicated in \cite{HarderB}. Following the work of Schwermer \cite{SchK}, for a regular weight $\lambda$, due to the $E_2$-degeneracy of the Leray-Serre spectral sequence corresponding to this fibration, we obtain:
\[
    H^{p}(\partial_{P}S^{G},\mathcal{M}_{\lambda}) = \oplus_{\xi} H^{p-q}(S^{M}_{K_{f,M}(\xi)},H^{q}(\Gamma_U\backslash U(\mathbb{R}),\mathcal{M}_{\lambda})).
\]
By a theorem of van-Est \cite{vanEst}, if $\mathfrak{u}$ is the Lie algebra of $U$ then:
\[
H^{\bullet}(\Gamma_U\backslash U(\mathbb{R}),\mathcal{M}_{\lambda}) \xrightarrow{\sim} H^{\bullet}(\mathfrak{u}, \mathcal{M}_{\lambda}).
\] 
The unipotent Lie algebra cohomology $H^{\bullet}(\mathfrak{u}, \mathcal{M}_{\lambda})$ is an $M(\mathbb{R})$-module; so we can attach a sheaf to it denotes by a same notation as before, and we have
\[
    H^{\bullet}(\partial_P S^{G}, \mathcal{M}_{\lambda}) = \bigoplus_{\xi}H^{\bullet-q}(S_{K_{f,M}(\xi)}^{M},H^{q}(\mathfrak{u},\mathcal{M}_{\lambda})).
\]

The direct limit of the boundary cohomology attached to $P$  and $M$, over open compact subgroups $K_f$  form the cohomology groups denoted by $H^{\bullet}(\partial_{P} \underline{S}^{G}, \mathcal{M}_{\lambda})$ and $H^{\bullet}(\partial \underline{S}^{M}, \mathcal{M}_{\lambda})$, where the underline indicates that it represents the direct limit of the locally symmetric spaces with the same notation. We can rewrite the equation above as
\[
    H^{\bullet}(\partial_P \underline{S}^{G}, \mathcal{M}_{\lambda})^{K_f} = \bigoplus_{\xi}H^{\bullet-q}(\underline{S}^{M},H^{q}(\mathfrak{u},\mathcal{M}_{\lambda}))^{K_{f,M}(\xi)}.
\]

Let $^{a}\!\operatorname{Ind}$ denote the algebraic (or un-normalized) induction in an ordinary group theoretic way. By Mackey theory and following \cite[Prop 4.3.]{HR}:
    \begin{equation}\label{indMG}
        H^{\bullet}(\partial_P \underline{S}^G, \mathcal{M}_{\lambda}) =\; ^{a}\!\operatorname{Ind}_{P(\mathbb{A}_f)}^{G(\mathbb{A}_f)}H^{\bullet-q}(\underline{S}^{M},H^{q}(\mathfrak{u},\mathcal{M}_{\lambda})).
    \end{equation}

 Note that in our case $G(\mathbb{R})$, $P(\mathbb{R})$, and $K_{\infty}^{M}$ are all connected, so $\pi_0(G(\mathbb{R}))$ is trivial. Therefore, the study of the cohomology of the boundary components reduces to the study of a smaller cohomology group corresponding to the Levi factor. In the next section, we will review the work of Kostant \cite{kos} which gives us an explicit description of the local system $H^{q}(\mathfrak{u},\mathcal{M}_{\lambda})$ in terms of the highest weight representations of the Levi factor.

\subsection{Kostant theorem}\label{sec:kos}
Let $\tau: F \to E$ be as described in \ref{sec:fdr}. Define $\Phi^{\tau}$, $(\Phi^{+})^{\tau}$ and $\Delta^{\tau}$ to be the set of roots, positive roots, and simple roots assigned to the image of the group over $F$ under the embedding $\tau$. Accordingly, we can define $W_{0}^{\tau}$ to be the group generated by reflections of the roots in $\Delta^{\tau}$ and define the Weyl group of $G $ to be $W = \prod_{\tau: F \to E} W_{0}^{\tau}$. Similarly, we can define the Weyl group $W_{M}$ of the Levi quotient $M$ generated by its simple reflections assigned to the simple roots $\Delta_{M}$. The Kostant representatives in the Weyl group $W$ corresponding to the parabolic subgroup $P$ is the canonical minimal system of representatives of the quotient $W_{M}\backslash W$, defined as:
\begin{align}\label{eq:kosrep}
    W^{P} &= \{ w= (w^{\tau}) \in W \;| \; w^{\tau} \in (W_{0}^{\tau})^{P_{0}^{\tau}} \},\\
    (W_{0}^{\tau})^{P_{0}^{\tau}} &:= \{ w^{\tau}\in W_{0}^{\tau} \; | \; (w^{\tau})^{-1}\alpha ^{\tau}>0 ,\;\forall \alpha^{\tau}\in \Delta_{M_{0}^{\tau}}\}.
\end{align}

Let $\mu \in X^{+}(T_{M} \times E)$ be a dominant integral weight of the Levi quotient $M$. Following \cite{HR}, by abuse of notation, we write $\mathcal{M}_{\mu}$ for an irreducible representation of the Levi quotient $M$  with the highest weight $\mu$. It will be clear from the context when $\mathcal{M}_{\Box}$ is an irreducible representation of $G$ or $M$. Moreover, for any dominant weight $\lambda \in X^{+}(T\times E)$ and $w\in W^{P}$, we can define the twisted action of $w$ on the weight $\lambda$ by 
\begin{equation}\label{eq:twact}
    w.\lambda = (w^{\tau}(\lambda^{\tau} + \rho_{G})-\rho_{G})_{\tau:F\to E}\in X^{+}(T_M \times E),
\end{equation}
where $\rho_{G}$ be the half-sum of the positive roots of $G$ as defined in (\ref{eq:halfsum}). The weights $w.\lambda$ defined above are all dominant integral weights of the Levi quotient $M$. This dominant weight induces the irreducible finite-dimensional representation $\mathcal{M}_{w.\lambda}$ of highest weight $w.\lambda$. Kostant's theorem \cite{kos} gives the following decomposition:
\begin{equation}\label{eq:kos}
    H^{*}(\mathfrak{u},\mathcal{M}_{\lambda}) \simeq \bigoplus_{w\in W^P} \mathcal{M}_{w.\lambda}.
\end{equation}
For a given dominant integral weight $\lambda$ for $G$ and Kostant representatives $w$ ranging in $W^{P}$, all the $w.\lambda$ are distinct weights for $M$. The module $\mathcal{M}_{w.\lambda}$ is a irreducible representation of $M$ with highest weight $w.\lambda$. Therefore, the decomposition of $H^{q}(\mathfrak{u},\mathcal{M}_{\lambda})$  into $M(E)$-modules as shown in (\eqref{eq:kos}) is multiplicity-free. Moreover, for a given degree $q$, only the Kostant representatives $w$ with length $l(w)=q$ will contribute to the decomposition of $H^{q}(\mathfrak{u},E)$ in (\ref{eq:kos}).

Now, combining Kostant's theorem with (\ref{indMG}) yields the following result 
\begin{align}
    H^{q}(\partial_P \underline{S}^{G}, \mathcal{M}_{\lambda}) =
    \bigoplus_{w\in W^P} {}^{a}\!\operatorname{Ind}_{P(\mathbb{A}_f)}^{G(\mathbb{A}_f)}(H^{q-l(w)}(\underline{S}^{M},\mathcal{M}_{w.\lambda})).
\end{align}
 
For simplicity in the rest of this section, we will fix an embedding $\tau : F \to E$ and omit the $\tau$ from our notations. To determine the set of Kostant representatives for each standard parabolic subgroup, we need to find the action of $w^{-1}$ on the simple roots $\alpha$ and $\beta$ for all $w \in W$. These calculations are summarized in the following table:

\vspace{\baselineskip}

\begin{table}[ht]
    \caption{Action of the Inverse Weyl Group on the Simple Roots}
    \vspace*{6pt}
    \centering
    \begin{tabular}{||c c c c||}      
    \hline
     $l(w)$ & $w$ &  $w^{-1}\alpha$ & $w^{-1}\beta$  \\ [1 ex] 
     \hline\hline
     $0$ & $1$ & $\alpha$ & $\beta$ \\ [1 ex] 
     \hline\hline
     $1 $& $w_{\beta}$& $\alpha+\beta$ & $-\beta$\\[1 ex] 
     \hline
     $2$ &$ w_{\beta\alpha}$& $\gamma_s$ & $-(3\alpha+\beta)$\\ [1 ex] 
     \hline
     $3$ & $w_{\beta\alpha\beta}$ & $\gamma_{s} $& $-\gamma_{l}$\\[1 ex] 
     \hline
     $4 $&  $w_{\beta\alpha\beta\alpha}$ & $\alpha+\beta$ & $-\gamma_l$\\ [1 ex] 
     \hline
     $5$ &$ w_{\beta\alpha\beta\alpha\beta}$ &$ \alpha $& $-(3\alpha+\beta)$\\[1 ex] 
     \hline\hline
     $1$ & $w_{\alpha}$ & $-\alpha $& $3\alpha+\beta$\\[1 ex] 
     \hline
     $2$ &$ w_{\alpha\beta}$ & $-(\alpha+\beta)$ & $\gamma_l$\\[1 ex] 
     \hline
     $3$ &$ w_{\alpha\beta\alpha}$ & $-\gamma_s$ &  $\gamma_{l}$\\ [1 ex] 
     \hline
     $4$ & $w_{\alpha\beta\alpha\beta}$  & $-\gamma_s$ & $3\alpha+\beta$\\ [1 ex] 
     \hline
     $5$ & $w_{\alpha\beta\alpha\beta\alpha}$ & $-(\alpha+\beta)$ & $\beta$\\ [1 ex] 
     \hline\hline
     $6$ &$ w_G$  & $-\alpha$ & $-\beta$\\[1ex] 
     \hline
    \end{tabular}
\end{table}

\vspace{\baselineskip}

Since $\Delta_{B}=\emptyset$, $\Delta_{M_{\alpha}} = \{\alpha\}$, and  $\Delta_{M_{\beta}} = \{\beta\}$,  using Equation (\ref{eq:kosrep}), we get:
\begin{align*}
    W^{B} & = W,\\
    W^{P_{\alpha}}&= \{1, w_{\beta}, w_{\beta\alpha}, w_{\beta\alpha\beta},w_{\beta\alpha\beta\alpha}, w_{\beta\alpha\beta\alpha\beta}\},\\
    W^{P_{\beta}}&= \{1, w_{\alpha}, w_{\alpha\beta} ,w_{\alpha\beta\alpha},w_{\alpha\beta\alpha\beta}, w_{\alpha\beta\alpha\beta\alpha}\}.
\end{align*}

\subsection{Inner cohomology}\label{sec:inner}
The image of $\mathfrak{i}^{\bullet}$, in the long exact sequence (\ref{les}), is called inner cohomology and denoted by:
\[
    H_{!}^{\bullet}(S^{G},\mathcal{M}_{\lambda}) := \operatorname{image}(H^{\bullet}_{c}(S^{G},\mathcal{M}_{\lambda}) \xrightarrow{\mathfrak{i}^{\bullet}} H^{\bullet}(S^{G},\mathcal{M}_{\lambda}) ).
\]
Following \cite[Section 2.3.5]{HR}, as we assumed $E$ is large enough, the inner cohomology is a semisimple module under the Hecke algebra and there is an isotypical decomposition,
\begin{equation}\label{eq;isoinner}
    H_{!}^{\bullet}(S^{G},\mathcal{M}_{\lambda}) = \bigoplus_{\pi_{f}\in Coh_{!}(G,K_f,\lambda)}H_{!}^{\bullet}(S^{G},\mathcal{M}_{\lambda})(\pi_f),
\end{equation}
where $Coh_{!}(G,K_f,\lambda)$ is a finite set of the isomorphism types of Hecke-module that occur in the inner cohomology. We also form the inner spectrum of $G$ with coefficient $\lambda$ as the union over all the compact open subgroups $K_f$ as:
\[
    Coh_{!}(G,\lambda) := \cup_{K_f} Coh_{!}(G,K_f,\lambda).
\]

The field embedding $\iota: E \to E'$, induces a map of the cohomology groups that preserves the inner cohomology groups:
\[
    \iota^{\bullet} :  H_{!}^{\bullet}(S^{G},\mathcal{M}_{\lambda}) \to H_{!}^{\bullet}(S^{G},\mathcal{M}_{^{\iota}\lambda, E'}),
\]
it maps an isotypic component of $ H_{!}^{\bullet}(S^{G},\mathcal{M}_{\lambda})$ attached to $\pi_f \in Coh_{!}(G,K_f,\lambda)$ onto the $^{\iota}\pi_f$-isotypic component of $H_{!}^{\bullet}(S^{G},\mathcal{M}_{^{\iota}\lambda, E'})$ where $^{\iota}\pi_f \in Coh_{!}(G,K_f,\lambda)$. We can split off the $\pi_f$-isotypic component of the inner cohomology for the large enough field extension $E$, then following \cite{HR}, define the rationality field of $\pi_f$ as:
\[
    E(\pi_f) := E^{\{\gamma\in Gal(E/G) : ^{\gamma}\pi_f = \pi_f\}}.
\]

We can now focus on studying the cohomology of the Levi factor, isomorphic to $GL_2$,  which is well understood from \cite{H2}, \cite{HarderB} and \cite{RagTI}. Following a base change to $\mathbb{C}$, the inner cohomology is encapsulated by the cohomology of the discrete spectrum. Let $S$ be a finite set of places as usual, by the strong multiplicity one theorem for the discrete spectrum of $GL_n$ \cite{JacSMO,MogWal}, the $\sigma_f \in Coh_{!}(GL_2,\mu)$ is determined by its restriction to the central subalgebra $\mathcal{H}^{G, S}$ of $\mathcal{H}^{G}$. Moreover, it implies that the field $E(\sigma_f)$ is the subfield of $E$ generated by the values of $\sigma_p$ for $p \in S$.

\subsection{Cohomology at transcendental level} \label{trans}
Fix $\iota: E \to \mathbb{C}$, it induces a bijection $X^{+}(T\times_{\tau,F} E) \to X^{+}(T\times_{\iota,E}\mathbb{C})$ by mapping $\mu$ to $^{\iota}\mu$. For any dominant weight $\mu \in X^{+}(T\times_{\tau,F} E)$, there exists an absolutely irreducible representation $\mathcal{M}_{\mu,E}$ of $M\times_{\tau,F} E$ with the highest weight $\mu$. We can then extend this irreducible representation to an irreducible representation $\mathcal{M}_{{^{\iota}\mu},\mathbb{C}}:=\mathcal{M}_{{^{\iota}\mu}}\otimes \mathbb{C}$ of $M\times_{\iota,E} \mathbb{C}$ with the highest weight $^{\iota}\mu$. Moreover, we have the local system of $\mathbb{C}$-vector spaces $\mathcal{M}_{{^{\iota}\mu},\mathbb{C}}$ following the discussion in \ref{sec:fdr}.

\begin{remark}
     We can define the dual form $\mathcal{M}_{{^{\iota}\mu},\mathbb{C}}^{\vee} := Hom_{\mathbb{C}}(\mathcal{M}_{{^{\iota}\mu},\mathbb{C}}, \mathbb{C})$ and the complex conjugate $\overline{\mathcal{M}}_{{^{\iota}\mu},\mathbb{C}} := \mathcal{M}_{^{c}\mu,\mathbb{C}} = \mathcal{M}_{^{c}\mu} \otimes_{\iota,E} \mathbb{C}$, as discussed in \ref{sec:fdr}, where $c \in Aut(\mathbb{C}/\mathbb{R})$ is the complex conjugate. Both of these modules form rational representations with the highest weights \cite[Section 8.1.1.]{HarderB}. Denote the longest Weyl element of $W_M$ as $w_0$, the highest weight of $\mathcal{M}_{{^{\iota}\mu},\mathbb{C}}^{\vee}$ is $-w_0(\mu)$. On the other hand, the highest weight of $\overline{\mathcal{M}}_{{^{\iota}\mu},\mathbb{C}} := \mathcal{M}_{^{c}\mu,\mathbb{C}}$ by definition is $^{c}\mu$. The module $\mathcal{M}_{{^{\iota}\mu},\mathbb{C}}$ is conjugate-auto dual if $^{c}\mu =  -w_0(\mu)$. In the case where the rational group is split, the complex conjugate action is trivial. Thus, being conjugate-auto dual for a split group implies that $\mu = -w_0(\mu)$.
\end{remark}

Following our earlier discussion about arithmetic level, the sheaf $\mathcal{M}_{{^{\iota}\mu},\mathbb{C}}$ attached to the finite-dimensional representation $(\rho_{{^{\iota}\mu},\mathbb{C}},\mathcal{M}_{{^{\iota}\mu},\mathbb{C}})$ of the group $M(\mathbb{R})$ on the complex vector space $\mathcal{M}_{{^{\iota}\mu},\mathbb{C}}$ forms a local system. Let $\Omega^{q}(S^{M},\mathcal{M}_{{^{\iota}\mu},\mathbb{C}})$ be the space of $\mathcal{M}_{{^{\iota}\mu},\mathbb{C}}$-valued differential forms on the space $S^{M}$ of degree $q\geq 0$. The cohomology group $H^{*}(S^{M},\mathcal{M}_{{^{\iota}\mu},\mathbb{C}})$ of the manifold $S^{M}$ with coefficients in the local system $\mathcal{M}_{{^{\iota}\mu},\mathbb{C}}$ is the de-Rham cohomology of the complex $\Omega^{*}(S^{M},\mathcal{M}_{{^{\iota}\mu},\mathbb{C}})$, \cite{BoWal}.

On the other hand, we can define the relative Lie algebra cohomology for the group of real points $M(\mathbb{R})$ with respect to its maximal compact subgroup $K_{\infty}$. We can identify the cohomology groups:
\begin{equation}\label{eq:LieCoh}
    H^{\bullet}(S^{M}, \mathcal{M}_{{^{\iota}\mu},\mathbb{C}}) \simeq H^{\bullet}(\mathfrak{g},K_{\infty}; \mathcal{C}^{\infty}(M(\mathbb{Q})\backslash M(\mathbb{A})/K_f) \times \mathcal{M}_{{^{\iota}\mu},\mathbb{C}}).
\end{equation}

 The quotient $M(\mathbb{Q}) S(\mathbb{R})^{0} \backslash M(\mathbb{A})/ K_f$ has a finite volume \cite{BHC}. Let $\omega_{\infty}$ be the restriction of the central character of $\mathcal{M}_{\mu}$ to $S(\mathbb{R})^{0}$. Define the space of functions $\mathcal{C}^{\infty}(M(\mathbb{Q})\backslash M(\mathbb{A})/K_f, \omega_{\infty}^{-1})$ as the set of all smooth functions $\phi: M(\mathbb{A})\to \mathbb{C}$ such that:
\[
    \phi(\gamma.g.k_{f}.a_{\infty}) = \omega_{\infty}^{-1}(a_{\infty})\phi(g), \forall g\in M(\mathbb{A}), \gamma \in M(\mathbb{Q}), k_f\in K_f, a_{\infty}\in S(\mathbb{R})^{0}.
\]
Moreover, we can define the space of square-integrable and smooth cuspidal functions. Following \cite[Theorem 8.1.1]{HarderB}, define the cuspidal cohomology of the symmetric space:
\begin{equation}\label{eq:cusp}
     H^{\bullet}_{\textit{cusp}}(S^{M},\mathcal{M}_{{^{\iota}\mu},\mathbb{C}}) := H^{\bullet}(\mathfrak{g},K_{\infty},\mathcal{C}^{\infty}_{\textit{cusp}}(M(\mathbb{Q})\backslash M(\mathbb{A})/K_f ,w_{\infty}^{-1}) \otimes \mathcal{M}_{\mu}).
\end{equation}
Furthermore, by \cite[8.27]{HarderB}:
\[
    H^{\bullet}_{\textit{cusp}}(S^{M},\mathcal{M}_{{^{\iota}\mu},\mathbb{C}}) \subset H^{\bullet}_{!}(S^{M},\mathcal{M}_{{^{\iota}\mu},\mathbb{C}}) \subset H^{\bullet}_{(2)}(S^{M},\mathcal{M}_{{^{\iota}\mu},\mathbb{C}}) \subset H^{\bullet}(S^{M},\mathcal{M}_{{^{\iota}\mu},\mathbb{C}}).
\]
Now we can define $Coh_{\textit{cusp}}(M,\mu) \subset Coh_{!}(M,\mu)$, as the set of all Hecke summands in the isotypical decomposition of inner cohomology (\ref{eq;isoinner}) contributing to cuspidal cohomology at the transcendental level. The decomposition of cusp forms to the automorphic cuspidal representation $\sigma = \sigma_{\infty}\times \sigma_f$, will give us:
\begin{equation}\label{eq:cuspdec}
    H^{\bullet}_{\textit{cusp}}(S^{M},\mathcal{M}_{{^{\iota}\mu},\mathbb{C}}) = \bigoplus_{\sigma_f \in Coh_{\textit{cusp}}(M,\mu,K_f)} m(\sigma) H^{\bullet}(\mathfrak{g},K_{\infty},\sigma_{\infty} \otimes \mathcal{M}_{\mu}) \otimes \sigma_f,
\end{equation}
where $m(\sigma)$ is the multiplicity of $\sigma$. In our case, multiplicity is one as the Levi factor is isomorphic to $GL_2$.

\section{Langlands-Shahidi method: Case $G_2$}\label{ch:LS}
In this section, we will review the Langlands-Shahidi method and the automorphic $L$-functions appearing in the constant term of the Eisenstein series associated with the cuspidal representations of the maximal parabolic subgroups of the group $G$.

\subsection{The case $P_{\beta}$}
Let us consider the maximal parabolic subgroup $P_{\beta}$ as in Fig \ref{fig:Dynkinb}. Therefore, $M_{\beta}\simeq GL_2$ and the root space of Lie Algebra of the unipotent radical $U_{\beta}$, are generated by the root spaces corresponding to the roots:
\[
    \alpha,\; \alpha+\beta,\; 3\alpha+2\beta,\; 2\alpha+\beta,\; 3\alpha+\beta.
\]

 Let $X^{*}(M_{\beta})=Hom(M_{\beta} ,\mathbb{Q})$, and let $A_{\beta}$ be the maximal split torus in the center of $M_{\beta}$. By considering the restriction map from $M_{\beta}$ to $A_{\beta}$, we can define the real Lie algebra of $A_{\beta}$ which is a vector space:
 \[
    \mathfrak{a}_{\beta} = \operatorname{Hom}(X^{*}(M_{\beta}),\mathbb{R}),
 \]
 and its dual is given by:
 \[
   \mathfrak{a}_{\beta}^{*}=X^{*}(M_{\beta})\otimes \mathbb{R}, \quad \mathfrak{a}_{\beta,\mathbb{C}}^{*}=X^{*}(M_{\beta})\otimes \mathbb{C}.
 \]
The half sum of the positive roots for $M_{\beta}$ is:
\[
\rho_{\beta} = 5\alpha + \frac{5}{2}\beta= \frac{5}{2}\gamma_s \in X^{*}(M_{\beta}),
\]
The modulus character of $P_{\beta}$,  is trivial on $U_{\beta}$, and on $M_{\beta}$ is given by:
\[
    \delta_{\beta}(m) := |det(Ad_{U_{\beta}}(m))|=|2\rho_{\beta}(m)| = |det(m)|^{5} \in \mathfrak{a}_{\beta}^{*}.
\]
Note that $\alpha_{\beta}:=\alpha$ is the single simple root of $G_2$ in $U$, and $(\;,\;)$ is a Weyl group invariant inner product on $ \mathfrak{a}_{\beta}^{*}$. We have:
\[
    \langle \rho_{\beta},\alpha_{\beta}\rangle = 2 \frac{(\rho_{\beta},\alpha_{\beta})}{(\alpha_{\beta},\alpha_{\beta})}= \frac{5}{2}.
\]
The fundamental weight corresponding to the simple root $\alpha_{\beta}$ will be:
\[
    \gamma_{\beta}:=\tilde{\alpha}_{\beta} := \langle \rho_{\beta},\alpha_{\beta}\rangle^{-1}\rho_{\beta}=\frac{2}{5}\rho_{\beta}=\gamma_s.
\]

\begin{defn}\label{def:poe}
     The half-integer $k$, defined as
\[
    k = k_{\beta} := -\langle \rho_{\beta}, \alpha_{\beta}\rangle = -\frac{5}{2},
\]
is called the point of evaluation for the case $P_{\beta}$.   
\end{defn}

The connected complex reductive group $^{L}G^{\circ} = G(\mathbb{C})$ is the Langlands dual of $G$, and $ ^{L}\!{P}_{\beta}$, $^{L}\!{M}_{\beta}=GL_{2}(\mathbb{C})$, and $^{L}\!{\mathfrak{u}}_{\beta}$ be the Lie algebra of the $^{L}\!{U}_{\beta}$. The dual group $^{L}\!{M}_{\beta}$ acts on $^{L}\!{\mathfrak{u}}_{\beta}$ by the adjoint action $r$. For each positive root $\lambda$ where $X_{\lambda^{\vee}}\in\mathfrak{u}_{\beta}$, we have $m=2$, i.e. $1\leq \langle \tilde{\alpha}_{\beta}, \lambda \rangle \leq m=2$. Also,
\begin{align*}
    V_1&= \langle X_{\alpha^{\vee}}, X_{\alpha^{\vee}+\beta^{\vee}}, X_{\gamma_l^{\vee}},X_{\alpha^{\vee}+\beta^{\vee}}\rangle,\\
    V_2&= \langle X_{\gamma_s^{\vee}}\rangle.
\end{align*}

The adjoint action $r$ leaves $V_j$ stable. Let us define $r|_{V_{j}}=: r_j$ and denote the contragredient representation of $r_{j}$ by $\tilde{r}_j$. Let $\rho_2$ be the standard representation of $GL_{2}(\mathbb{C})$, 1-dimensional representation $\wedge^{2}\rho_2$ is the exterior square representation of $GL_{2}(\mathbb{C)}$ given by central character $\omega$ and the 4-dimensional representation $Ad^{3}(\rho_2) := Sym^3(\rho_2)\otimes (\wedge^2 \rho_2)^{-1}$ \cite{SH89}. Then we have the following decomposition:
\begin{align}
        r = r_1\oplus r_2 = Sym^3(\rho_2)\otimes (\wedge^2\rho_2)^{-1} \oplus \wedge^{2}\rho_2.
\end{align}

\begin{figure}[h]
    \centering
    \begin{tikzpicture}
        \node[circle,radius=1pt,draw] (A) at (0,0) {};
        \node [] (C) at (0.5,0) {};
        \node [circle,fill=black,radius=1pt,draw,label={$\beta$}] (B) at (1,0) {};
        \draw[-] (A) -- (C);
        \draw[-{Classical TikZ Rightarrow[black,length=2mm]}] (B) -- (C.west);
        \draw[-] (0.13,0.1)--(0.87,0.1);
        \draw[-] (0.13,-0.1)--(0.87,-0.1);
        \draw[decoration={
            text along path, text={{$m=2; \quad r_1=Ad^{3}(\rho_2), \; r_2=\wedge^{2}\rho_2$}{}},
            text align={center},
            raise=0.2cm},
            decorate] (0.13,-1)--(0.87,-1); 
    \end{tikzpicture}
    \caption{Simple Factors $P_{\beta}$ of $G_2$}
    \label{fig:Dynkinb}
\end{figure} %Dynkin Diagram

Let $(\sigma, V_{\sigma})$ be an irreducible admissible representation of $M_{\beta}$. Then \mbox{$H_{\beta}:M_{\beta} \to \mathfrak{a}_{\beta}$} is the Harish-Chandra homomorphism, and  $\operatorname{Ind}$ is normalized parabolic induction. Identify $s\in \mathbb{C}$ with $\nu_s := s\tilde{\alpha}_{\beta} = s\gamma_s \in \mathfrak{a}^{*}_{P_{\beta},\mathbb{C}}$. Define the induced representation
\begin{align*}
     I_{\beta}(s,\sigma):=I^{G}_{P_\beta}(s, \sigma) &:= \operatorname{Ind^{G}_{P_{\beta}}}(\sigma\otimes |\;|^{s}),
\end{align*}
where $Ind$ is a normalized parabolic induction. The representation space $V(s,\sigma):= V(s\gamma_s, \sigma)$ consisting of all locally constant functions $f:G\to V_{\sigma}$ such that
\begin{align*}
    f(mng)  &=\sigma(m)\exp(\langle \nu_{s}+\rho_{\beta},H_{\beta}(m)\rangle)f(g), &\forall g\in G, \; m \in M_{\beta}, \; n\in U_{\beta}, \\
            &= \sigma(m)\delta_{P_\beta}(m)^{\frac{1}{2}+\frac{s}{5}}f(g).
\end{align*}
Moreover, in terms of algebraically parabolically induced representation, which means un-normalized, we have
\begin{equation}
    I_{\beta}(s,\sigma) = {^{a}}Ind(\sigma \otimes |\;|^{s+5/2}).
\end{equation}
Therefore, at the point of evaluation, we have:
\begin{equation}
    I_{\beta}(-5/2,\sigma) = {^{a}}Ind(\sigma).
\end{equation}

Let $K = K_{\infty} \times \prod_{v < \infty}G(\mathcal{O}_{v}) \subset G(\mathbb{A})$ be a maximal compact subgroup of $G(\mathbb{A})$.
Let $G_{v} = G(F_{v})$ and $P_{\beta,v}, M_{\beta,v}, U_{\beta,v}$ be the corresponding group of $F_{v}$-rational points. There is a natural embedding $X^{*}(M_{\beta})_{F}\to X^{*}(M_{\beta})_{F_{v}}$ and, consequently, an embedding:
\[
    \mathfrak{a}_{P_{{\beta}},v}:= \operatorname{Hom}(X^{*}(M_{\beta})_{F_{v}},\mathbb{R})\to \mathfrak{a}_{P_{{\beta}}}:= \operatorname{Hom}(X^{*}(M_{\beta})_{F},\mathbb{R}).
\]

Let $\sigma = \otimes'_{v} \sigma$ be a cuspidal representation of $M_{\beta}(\mathbb{A})$ with central character $\omega_{\sigma}$. For any $K$-finite function $\phi$ in the space $\sigma$, we extend $\phi$ to the unique function $\tilde{\phi}$ on $G(\mathbb{A})$ \cite{Shahidi1981OnCL} and set:
\begin{align*}
    \Phi_{s}(g) &:= \tilde{\phi}(g)\exp(\langle \nu_{s}+\rho_{\beta},H_{\beta}(g)\rangle), & g\in G(\mathbb{A}).
\end{align*}

For each $s\in \mathbb{C}$, the representation of $G(\mathbb{A})$ by right shifts on the space of $\Phi_{s}$ is equivalent to $I_{\beta}(s,\sigma)$. The Eisenstein series is defined by
\[
    E(s,\tilde{\phi},g,P_{\beta}) = \sum_{\gamma \in P_{\beta}(F)\backslash G(F)} \Phi_{s}(\gamma g),
\]
which converges for $Re(s) \gg 0$ and extends to a meromorphic function of $s\in \mathbb{C}$ with only finite number of simple poles in the half-plane $Re(s)\geq 0$.

As $w_{\gamma_s}(\{\beta\}) = \{\beta\}$ and $w_{\gamma_s}(\alpha)<0$, set $w_{0}:= w_{\gamma_s}$ and then the associate parabolic subgroup of $P_{\beta}$ will be 
\[
w_{0}P_{\beta}:= w_{\gamma_s}M_{\beta}w_{\gamma_s}^{-1} = P_{\beta},
\]
which implies that $P_{\beta}$ is self-associate. We can form the constant term of Eisenstein series along $P_{\beta}$, as follow:
\[
  E_{0}(s,\tilde{\phi}, g,P_{\beta}) :=   \int_{U_{\beta}(\mathbb{Q})\backslash U_{\beta}(\mathbb{A})} E(s,\tilde{\phi}, ng,P_{\beta})  dn.
\]

To describe the constant term, let us define the standard intertwining operator :
\begin{align*}
    T_{st,\iota}(s,\sigma,w_{0})f(g) &:=\int_{U^{-}(\mathbb{A})} f({w_{0}}^{-1}ng)dn, &g\in G(\mathbb{A}), f\in I_{\beta}(s,\sigma).
\end{align*}
which intertwines $I_{\beta}(s,\sigma)$ and $I_{\beta}(-s,w_{0}(\sigma))$ away from the poles. The action of $w_{\gamma_s}$ on $\sigma$ is given by $ w_{0}(\sigma)(m) = \sigma(\gamma_s^{-1} m \gamma_s )$ for $m\in M_{\beta}$. This operator is a meromorphic function of $s\in \mathbb{C}$. The constant term corresponds to $w_{\gamma_s}= w_{\alpha\beta\alpha\beta\alpha}$ is the only non-trivial term \cite{Lang}.  Therefore, for any $f\in I_{\beta}(s,\sigma)$, we have:
\begin{align*}
        E_{0}(s,f, g,P_{\beta}) &= \sum _{w\in W_{\beta}}(T_{st,\iota}(s,\sigma,w)f)(g).
\end{align*}

We can decompose $I_{\beta}(s,\sigma) = \otimes_{v}I_{\beta}(s,\sigma_{v})$, and we can define the local standard intertwining operator assigned to $P_{\beta}$ as:
\begin{align*}
    T_{st,v}(s)f_{v}(g) &=\int_{U_{\beta,v}^{-}} f_v(\gamma_{1}^{-1}ng)dn, &g\in G_{v}, f_{v}\in I_{\beta}(s,\sigma),
\end{align*}
where for $v\notin S$, the vector $f_{v}$ is spherical; i.e. it is the unique $K_{v}$-fixed function normalized by $f_{v}(e_{v})=1$, where $e_v$ is the identity in $G(k_v)$. The intertwining operator $T_{\operatorname{st}}$ maps $I_{\beta}(s,\sigma_{v})$ to $I_{Q}(-s, w_{0}\sigma_v)$, where $\tilde{f}_{v}$ is a spherical vector in the image of the local intertwining operator. 

The partial and global L-functions attached to $\sigma$ and $r_j$ are $L^{S}(s,\sigma,r_j) = \prod_{v\notin S} L(s,\sigma_v,r_j)$ and $L(s,\sigma,r_j)=\prod_{v} L(s,\sigma_v,r_j)$, respectively. Here,  $L(s,\sigma_v,\tilde{r_j})$ is the Langlands' local L-function attached to $\sigma_{v}$ and $r_j$. Let $f = \otimes_v f_v \in \otimes_{v} I_{\beta}(s,\sigma)$, where $f_v$ is a spherical vector for $v\notin S$. Then, we have:
\begin{align*}
    T_{st,\iota}(s,\sigma, w_{\gamma_s})f = &\bigotimes_{v\in S} A(s,\sigma_{v},w_{\gamma_s})f_{v} \otimes\bigotimes_{v\notin S}\tilde{f}_{v}\times r^{S}(s,\sigma,w_{\gamma_s}),
\end{align*}
where $r^{S}(s,\sigma,w_{0})$ is the ratio of the certain partial $L$-functions:
\begin{align*}
    r^{S}(s,\sigma,w_{0})&=\frac{L^{S}(s,\sigma,Ad^3(\rho_2))L^{S}(2s,\sigma,\wedge^2 \rho_2)}{L^{S}(1+s,\sigma,Ad^3(\rho_2))L^{S}(1+2s,\sigma,\wedge^2 \rho_2)}.
\end{align*}

Let $^{L}M_{\beta}$ be the $L$-group of $M_{\beta}$ as a group over $F_{v}$, having the natural homomorphism $\iota_{v}: \;^{L}M_{v}\to\; ^{L}M$, we can define the representation $r_{j,v}:= r_j \circ \iota_{v}$ of $^{L}M_v$. The $L$-functions appearing in the constant Eisenstein series, are \cite{SH89}:
\begin{align*}
    L(s,\sigma_v, \wedge^2 \rho_2) &= L(s,\omega_{\sigma_{v}})& (\text{Hecke L-function}),\\
   L(s,\sigma_v\otimes \omega_{\sigma_v},Ad^{3}(\rho_2))&=  L(s,\sigma_v, Sym^{3}(\rho_2)) & (\text{Symmetric cube L-function}).
\end{align*}

If $\omega_{\sigma}$ is non-trivial, $L(s,\sigma)$ is entire while if $\omega_{\sigma}$ is trivial then $L(s.\omega_\sigma)$ is holomorphic and may have simple poles at $s=0$ and $s=1$ \cite{TateJohnTorrence1997FAiN}. By \cite{KSannals}, the $L$-function $L(s,\sigma,Ad^{3}(\rho_2))$ is entire unless the representation $\sigma$ is monomial. If $\sigma$ is monomial, the $L$-function attached to the adjoint cube will break into multiple of smaller $L$-functions, and it may have simple poles at $s=0$ and $s=1$. 

As $P_{\beta}$ is self-associate, the Eisenstein series $E_{0}(s,\tilde{\phi},P_{\beta})$ is holomorphic in $Re(s)>0$ unless $\sigma$ is self-dual, i.e. $w_{0}(\sigma)\simeq \sigma$ \cite{MogWal}. For self-dual representations $\sigma$, our Eisenstein series has poles with $Re(s)>0$ determined by determining the poles for the constant term. Now using the information about the poles of the certain L-functions summarized in the previous section, the Eisenstein series only has a pole in the following cases:
\begin{enumerate}
    \item At $s=\dfrac{1}{2}$, when the central character is trivial and the adjoint cubic L-function does not vanish at $s=\dfrac{1}{2}$.
    \item At $s=1$, when $\sigma$ is a monomial representation and the adjoint cubic L-function has a pole at $s=1$.
\end{enumerate}

\subsection{The case $P_{\alpha}$}
By removing the long root, one can form the parabolic subgroup assigned to the short root, we have $\alpha$ as a simple root of $M_{\alpha}$, as depicted in Fig \ref{fig:Dynkin}. The main idea in this section builds on the previous one, so we only include the key data needed to clarify the notation for future use. The Lie Algebra $\mathfrak{u}_{\alpha}$ of the unipotent radical $U_{\alpha}$ is generated by the root spaces corresponding to the roots:
\[
    \beta,\; \alpha+\beta,\; 3\alpha+2\beta,\; 2\alpha+\beta,\; 3\alpha+\beta.
\]
In this case $X^{*}(M_{\alpha})=Hom(M_{\alpha} ,\mathbb{Q})$, and the half sum of positive roots for $M_{\alpha}$ is:
\[
\rho_{\alpha} = \frac{9}{2}\alpha + 3\beta= \frac{3}{2}\gamma_s \in X^{*}(M_{\alpha}).
\]

The modulus character of $P_{\alpha}$ is given by:
\[
    \delta_{\alpha}(m) := |det(Ad_{U_{\alpha}}(m))|=|2\rho_{\alpha}(m)| = |det(m)|^{3} \in \mathfrak{a}_{\alpha}^{*}.
\]

We can see $\alpha_{\alpha}=\beta$ is the single simple root in $U$, and $(\;,\;)$ is a Weyl group invariant inner product on $ \mathfrak{a}_{\alpha}^{*}$. We have:
\[
    \langle \rho_{\alpha},\alpha_{\alpha}\rangle = 2 \frac{(\rho_{\alpha},\alpha_{\alpha})}{(\alpha_{\alpha},\alpha_{\alpha})}= \frac{3}{2}.
\]
The fundamental weight corresponding to the simple root $\alpha_{\alpha}$ will be:
\[
    \gamma_{\alpha}:=\tilde{\alpha}_{\alpha} := \langle \rho_{\alpha},\alpha_{\alpha}\rangle^{-1}\rho_{\alpha}=\frac{2}{3}\rho_{\alpha}=\gamma_l.
\]
The point of evaluation is
\begin{equation}
    k_{\alpha} := -\langle \rho_{\alpha}, \alpha_{\alpha}\rangle = -\frac{3}{2}.
\end{equation}

We have $m=3$, i.e. $1\leq \langle \tilde{\alpha}_{\alpha}, \lambda \rangle \leq m=3$. Also,
\begin{align*}
    V_1&= \langle X_{\beta^{\vee}}, X_{\alpha^{\vee}+\beta^{\vee}}\rangle,\\
    V_2&= \langle X_{\gamma_l^{\vee}}\rangle,\\
    V_3&= \langle X_{\gamma_s^{\vee}}, X_{\alpha^{\vee}+3\beta^{\vee}}\rangle.\\
\end{align*}

Then we have the following decomposition:
\begin{align}
        r = r_1\oplus r_2 \oplus r_3 = \rho_2 \oplus \wedge^{2}\rho_2 \oplus \rho_2 \otimes \wedge^{2}\rho_2.
\end{align}
\begin{figure}[htp]
    \centering
    \begin{tikzpicture}
        \node[circle,fill=black,radius=1pt,draw,label={$\alpha$}] (A) at (0,0) {};
        \node [] (C) at (0.5,0) {};
        \node [circle,radius=1pt,draw] (B) at (1,0) {};
        \draw[-] (A) -- (C);
        \draw[-{Classical TikZ Rightarrow[black,length=2mm]}] (B) -- (C.west);
        \draw[-] (0.13,0.1)--(0.87,0.1);
        \draw[-] (0.13,-0.1)--(0.87,-0.1);
        \draw[decoration={
            text along path, text={{$m=3; \quad r_1= \rho_2, \; r_2=\wedge^{2}\rho_2,\; r_3 = \rho_2 \otimes \wedge^{2}\rho_2$}{}},
            text align={center},
            raise=0.2cm},
            decorate] (0.13,-1)--(0.87,-1); 
    \end{tikzpicture}
    \caption{Simple Factors $P_{\alpha}$ of $G_2$}
    \label{fig:Dynkin}
\end{figure} %Dynkin Diagram

Let $(\sigma, V_{\sigma})$ be an irreducible admissible representation of $M_{\alpha}$, and identify $s\in \mathbb{C}$ with $\nu_s := s\tilde{\alpha}_{\alpha} = s\gamma_l \in \mathfrak{a}^{*}_{P_{\alpha},\mathbb{C}}$. Define the induced representation:
\begin{equation}
     I_{\alpha}(s,\sigma):=I^{G}_{P_\alpha}(s, \sigma) := Ind^{G}_{P_{\alpha}}(\sigma\otimes |\;|^{s}) = {^{a}}Ind(\sigma\otimes |\;|^{s+3/2}).
\end{equation}
Then, at the point of evaluation, we have:
\begin{equation}
     I_{\alpha}(s,\sigma) = {^{a}}Ind(\sigma).
\end{equation}
In the same way as the previous section, we can define the Eisenstein series, where $W_{\alpha} = \{1, w_{\gamma_{1}}= w_{\beta\alpha \beta\alpha\beta} \}$ and parabolic subgroup $P_{\alpha}$ is self-associate. Using the obvious analogue 
 for notation from the previous section, we have:
\begin{align*}
    T_{st,\iota}s,\sigma, w_{\gamma_l})f = &\bigotimes_{v\in S} A(s,\sigma_{v},\gamma_l)f_{v} \otimes\bigotimes_{v\notin S}\tilde{f}_{v}\times r^{S}(s,\sigma,w_{\gamma_l}),\\
    r^{S}(s,\sigma,w_{\gamma_l})&=\frac{L^{S}(s,\sigma,\rho_2)L^{S}(2s,\sigma,\wedge^2 \rho_2)L^{S}(3s,\sigma,\rho_2\otimes \wedge^2 \rho_2)}{L^{S}(1+s,\sigma,\rho_2)L^{S}(1+2s,\sigma,\wedge^2 \rho_2)L^{S}(1+3s,\sigma,\rho_2\otimes \wedge^2 \rho_2)}.
\end{align*}
where $r^{S}(s,\sigma,w_{\gamma_l})$ is the ratio of the certain partial $L$-functions:
\[
    L^{S}(s,\sigma,r_j) := \prod_{v\notin S} L(s,\sigma_{v},r_{j,v}).
\]
Moreover, the $L$-functions involved in this case are:
\begin{align*}
    L(s,\sigma_v, \rho_2) &= L(s,\sigma_v)& (\text{The standard L-function for }GL_{2}),\\
    L(s,\sigma_v, \wedge^2 \rho_2) &= L(s,\omega_{\sigma_{v}})& (\text{The Hecke L-function})\\
    L(s,\sigma_v, \rho_2\otimes \wedge^2 \rho_2) &= L(s,\sigma_v \otimes \omega_{\sigma_{v}})& (\text{Twisted standard L-function for } GL_{2}).
\end{align*}
For a self-dual representation $\sigma$, the Eisenstein series may have a pole with $Re(s)>0$, only at $s=\frac{1}{2}$, when the central character is trivial and the standard L-function does not vanish at $s=\frac{1}{2}$.

\section{Strong purity and the strongly inner cohomology}\label{ch:SVL}
To provide a cohomological interpretation of the Langlands-Shahidi method, we need to identify the arithmetic cohomological counterpart of the cuspidal representation and generate Eisenstein cohomological classes using the representation-theoretic approach of the Langlands-Shahidi method.

\subsection{Strong purity condition}
Let ${^{\iota}}\sigma = {^{\iota}}\sigma_f\times {^{\iota}}\sigma_{\infty}$ be a cuspidal representation of $GL_2(\mathbb{A})$ with ${^{\iota}}\sigma_f \in Coh_{cusp}(GL_2, {^{\iota}}\mu)$. We are interested in the representations $\sigma$ that contribute non-trivially to the cuspidal cohomology as defined in (\ref{eq:cusp}), or simply when $\sigma_f \in Coh_{cusp}(GL_2, \mu)$. This condition will force our representation to be essentially unitary \cite{Racub}, and by Wigner's lemma it has to be essentially conjugate self-dual \cite{Racub}. Therefore, we call the cuspidal representation that contributes to the cuspidal cohomology, the cuspidal cohomological representation. A weight $\mu = (\mu^{\eta})_{\eta:F\to \mathbb{C}} \in X^{+}_{\textit{alg}}(T_2 \times \mathbb{C})$ is called pure  if it satisfies the purity condition:
\begin{equation}\label{def:pure}
    a^{\eta}+ b^{\bar{\eta}} = \mathbf{w(\mu)} = a^{\bar{\eta}}+ b^{\eta} \quad \forall \eta: F\to \mathbb{C},
\end{equation}
where $\mathbf{w(\mu)}$  is an integer, referred to as the purity weight of $\mu$. We will denote the set of all pure weights by $X_{0}^{+}(T\times \mathbb{C})$. By \cite[Thm 4.9]{cl}, pure weights are the only ones supporting cuspidal cohomological representations.

For $\gamma \in Aut(\mathbb{C})$, let ${^{\gamma}\mu} = ({\mu^{\gamma^{-1}\circ\eta}})_{\eta: F \to \mathbb{C}}$. By \cite[Lem 3.19]{cl}, the weight ${^{\gamma}\mu}$ supports cuspidal cohomology for $GL_2$ if $\mu \in X^{+}_{0}(T\times \mathbb{C})$. Therefore, it should satisfy the purity condition with the purity weight $\mathbf{w({^{\gamma}\mu})}$. Following \cite[Section 2.3.1]{RagTI}, a weight $\mu$ as above satisfies the strong purity condition if there is an integer $\mathbf{w}$ such that $\mathbf{w} = \mathbf{w({^{\gamma}\mu})}$ for all $\gamma \in Aut(\mathbb{C})$. The weight $\mu$ satisfying the strong purity is called strongly pure weight with purity weight $\mathbf{w}$. Denote the set of all strongly pure weights by $X^{*}_{00}(T\times \mathbb{C})$. It is clear from the definition that strongly pure weights support cuspidal cohomology of $GL_2$, therefore:
\[
    X^{*}_{00}(T\times \mathbb{C}) \subset X^{*}_{0}(T\times \mathbb{C}).
\]
The notion of cuspidal cohomology is not available at the arithmetic level. However, as the strong purity condition is algebraic, it allows us to replicate the same definition at the arithmetic level. Following Section \ref{sec:char}, let $E$ be the large enough finite Galois extension of $\mathbb{Q}$. Any embedding $\iota:E\to \mathbb{C}$ gives a bijection $X^{*}(T\times E) \to X^{*}(T\times \mathbb{C})$ that maps $\mu=(\mu^{\tau})_{\tau:F\to E}$ to $^{\iota}\mu= (\mu^{\iota^{-1}\circ\tau})_{\tau:F\to \mathbb{C}} $. 

\begin{defn}
    The weight $\mu \in X^{+}_{alg}(T\times E)$ is a strongly pure weight over $E$ with the purity weight $\mathbf{w}$, if for any embedding $\iota: E \to \mathbb{C}$ it maps to a strongly pure weight with purity weight $\mathbf{w}$ in transcendental level.
\end{defn}
We now need to generalize the purity condition to the strong purity condition to capture the combinatorics of strongly pure weights. The strong purity condition depends on the choice of embeddings $\tau: F \to E$ and $\iota: E \to \mathbb{C}$. Following \cite[Proposition 2.4]{RagTI}, the strong purity condition can be defined in three equivalent ways. A weight $\mu = (\mu^{\tau})_{\tau: F \to E} \in X^{+}_{alg}(T\times E)$ where $\mu^{\tau} = (a^{\tau},b^{\tau})$, is strongly pure if it satisfies one of the following equivalent conditions:
    \begin{enumerate}
        \item[(i)] For an embedding $\iota: E \to \mathbb{C}$, there exists an integer $\mathbf{w}$ such that:
        \[
         a^{\iota^{-1}\circ \gamma^{-1} \circ \eta}+ b^{\iota^{-1}\circ \gamma^{-1} \circ \bar{\eta}}=\mathbf{w}
        \]
        for all $\gamma \in Gal(\bar{\mathbb{Q}}/\mathbb{Q})$ and all $\eta:F\to \mathbb{C}$.
        
        \item[(ii)] There exists an integer $\mathbf{w}$ such that for any embedding $\iota: E \to \mathbb{C}$:
        \[
         a^{\iota^{-1}\circ \gamma^{-1} \circ \eta}+ b^{\iota^{-1}\circ \gamma^{-1} \circ \bar{\eta}}=\mathbf{w}
        \]
        for all $\gamma \in Gal(\bar{\mathbb{Q}}/\mathbb{Q})$ and all $\eta:F\to \mathbb{C}$.
        
        \item[(iii)]  For an embedding $\iota: E \to \mathbb{C}$, there exists an integer $\mathbf{w}$ such that:
        \[
         a^{\iota^{-1} \circ \eta}+ b^{\iota^{-1} \circ \bar{\eta}}=\mathbf{w}
        \]
        for all $\gamma \in Gal(\bar{\mathbb{Q}}/\mathbb{Q})$ and all $\eta:F\to \mathbb{C}$.
    \end{enumerate}
The set of all strongly pure weights over $E$ is denoted by $X^{*}_{00}(T\times E)$. We will have the following inclusions for character groups over $E$:
\[
    X^{*}_{00}(T\times E) \subset X^{+}_{alg}(T\times E) \subset X^{+}(T\times E) \subset X^{*}(T\times E).
\]

\begin{remark}
    These inclusions are generally strict. For some examples, we refer the reader to \cite{RagNH}. When $F$ is a CM field, then a pure weight is strongly pure \cite{RagTI}. This is not the case for general totally imaginary fields.
\end{remark}

Assume we have $F_0\subset F_1 \subset F$ as defined in Section \ref{sec:Preliminaries}, for the $TR$-case and $CM$-case. By \cite[Prop. 2.6.]{RagTI}, for $\mu \in X_{00}^{*}(Res_{F/\mathbb{Q}}(T)\times E)$ then there exists $\mu_1\in X_{00}^{*}(Res_{F_1/\mathbb{Q}}(T)\times E)$ such that $\mu$ is a base change of $\mu_1$ from $F_1$ to $F$, i.e $\forall\; \tau: F \to E$, if $\tau_1 = \tau|_{F_1}$ then $\mu_{1}^{\tau}=\mu^{\tau_1}$.

\subsection{Strongly inner cohomology}
The strongly pure weights are the one that contributes to the cuspidal cohomology groups at the transcendental level. We want to show that the strongly pure weights over $E$, capture the cuspidal cohomology group at the arithmetic level. For group $GL_n$ over a totally imaginary field it follows from \cite[Section 2.4]{RagTI}. Following the previous section, we have:

\begin{defn}
For the level structure $K_f$ of $GL_2/F$, the subset of inner spectrum of $\mu \in X_{00}^{*}(T\times E)$ for the level structure $K_f$ is defined as:
\[   
    Coh_{!!}(GL_2 ,\mu) = \{\sigma_f\in Coh_{!}(GL_2,\mu) \; | \; \exists\; \iota:E\to \mathbb{C} \text{ st. } {^{\iota}\sigma_f} \in Coh_{\textit{cusp}}(GL_2,{^{\iota}\mu})\}.
\]
It is called the strongly inner spectrum of $GL_2$ with $\mu$-coefficients and level structure $K_f$.
\end{defn}

The cuspidal cohomology at the transcendental level admits a rational structure by \cite[Thm 4.9]{cl}. Thus, the definition of the strongly inner spectrum of $GL_2$ with $\mu$-coefficients is independent of the choice of the embedding $\iota$. Following \cite{RagTI}, we can define the arithmetic counterpart of the cuspidal cohomology as:
\begin{equation}\label{def:strinner}
    H_{!!}^{\bullet}(S^{GL_2},\mathcal{M}_{\lambda}) = \bigoplus_{\sigma_f \in Coh_{!!}(GL_2, \mu)} H_{!}^{\bullet}(S^{GL_2},\mathcal{M}_{\lambda})(\sigma_f),
\end{equation}
it is called the strongly inner cohomology and it captures the cuspidal cohomology at the arithmetic level:
\[
    H_{!!}^{\bullet}(S^{GL_2},\mathcal{M}_{\lambda}) \otimes \mathbb{C} \simeq H_{cusp}^{\bullet}(S^{GL_2},\mathcal{M}_{\lambda}) .
\]

For $\iota: E \to \mathbb{C}$, we can form a cuspidal automorphic representation $^{\iota}\sigma = {^{\iota}\sigma_f} \times {^{\iota}\sigma_{\infty}}$ of $GL_2(\mathbb{A}_F)$ with $(^{\iota}\sigma)_f = \;^{\iota}\sigma_f$,  where $\sigma_f$ is a strongly inner Hecke summand of $\mu \in X^{*}_{00}(T\times E)$. Now, we need to find the representation at infinity ${^{\iota}\sigma_{\infty}}$, to be able to make a correspondence between strongly inner Hecke summand $\sigma_f$ in arithmetic level and a cuspidal cohomological representation $^{\iota}\sigma$ in transcendental level.

\subsection{Cohomological representations of $GL_2(\mathbb{C})$}\label{sec:cohrep}
We fix $\iota: E \to \mathbb{C}$ and omit $\iota$ from the notation for the rest of the section. Let $\mu \in X^{*}_{00}(T\times_{E} \mathbb{C})$ be a strongly pure weight. For any archimedean place $v$, fix the embeddings $\{\eta_v, \bar{\eta}_v \}$ of $F$ to $\mathbb{C}$, as in Section \ref{sec:Preliminaries}. Define the cuspidal parameters of $\mu$ at $v$:
\begin{equation}\label{def:cuspidalpar}
    \alpha^{v}: = -w_{0}\mu^{\eta_{v}}+\rho_{GL_2}, \quad \beta^{v}: = -\mu^{\bar{\eta}_{v}}-\rho_{GL_2},
\end{equation}
where $w_0$ is the longest element of $W^{P}$ as defined for each maximal parabolic subgroup. Denoting $\alpha^{v}=(\alpha^{v}_1,\alpha^{v}_{2})$ and $\beta^{v}=(\beta^{v}_1,\beta^{v}_{2})$, then the purity condition implies $\alpha^{v}_{j}+\beta^{v}_{j}= -\mathbf{w}$. The normalized parabolic induction:
\[
    \mathbb{J}_{\mu_{v}} := \mathbb{J}(\mu^{\eta_v},\mu^{\bar{\eta}_v}) := \operatorname{Ind}_{B_{2}(\mathbb{C})}^{GL_{2}(\mathbb{C})}(z^{\alpha_1^{v}}\bar{z}^{\beta_1^{v}}\otimes z^{\alpha_2^{v}}\bar{z}^{\beta_2^{v}}),
\]
defines a representation of $GL_2(\mathbb{C})$. Now we can form
\[
    \mathbb{J}_{\mu} := \otimes_{v\in S_{\infty}} \mathbb{J}_{\mu_v},
\]
which is a representation of $\prod_v GL_2(F_v)$. The following proposition explicitly determines the archimedean constituents of a cuspidal cohomological representation attached to a strongly inner Hecke summand.
\begin{prop}\label{archrep}
    Let ${\mu} \in X_{00}^{*}(T\times_{E} \mathbb{C})$ and $\mathbb{J}_{\mu}$ as above. Then:
    \begin{enumerate}
        \item[(i)] $\mathbb{J}_{\mu}$ is an irreducible essentially tempered representation admitting a Whittaker model.
        \item[(ii)] $H^{\bullet}(\mathfrak{gl}_2,K_{2,\infty}; \mathbb{J}_{^{\iota}\mu}\otimes \mathcal{M}_{{\mu},\mathbb{C}})\neq 0$.
        \item[(iii)] Let $\vartheta$ be an irreducible essentially tempered representation of $GL_2(\mathbb{R})$. If $H^{\bullet}(\mathfrak{gl}_2,K_{2,\infty}; \vartheta \otimes \mathcal{M}_{{\mu},\mathbb{C}})\neq 0$ then $\vartheta \simeq \mathbb{J}_{\mu}$.
        \item[(iv)] If $\sigma \in Coh_{\textit{cusp}}(G,\mu)$, then $\sigma_{\infty}=\mathbb{J}_{\mu}$.
    \end{enumerate}
\end{prop}
These archimedean constituents for the $CM$-case and $TR$-case are explicitly defined in \cite{RagTI}.

\subsection{The Tate twist}\label{sec:tt}
Let $\mu \in X^{*}_{00}(T_2\times E)$ be a strongly pure weight with purity weight $\mathbf{w}$, where $T_2$ is the maximal torus of $\mathbf{GL_2} := Res_{F/\mathbb{Q}}(GL_2)$ . For any integer $m \in \mathbb{Z}$, one can define $\mu+ m\delta_2$ where $\delta_2$ is the determinant character of $\mathbf{GL_2}$ and $\mu = (\mu^{\eta})_{\eta:F\to E} = (a^{\eta},b^{\eta})_{\eta:F\to E}$ by:
\[
    \mu+ m\delta_2 = (a^{\eta}+m,b^{\eta}+m)_{\eta:F\to E}.
\]
Furthermore, we have:
\begin{align*}
    a^{\eta}&+m, b^{\eta}+m \in \mathbb{Z}, &\hfill \textit{Integrality}\\
    a^{\eta}&+m\geq b^{\eta}+m, &\hfill \textit{Dominance}\\
    a^{\eta}&+m+ b^{\bar{\eta}}+m=b^{\eta}+m+ a^{\bar{\eta}}+m=\mathbf{w}+2m, &\hfill \textit{Strong Purity}
\end{align*}
so $\mu+m\delta_2 \in X^{*}_{00}(T_2\times E)$ is a strongly pure weight with purity weight $\mathbf{w}+2m$. By \cite[Section 5.2.4.]{HR}, we have the isomorphism $\mathcal{M}_{\mu,\mathbb{Q}}\otimes \mathbb{Q}[\delta_2] \xrightarrow{\sim} \mathcal{M}_{\mu+\delta_2,\mathbb{Q}}$, and the cup product with the fundamental class $e_{\delta_2}^m \in H^0 (S^{\mathbf{GL_2}},\mathbb{Q}[\delta_2])$ gives us the following isomorphism:
\[
    T^{\bullet}_{Tate}(m):H^{\bullet}_{!!}(S^{\mathbf{GL_2}},\mathcal{M}_{\mu}) \xrightarrow{\sim} H^{\bullet}_{!!}(S^{\mathbf{GL_2}},\mathcal{M}_{\mu+m\delta_2}),
\]
which maps any $\sigma_f\in Coh_{!!}(G,\mu)$ to $\sigma_{f}(-m) := \sigma_{f}\otimes | \;|^{-m}$.

\section{Critical Points of the Products of $L$-functions}\label{sec:critical}
In this section, we will compute the critical values of the $L$-functions attached to the maximal parabolic subgroup $P_{\beta}$. Since the calculations for both maximal parabolic subgroups follow the same procedure verbatim, we will omit the cumbersome calculations for the case of $P_{\alpha}$, and only report the results at the end of this section.

\subsection{The case $P_{\beta}$}\label{sec:criticalb}
Let $\mathbf{GL_2} = Res_{F/\mathbb{Q}}(GL_2)$ be the Levi factor of $P_\beta$. Fix an embedding $\iota: E \to \mathbb{C}$, which we omit from the notation throughout this section. Let $\mu \in X^*_{00}(T \times \mathbb{C})$. Write $\mu = (\mu^\eta)_{\eta: F \to \mathbb{C}}$, where $\mu^\eta = (a^\eta, b^\eta) \in \mathbb{Z}^2$ satisfies $a^\eta \geq b^\eta$ and $a^\eta + b^{\bar{\eta}} = \mathbf{w}$, where $\mathbf{w}$ is the purity weight. Let $\sigma_f \in \mathrm{Coh}_{!!}(\mathbf{GL_2}, \mu)$. By proposition \ref{archrep}, we can define $\sigma_\infty = \mathbb{J}_\mu$, so that $\sigma = \sigma_f \times \sigma_\infty$ is a cuspidal automorphic representation, cohomological with respect to the module $\mathcal{M}_\mu$ of highest weight $\mu$.

The adjoint symmetric transfer of $\mu$ is denoted by $Ad^3(\mu) = (Ad^3(\mu^{\eta}))_{\eta:F\to \mathbb{C}}$, we have:
\begin{align*}
    Ad^3(\mu^{\eta}) &:= (2a^{\eta}-b^{\eta},a^{\eta},b^{\eta},2b^{\eta}-a^{\eta})\in \mathbb{Z}^{4},\\
    b^{\eta}\leq a^{\eta} &\implies 2b^{\eta}-a^{\eta}\leq b^{\eta} \leq a^{\eta} \leq 2a^{\eta}-b^{\eta},\\
    2a^{\eta}-b^{\eta}+2b^{\bar{\eta}}-a^{\bar{\eta}} &=\mathbf{w},\qquad a^{\eta}+b^{\bar{\eta}}=\mathbf{w},\\
    b^{\eta}+a^{\bar{\eta}} &= \mathbf{w}, \qquad 2b^{\eta}-a^{\eta}+2a^{\bar{\eta}}-b^{\bar{\eta}}=\mathbf{w}.
\end{align*}

\begin{theorem}
    Let $\mu \in X^{*}_{00}(T\times \mathbb{C})$ . If $\sigma_f \in Coh_{!!}(\mathbf{GL_2},\mu)$ and  $Ad^3(\sigma)$ be a cuspidal automorphic representation of $Res_{F/\mathbb{Q}}(GL_{4})$, then:
    \[
        Ad^3(\sigma) \in Coh_{!!}(GL_{4},Ad^3(\mu)).
    \]
\end{theorem} 
\begin{proof}
    The proof following \cite{Racub} for the symmetric power transfers. Consider the archimedean place $v\in S_{\infty}$ and fix the pair of complex embeddings $(\eta,\bar{\eta})$ of $F$ into the field of complex numbers. We may assume that $\eta$ is the embedding that identifies the completion $F_{v}$ with $\mathbb{C}$. Now, define the cuspidal parameters of $Ad^3(\mu)$ at $v$ as follows:
\begin{align*}
    \alpha^{v} &:= -w_{0} Ad^{3}(\mu^{\eta})+\rho_4 = (\alpha^{v}_1,\dots,\alpha^{v}_4) \\
    &= (-2b^{\eta}+a^{\eta}+\frac{3}{2}, -b^{\eta}+\frac{1}{2},-a^{\eta}-\frac{1}{2},-2a^{\eta}+b^{\eta}-\frac{3}{2}),\\
    \beta^{v} &:= -Ad^{3}(\mu^{\bar{\eta}})-\rho_4=(\beta^{v}_1,\dots,\beta^{v}_4) \\
    &= (2b^{\eta}-a^{\eta}-\frac{3}{2}-\mathbf{w}, b^{\eta}-\frac{1}{2}-\mathbf{w},a^{\eta}+\frac{1}{2}-\mathbf{w},2a^{\eta}-b^{\eta}+\frac{3}{2}-\mathbf{w}),
\end{align*}
 where $\rho_4$ is half the sum of positive roots for $GL_4$. For any $1\leq i \leq 4$, we have 
 \begin{align*}
    \alpha^{\eta}_i, \beta^{\eta}_i \in \frac{3}{2}+\mathbb{Z}, \quad \alpha^{\eta}_i+\beta^{\eta}_i = -\mathbf{w}.
\end{align*} 

Then, one can form the representation $\mathbb{J}_{Ad^{3}(\mu^{v})}$ of $GL_4(\mathbb{C})$ using normalized parabolic induction:
 \[
    \mathbb{J}_{Ad^{3}(\mu^{v})} := \operatorname{Ind}_{B_4(\mathbb{C})}^{GL_4(\mathbb{C})}(z^{\alpha^{v}_1}\bar{z}^{\beta^{v}_1}\otimes\dots\otimes z^{\alpha^{v}_4}\bar{z}^{\beta^{v}_4}),
 \]
where $B_4$ is a subgroup of upper-triangular matrices of $GL_4$. Consequently, we can define a representation of $GL_4(\mathbb{R}) = \prod_v GL_4(F_v)$:
\[
    \mathbb{J}_{Ad^{3}(\mu)} := \bigotimes_{v\in S_{\infty}} \mathbb{J}_{Ad^{3}(\mu^{v})} .
\]
On the other hand, it is easy to see that the $Ad^3$ transfer of $\sigma_v = \mathbb{J}_{\mu_v}$ gives the same representation at infinity as it will have the same Langlands parameters. Therefore, $Ad^{3}(\sigma_v)$ has the highest weight $Ad^{3}(\mu_v)$.
\end{proof}

\begin{remark}
    By \cite{Kim2002CuspidalityOS}, the representation $Ad^{3}(\sigma)$ is cuspidal unless $\sigma$ is either monomial or tetrahedral. The tetrahedral case doesn't support cohomological representation by \cite{GrobRasym}. Therefore, for the remainder of this section, we can safely assume that $\sigma$ is non-monomial.
\end{remark}

By Local-Langlands correspondence for archimedean case \cite{knappArch}, the associated local $L$-factors of $L(s,Ad^{3}(\sigma))$ is:
\begin{align*}
    L_{\infty}(s,Ad^3(\sigma)) &\approx \prod_{v\in S_{\infty}} \Gamma(s-\frac{\mathbf{w}}{2}+\frac{|4b^{\eta}-2a^{\eta}-\mathbf{w}-3|}{2})\\
    &\hspace{40pt}.\Gamma(s-\frac{\mathbf{w}}{2}+\frac{|2b^{\eta}-\mathbf{w}-1|}{2})\\
    &\hspace{40pt}.\Gamma(s-\frac{\mathbf{w}}{2}+\frac{|2a^{\eta}-\mathbf{w}+1|}{2})\\ 
    & \hspace{40pt}.\Gamma(s-\frac{\mathbf{w}}{2}+\frac{|4a^{\eta}-2b^{\eta}-\mathbf{w}+3|}{2}),
\end{align*}
where $\approx$ means up to nonzero constants and exponential functions. Moreover, let $Ad^{3}(\sigma)^{\vee}$ be the dual of $Ad^{3}(\sigma)$, we have:
\begin{align*}
    L_{\infty}(1-s,Ad^3(\sigma)^{\vee}) &\approx \prod_{v\in S_{\infty}} \Gamma(1-s+\frac{\mathbf{w}}{2}+\frac{|4b^{\eta}-2a^{\eta}-\mathbf{w}-3|}{2})\\
    &\hspace{40pt}.\Gamma(1-s+\frac{\mathbf{w}}{2}+\frac{|2b^{\eta}-\mathbf{w}-1|}{2})\\
    &\hspace{40pt}.\Gamma(1-s+\frac{\mathbf{w}}{2}+\frac{|2a^{\eta}-\mathbf{w}+1|}{2})\\ 
    &\hspace{40pt}.\Gamma(1-s+\frac{\mathbf{w}}{2}+\frac{|4a^{\eta}-2b^{\eta}-\mathbf{w}+3|}{2}).
\end{align*}

Following \cite[Section 7.1]{HR}, we define quantities: 
\begin{align*}
    a_1(\mu) &= -\frac{\alpha^{\eta}_i+\beta^{\eta}_i}{2}= \frac{\mathbf{w}}{2},\\
    \ell_1(\mu) &= \min_{\eta:F\to \mathbb{C}} \{ |4b^{\eta}-2a^{\eta}-\mathbf{w}-3|,|2b^{\eta}-\mathbf{w}-1|,\\ &\qquad\qquad|2a^{\eta}-\mathbf{w}+1|, |4a^{\eta}-2b^{\eta}-\mathbf{w}+3\}.
\end{align*}
that are both invariant under different embeddings. We call $a_1(\mu)$ and $\ell_{1}(\mu)$ the abelian width and cuspidal with of $Ad^{3}(\mu)$, respectively. We may omit $\mu$ from the notation for simplicity. The set of critical points of $L(s,Ad^{3}(\sigma))$ is denoted by $Crit(Ad^3({\mu}))$. Using the fact that the gamma factor $\Gamma(s)$ has no poles for $Re(s) > 0$, we have:
\[
    Crit(Ad^3(\mu)) = \{ s_0 \in \frac{3}{2}+\mathbb{Z} \;|\; 1+a_1-\frac{\ell_1}{2} \leq s_0 \leq a_1+\frac{\ell_1}{2}\},
\]
which form a set of contiguous half integers centered around $\frac{1}{2}+a_1$ and of length $\ell_1$. We want the point of evaluation, $-\frac{5}{2}$, and $1-\frac{5}{2}$ to be in $Crit(Ad^{3})$. This condition satisfies if and only if:
\[
    -\frac{3}{2}-\frac{\ell_1}{2}\leq a_1 \leq -\frac{7}{2}+\frac{\ell_1}{2},
\]
or equivalently:
\[
    -3-\ell_1 \leq \mathbf{w} \leq -7+\ell_1.
\]
Note that inequalities above are valid only if $\ell_1\geq 2$.

The central character of $\sigma$ is $\omega_{\sigma} \in Coh_{!!}(\operatorname{Res}_{F/\mathbb{Q}}(GL_1), det(\mu))$ with the highest weight $det(\mu) = (det(\mu^{\eta}))_{\eta:F\to \mathbb{C}}$. We have:
\begin{align*}
    det(\mu^{\eta})=a^{\eta}+b^{\eta}\in \mathbb{Z},\\
    det(\mu^{\eta})+det(\mu^{\bar{\eta}})=2\mathbf{w}.
\end{align*}
So $det(\mu) \in X^{*}_{00}(T_1\times \mathbb{C})$, where $T_1$ is the torus $\operatorname{Res}_{F/\mathbb{Q}}(GL_1)$, is a strongly pure weight with the purity weight $2\mathbf{w}$. Likewise, one can define cuspidal parameters of $det(\mu)$ at $v\in S_{\infty}$:
\begin{align*}
    \alpha^{v} &= -det(\mu^{\eta}),\\ 
    \beta^{v} &= det(\mu^{\bar{\eta}}) = det(\mu^{\eta}) -2\mathbf{w}.
\end{align*}
It is easy to see that: 
\begin{align*}
    \alpha^{v}, \beta^{v} \in \mathbb{Z}, \quad \alpha^{v}+\beta^{v} = -2\mathbf{w},
\end{align*} 
the associated local $L$-factor of $L(s,\omega_{\sigma})$ will be:
\begin{align*}
    L_{\infty}(s,\omega_{\sigma}) &\approx \prod_{v\in S_{\infty}}\Gamma(s-\mathbf{w}+|a^{\eta}+b^{\eta}-\mathbf{w}|),\\
    L_{\infty}(1-s,\omega_{\sigma}^{\vee}) &\approx \prod_{v\in S_{\infty}}\Gamma(1-s+\mathbf{w}+|a^{\eta}+b^{\eta}-\mathbf{w}|).
\end{align*}
We define the following quantities that are invariant under different embeddings:
\begin{align*}
    a_2(\mu) &= -\frac{\alpha^{\eta}+\beta^{\eta}}{2} = \mathbf{w},\\
    \ell_2(\mu) &= \min_{\eta:F\to \mathbb{C}} \{|a^{\eta}+b^{\eta}-\mathbf{w}|\}.
\end{align*}
Then $s_0 \in \mathbb{Z}$ is a critical points of $L(s,\omega_{\sigma})$ if and only if
\[
    1+a_2-\ell_2 \leq s_0 \leq a_2+\ell_2.
\]
Note that we are interested in studying the critical points of $L(2s,\omega_{\sigma})$ not $L(s,\omega_{\sigma})$. Again, we need $2(-\frac{5}{2}) = -5$ and $1+2(-\frac{5}{2}) = -4$ to be critical points of $L(s,\omega_{\sigma})$ which are satisfied if and only if
\[
    -4-\ell_2(\mu) \leq a_2(\mu) \leq -6+\ell_2(\mu).
\]
or equivalently:
\[
    -4-\ell_2 \leq \mathbf{w} \leq -6+\ell_2.
\]
These inequalities are valid only if $\ell_2\geq 1$.

The set of critical points of the product of the $L$-functions $L(s,Ad^3(\sigma))L(2s,\omega_{\sigma})$ is the intersection of critical set for $L$-functions involved and denoted by $Crit(\mu, P_{\beta})$: 
\begin{align*}
    Crit(\mu,P_{\beta}) = \{ s_0 \in \frac{3}{2}+\mathbb{Z} \;|\;&  1+a_1-\frac{\ell_1}{2} \leq s_0 \leq a_1+\frac{\ell_1}{2}, \\
    & \frac{1+a_2-\ell_2}{2} \leq s_0 \leq\frac{ a_2+\ell_2}{2}\} .  
\end{align*}

The evaluation point $-\frac{5}{2}$ is a critical point of the products $L(s,Ad^3(\sigma))L(2s,\omega_{\sigma})$ and $L(1+s,Ad^3(\sigma))L(1+2s,\omega_{\sigma})$ if and only if 
\begin{align*}
    -3-\ell_1(\mu)&\leq \mathbf{w} \leq -7+\ell_1(\mu),\\
    -4-\ell_2(\mu) &\leq \mathbf{w} \leq -6+\ell_2(\mu),\\
    \ell_1(\mu)\geq 2,&\qquad \ell_2(\mu)\geq 1.
\end{align*}
We can divide these into two cases, based on a comparison between $\ell_1(\mu)$ and $\ell_2(\mu)$:
\begin{itemize}
    \item[(I)] $\ell_1(\mu) \leq \ell_2(\mu) \implies -4-\ell_2(\mu) \leq -3-\ell_1(\mu) \leq \mathbf{w} \leq -7+\ell_1(\mu)\leq -6+\ell_2(\mu)$,
    \item[(II)] $\ell_1(\mu) > \ell_2(\mu) \implies -3-\ell_1(\mu) \leq -4-\ell_2(\mu) \leq \mathbf{w} \leq -6+\ell_2(\mu)\leq -7+\ell_1(\mu)$.
\end{itemize}
For any $\eta: F\to \mathbb{C}$ as $a^{\eta}\geq b^{\eta}$, we have:
\[
    4b^{\eta}-2a^{\eta}-3-\mathbf{w} < 2b^{\eta}-1-\mathbf{w} < a^{\eta}+b^{\eta}-\mathbf{w} < 2a^{\eta}+1-\mathbf{w} < 4a^{\eta}-2b^{\eta}+3-\mathbf{w}.
\]
The fact that embeddings can impose different orderings for $\ell_1$ and $\ell_2$ adds complexity to writing the general case. To avoid this, we will focus on the case where $F$ is an imaginary quadratic number field, as this can naturally generalize to the totally imaginary field case. In this scenario, the only embeddings are $\{\eta,\bar{\eta}\}$. For simplicity, we denote $\mu^\eta = (a, b)$. With this setup, we can analyze the different cases based on the position of $0$ in the chain of inequalities above, as illustrated in Figure \ref{fig:zerosbeta}:
\begin{figure}[htp]
     \centering
      \includegraphics[width=\textwidth]{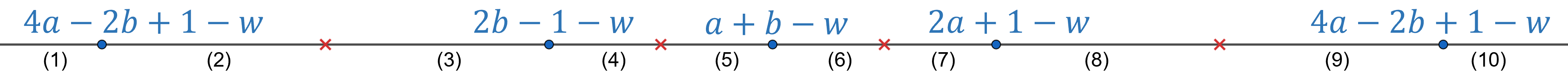}
     \caption{Cases for the Place of Zero, $P_{\beta}$.}
     \label{fig:zerosbeta}
 
\end{figure}
\begin{enumerate}
    \item $\mathbf{0 \leq 4b-2a-3-\mathbf{w}}$. In this, case we have
    \begin{align*}
        \ell_1 = 4b-2a-3-\mathbf{w} < a+b-\mathbf{w} = \ell_2.
    \end{align*}
    So it is the case \textbf{(I)}.
    \[
    \begin{cases}
       2a-4b+\mathbf{w}\leq \mathbf{w} \leq -10+4b-2a-\mathbf{w};\\
       4b-2a-3-\mathbf{w}\geq 2.
    \end{cases}
    \]
    We can rewrite it as:
    \[
    \begin{cases}
        a \geq b;\\
        a \leq 2b;\\
        a \leq 2b-\mathbf{w}-5;\\
        a \leq 2b-\frac{5+\mathbf{w}}{2}.
    \end{cases}
    \]
    Using the fact that $a\geq b$ and varying the value for $w$, we will have:
    \[
    \begin{cases}
        b \leq a \leq 2b, & \mathbf{w}\leq -5;\\
        b \leq a \leq 2b-\mathbf{w}-5, & \mathbf{w}\geq -5.
    \end{cases}
    \]
    
    \item $\mathbf{4b-2a-3-w\leq 0 \leq 3b-a-2-w}$ In this case we have
    \begin{align*}
        \ell_1 = 2a-4b+3+\mathbf{w} < a+b-\mathbf{w} = \ell_2.
    \end{align*}
    So it is the case \textbf{(I)}.
     \[
        \begin{cases}
            4b-2a-6-\mathbf{w}\leq \mathbf{w} \leq 2a-4b-4+\mathbf{w};\\
            2a-4b+3-\mathbf{w}\geq 2.
        \end{cases}
    \]
    like the previous part, we have:
    \[
    \begin{cases}
        2b-3-\mathbf{w} \leq a \leq 3b-2-\mathbf{w}, & \mathbf{w}\leq -5;\\
        2b+2 \leq a \leq 3b-2-\mathbf{w}, & \mathbf{w} \geq -5.
    \end{cases}
    \]
    
    \item $\mathbf{3b-a-2-w\leq 0 \leq 2b-1-w}$ In this case we have
    \begin{align*}
        \ell_1 = 2b-1-\mathbf{w} < a+b-\mathbf{w} = \ell_2.
    \end{align*}
    So it is the case \textbf{(I)}.
     \[
        \begin{cases}
            -2b-2+\mathbf{w} \leq \mathbf{w} \leq 2b-8-\mathbf{w};\\
            2b-1-\mathbf{w}\geq 2.
        \end{cases}
    \implies
        \begin{cases}
            -1 \leq b \leq \frac{a+2+\mathbf{w}}{3}, & \mathbf{w} \leq -5;\\
            4+\mathbf{w} \leq b \leq \frac{a+2+\mathbf{w}}{3}, & \mathbf{w}\geq -5.
        \end{cases}
    \]
    \item $\mathbf{2b-1-w\leq 0 \leq \frac{a+3b-1}{2}-w }$. In this case, we have:
    \begin{align*}
        \ell_1 = -2b+1+\mathbf{w} < a+b-\mathbf{w} = \ell_2.
    \end{align*}
    So it is the case \textbf{(I)}.
     \[
        \begin{cases}
            2b-4-\mathbf{w} \leq \mathbf{w} \leq -2b-6+\mathbf{w};\\
            -2b+1+\mathbf{w}\geq 2.
        \end{cases}
    \implies
        \begin{cases}
            \frac{1-a+2\mathbf{w}}{3} \leq b \leq \mathbf{w}+2, & \mathbf{w}\leq -5;\\
            \frac{1-a+2\mathbf{w}}{3} \leq b \leq -3 & \mathbf{w}\geq -5.
        \end{cases}
    \]
    \item $\mathbf{\frac{a+3b-1}{2}-w \leq 0 \leq a+b-w}$. In this case, we have:
    \begin{align*}
        \ell_1 = -2b+1+\mathbf{w} > a+b-\mathbf{w} = \ell_2.
    \end{align*}
    So it is the case \textbf{(II)}.
     \[
        \begin{cases}
            -4-a-b+\mathbf{w} \leq \mathbf{w} \leq -6+a+b-\mathbf{w};\\
            a+b-\mathbf{w}\geq 1.
        \end{cases}
    \]
    Then,
    \[
        \begin{cases}
            -4-a \leq b \leq \frac{1-a+2\mathbf{w}}{3}, & \mathbf{w}\leq -5;\\
            2\mathbf{w}-a+6 \leq b \leq \frac{1-a+2\mathbf{w}}{3}, & \mathbf{w}\geq -5.
        \end{cases}
    \]
    \item $\mathbf{a+b-w \leq 0 \leq \frac{3a+b+1}{2}-w}$. In this case, we have:
    \begin{align*}
        \ell_1 = 2a+1-\mathbf{w} > -(a+b)+\mathbf{w} = \ell_2.
    \end{align*}
    So it is the case \textbf{(II)}.
     \[
        \begin{cases}
            a+b-4-\mathbf{w} \leq \mathbf{w} \leq -a-b-6+\mathbf{w};\\
            -a-b+\mathbf{w} \geq 1.
        \end{cases}
    \]
    Then,
    \[
        \begin{cases}
            2\mathbf{w}-3a-1 \leq b \leq 2\mathbf{w}-a+4, & \mathbf{w}< -5;\\
            2\mathbf{w}-3a-1 \leq b \leq -a-6, & \mathbf{w}\geq -5.
        \end{cases}
    \]
    \item $\mathbf{\frac{3a+b+1}{2}-w \leq 0 \leq 2a+1-w }$. In this case, we have:
    \begin{align*}
        \ell_1 = 2a+1-\mathbf{w} < -(a+b)+\mathbf{w} = \ell_2.
    \end{align*}
    So we are back to the case \textbf{(I)} again.
     \[
        \begin{cases}
            -2a-4+\mathbf{w} \leq \mathbf{w} \leq 2a-6-\mathbf{w};\\
            2a+1-\mathbf{w}\geq 2.
        \end{cases}
    \implies
    \begin{cases}
        -2\leq a \leq \frac{2\mathbf{w}-b-1}{3}, & \mathbf{w}\leq -5;\\
        \mathbf{w}+3\leq a \leq \frac{2\mathbf{w}-b-1}{3} ,& \mathbf{w} \geq -5.
    \end{cases}
    \]
    \item $\mathbf{2a+1-w \leq 0 \leq 3a-b+2-w }$. In this case, we have:
    \begin{align*}
        \ell_1 = -2a-1+\mathbf{w} < -(a+b)+\mathbf{w} = \ell_2.
    \end{align*}
    So it is the case \textbf{(I)}.
     \[
        \begin{cases}
            2a-2-\mathbf{w} \leq \mathbf{w} \leq -2a-8+\mathbf{w};\\
            -2a-1+\mathbf{w}\geq 2.
        \end{cases}
    \implies
        \begin{cases}
            \frac{b-2+\mathbf{w}}{3} \leq a \leq \mathbf{w}+1, & \mathbf{w} \leq 5;\\
            \frac{b-2+\mathbf{w}}{3} \leq a \leq  -4, & \mathbf{w} \geq 5.
        \end{cases}
    \]
    \item $\mathbf{3a-b+2-w \leq 0 \leq 4a-2b+3-w}$. In this case, we have:
    \begin{align*}
        \ell_1 = 4a-2b+3-\mathbf{w} < -(a+b)+\mathbf{w} = \ell_2.
    \end{align*}
    So it is the case \textbf{(I)}.
    
     \[
        \begin{cases}
            2b-4a-6+\mathbf{w} \leq \mathbf{w} \leq 4a-2b-4-\mathbf{w};\\
            4a-2b+3-\mathbf{w}\geq 2.
        \end{cases}
    \]
    Then,
    \[
    \begin{cases}
        3a+2-\mathbf{w} \leq b\leq 2a+3,& \mathbf{w}\leq -5;\\
        3a+2-\mathbf{w} \leq b \leq 2a-2-\mathbf{w}, & \mathbf{w}\geq -5.
    \end{cases}
    \]
    \item $\mathbf{4a-2b+3-w \leq 0}$. In this case, we have:
    \begin{align*}
        \ell_1 = -4a+2b-3+\mathbf{w} < -(a+b)+\mathbf{w} = \ell_2.
    \end{align*}
    So it is the case \textbf{(I)}.
     \[
        \begin{cases}
            4a-2b-\mathbf{w} \leq \mathbf{w} \leq -4a+2b-10+\mathbf{w};\\
            -4a+2b-3+\mathbf{w}\geq 2.
        \end{cases}
    \implies
        \begin{cases}
            2a-\mathbf{w} \leq b \leq a, & \mathbf{w}\leq -5;\\
            2a+5 \leq b \leq a, & \mathbf{w}\geq 5.
        \end{cases}
    \]
\end{enumerate}

\begin{figure}[htp]
    \centering
        \includegraphics[width=.35\textwidth]{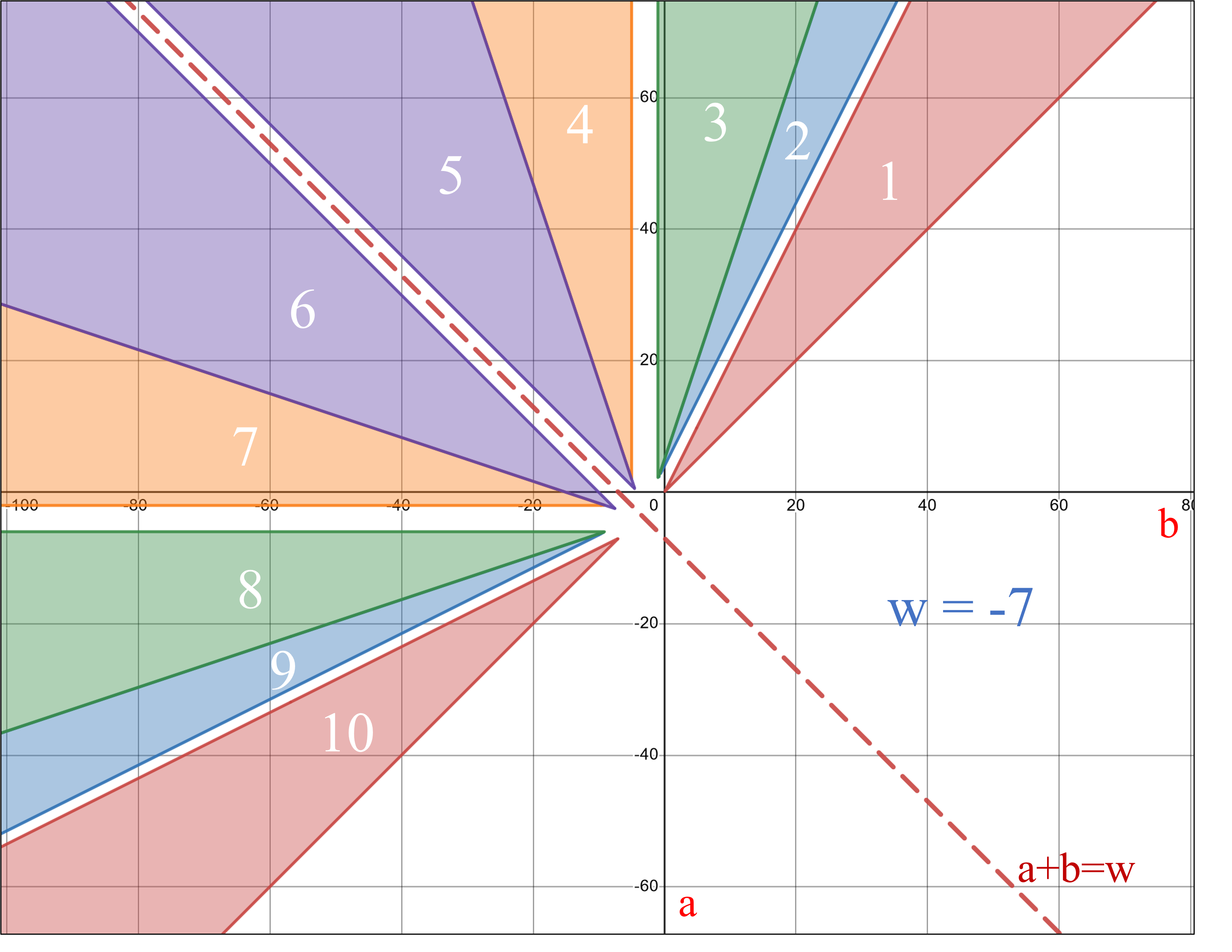}
        \includegraphics[width=.35\textwidth]{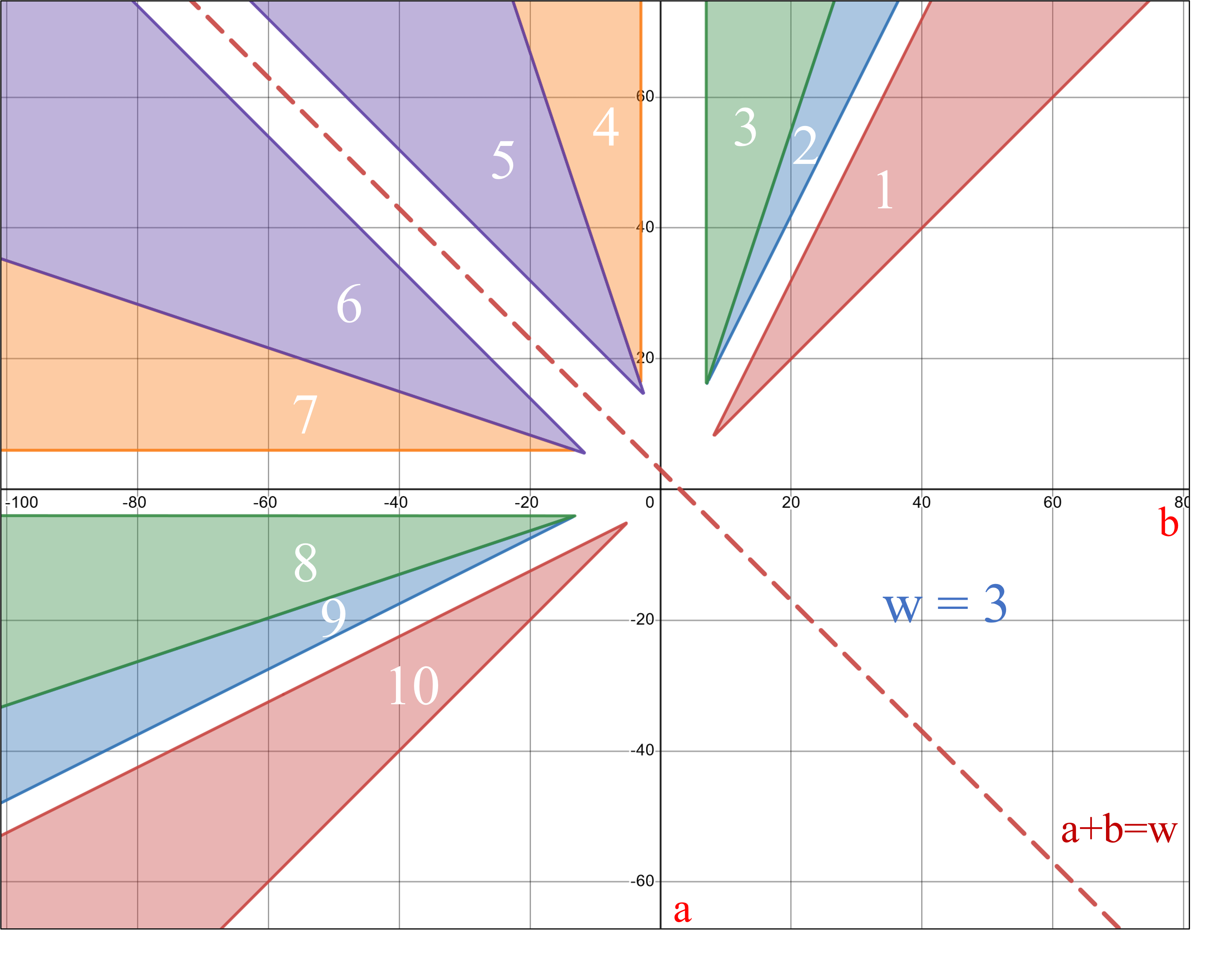}
    \caption{Critical Regions for the case $P_{\beta}$ at different $\mathbf{w}$.}
    \label{fig:criticalregbeta}
\end{figure}

We call each region above a critical region. The critical regions are disjoint and they share at most one boundary with another critical region. Figure \ref{fig:criticalregbeta} depicts the critical regions for different $\mathbf{w}\in \{-7,3\}$. Let us call the intersection point of boundaries of each critical region, the vertex of the region, and denote the vertex of the region $? \in \{(1),\; (2),\; \dots,\; (10)\}$ by $v_{(?)}=(b_0,a_0)$. 
\begin{itemize}
    \item Critical regions (2) and (3) share a boundary, the vertices are as follows:
    \[
     v_{(2)}=v_{(3)}=\begin{cases}
         (-1,-5-\mathbf{w}),&\mathbf{w}\leq -5;\\
         (4+\mathbf{w},10+2\mathbf{w}),&\mathbf{w}\geq -5.
     \end{cases}
    \]
    
     \item Critical regions (4) and (5) share a boundary, we have:
    \[
     v_{(4)}=\begin{cases}
         (2+\mathbf{w},-5-\mathbf{w}),&\mathbf{w}\leq -5;\\
         (-3,10+2\mathbf{w}), & \mathbf{w} \geq -5.
     \end{cases},\quad v_{(5)}=\begin{cases}
         (\frac{5}{2}+\mathbf{w},-\frac{13}{2}-\mathbf{w}),&\mathbf{w}\leq -5;\\
         (-\frac{5}{2},\frac{17}{2}+2\mathbf{w},-3),& \mathbf{w}\geq -5.
     \end{cases} 
    \]
    Even though critical regions share a boundary they have different vertices.
    
    \item Critical regions (6) and (7) share a boundary, we have:
    \[
     v_{(6)}=\begin{cases}
         (2\mathbf{w}+\frac{13}{2},-\frac{5}{2}),&\mathbf{w}\leq -5;\\
         (-\mathbf{w}-\frac{17}{2},\mathbf{w}+\frac{5}{2}), & \mathbf{w} \geq -5. 
     \end{cases},\quad v_{(7)}=\begin{cases}
         (2\mathbf{w}+5,-2),&\mathbf{w}\leq -5;\\
         (-\mathbf{w}-10,\mathbf{w}+3),& \mathbf{w}\geq -5.
     \end{cases}
    \]
    Therefore, the vertices are not aligned.
    
       \item Critical regions (8) and (9) share a boundary, we have:
    \[
     v_{(8)}=v_{(9)}=\begin{cases}
         (2\mathbf{w}+5,\mathbf{w}+1),&\mathbf{w}\leq -5;\\
         (-\mathbf{w}-10,-4), & \mathbf{w} \geq -5.
     \end{cases}
    \]
\end{itemize}

\begin{remark}\label{rmk: wrunit}
The analysis above shows that we can explicitly categorize different cases by considering whether $\mathbf{w} \leq -5$ or $\mathbf{w} > -5$. These two conditions have a deeper significance, which we elaborate on later in Section \ref{runit}, as they determine whether the representation lies on the right side of the unitary axis or not.
\end{remark}

\begin{prop}
    Assuming the notations above, there are two cases:
    \begin{enumerate}
    \item If $\ell_1 \leq \ell_2$ or equivalently $\frac{a^\eta+3b^\eta-1}{4}\leq \mathbf{w} \leq \frac{3a^\eta+b^\eta-1}{4}$, then :
        \[
            Crit(L(s,Ad^3(^{\iota}\sigma))L(2s,\omega_{^{\iota}\sigma})) = Crit(L(2s,\omega_{^{\iota}\sigma})) \subset Crit(L(s,Ad^3(^{\iota}\sigma))).
        \]
        \item Otherwise:
        \[
            Crit(L(s,Ad^3(^{\iota}\sigma))L(2s,\omega_{^{\iota}\sigma})) = Crit(L(s,Ad^3(^{\iota}\sigma))) \subset Crit(L(2s,\omega_{^{\iota}\sigma})).
        \]
    \end{enumerate}
\end{prop} 

\begin{remark}\
When $F$ is a TR-case, the product of Langlands-Shahidi L-functions at hand has no critical points. Assume $F$ is in $TR$-case. Base-change of a strongly pure weight over a totally real field $F_1$, as $\eta_v|_{F_1}= \bar{\eta}_v|_{F_1}$, we have:
 \[
    \alpha^{v}= w_0{\beta^{v}} \implies a^{\eta_v}= b^{\bar{\eta}_v} \;\textit{and} \; b^{\eta_v} = a^{\bar{\eta}_v}.
 \]
For the Hecke L-functions, it means: \[\ell_2(\mu) \leq |a^{\eta_v}+b^{\eta_v}- (a^{\bar{\eta}_v}+b^{\bar{\eta}_v})| = 0 \implies \ell_2(\mu) = 0. \]
\end{remark}

\subsection{The case $P_{\alpha}$}\label{sec:criticala} We summarize the results for the case $P_\alpha$ without providing detailed explanations, as the key arguments are identical to those for the case $P_\beta$. Let $\mathbf{GL_2} = Res_{F/\mathbb{Q}}(GL_2)$ be the Levi factor of $P_{\alpha}$, and we use the obvious analog of notations from the previous section. We aim to identify the critical points of the automorphic $L$-functions attached to $\sigma$, $\omega_{\sigma}$, and $\sigma \otimes \omega_{\sigma}$.

Let $\sigma_f \in Coh_{!!}(\mathbf{GL_2}, \mu)$. For $\sigma_{\infty} = \mathbb{J}_{\mu}$, the automorphic representation $\sigma = \sigma_f \otimes \sigma_{\infty}$ is cuspidal. By Local-Langlands correspondence, \cite{knappArch}:

\begin{align*}
    L_{\infty}(s,\sigma) &\approx \prod_{v\in S_{\infty}} \Gamma(s-\frac{\mathbf{w}}{2}+\frac{|2b^{\eta}-\mathbf{w}-1|}{2})\\
    &\hspace{40pt}.\Gamma(s-\frac{\mathbf{w}}{2}+\frac{|2a^{\eta}-\mathbf{w}+1|}{2}).
\end{align*}
 Moreover, let $\sigma^{\vee}$ be the dual of $\sigma$, we have:
\begin{align*}
      L_{\infty}(1-s,\sigma^{\vee}) &\approx \prod_{v\in S_{\infty}} \Gamma(1-s+\frac{\mathbf{w}}{2}+\frac{|2b^{\eta}-\mathbf{w}-1|}{2})\\
    &\hspace{40pt}.\Gamma(1-s+\frac{\mathbf{w}}{2}+\frac{|2a^{\eta}-\mathbf{w}+1|}{2}).
\end{align*}

Following the previous section, define:
\begin{align*}
    a_1(\mu) &= -\frac{\alpha^{\eta}_i+\beta^{\eta}_i}{2}= \frac{\mathbf{w}}{2},\\
    \ell_1(\mu) &= \min_{\eta:F\to \mathbb{C}} \{ |2b^{\eta}-\mathbf{w}-1|,|2a^{\eta}-\mathbf{w}+1|\}.
\end{align*}
By abuse of notation, we use the same notation as in the previous section. However, since the choice of parabolic will be clear, there will be no ambiguity. The set of critical points of $L(s,\sigma)$ is denoted by $Crit(\mu)$.  We have:
\[
    Crit(\mu) = \{ s_0 \in \frac{3}{2}+\mathbb{Z} \;|\; 1+a_1(\mu)-\frac{\ell_1(\mu)}{2} \leq s_0 \leq a_1(\mu)+\frac{\ell_1(\mu)}{2}\},
\]
which forms a set of contiguous half integers centered around $\frac{1}{2}+a_1(\mu)$ and of length $\ell_1(\mu)$. As we computed in \ref{ch:LS}, the point of evaluation is $m_0 = -\frac{3}{2}$. We want $-\frac{3}{2}$ and $1-\frac{3}{2}$ be in $Crit(\mu)$, which is satisfied if and only if
\[
    -\frac{1}{2}-\frac{\ell_1(\mu)}{2}\leq a_1(\mu) \leq -\frac{5}{2}+\frac{\ell_1(\mu)}{2},
\]
or equivalently:
\[
    -1-\ell_1 \leq \mathbf{w} \leq -5+\ell_1.
\]
The inequality above is valid only if $\ell_1(\mu)\geq 2$.

The calculations and notations for the critical points of $L(2s,\omega_{\sigma})$ are exactly the same as the one in the case $P_{\beta}$, the only difference is the final inequality for the cuspidal weight, as it depends on the point of evaluation. Therefore:
\[
    -2-\ell_2 \leq \mathbf{w} \leq -4+\ell_2,
\]
the inequality  above is valid only if $\ell_2(\mu)\geq 1$.

For $\sigma\otimes \omega_{\sigma}$, the highest weight is $\mu+\det(\mu) = (2a^{\eta}+b^{\eta},a^{\eta}+2b^{\eta})_{\eta:F\to \mathbb{C}}$, and:
\begin{align*}
    &2a^{\eta}+b^{\eta},a^{\eta}+2b^{\eta} \in \mathbb{Z},\\
    &2a^{\eta}+b^{\eta} > a^{\eta}+b^{\eta},\\
    &2a^{\eta}+b^{\eta}+ 2b^{\bar{\eta}}+a^{\bar{\eta}} = 3\mathbf{w} = a^{\eta}+2b^{\eta}+ b^{\bar{\eta}}+2a^{\bar{\eta}}.
\end{align*}
So, $\sigma\otimes \omega_{\sigma} \in Coh_{!!}(\mathbf{GL_2},\rho_2\otimes \wedge^{2}(\mu))$ with the purity weight $3\mathbf{w}$. The cuspidal parameters are:
\begin{align*}
    \alpha^{v} = (-a^{\eta}-2b^{\eta}+\frac{1}{2},-2a^{\eta}-b^{\eta}-\frac{1}{2}),
    \beta^{v} = (a^{\eta}+2b^{\eta}-\frac{1}{2}-3\mathbf{w},2a^{\eta}+b^{\eta}+\frac{1}{2}-3\mathbf{w}).
\end{align*}
We can form the $L$-factor:
\begin{align*}
    L_{\infty}(s,\sigma\otimes \omega_{\sigma}) &\approx \prod_{v\in S_{\infty}} \Gamma(s-\frac{3\mathbf{w}}{2}+|a^{\eta}+2b^{\eta}-3\mathbf{w}-1|)\\
    &\hspace{40pt}.\Gamma(s-\frac{3\mathbf{w}}{2}+|2a^{\eta}+b^{\eta}-3\mathbf{w}+1|).
\end{align*}
 Moreover, let $(\sigma\otimes \omega_{\sigma})^{\vee}$ be the dual of $\sigma$, we have:
\begin{align*}
      L_{\infty}(s,(\sigma\otimes \omega_{\sigma})^{\vee}) &\approx \prod_{v\in S_{\infty}} \Gamma(1-s+\frac{3\mathbf{w}}{2}+|a^{\eta}+2b^{\eta}-3\mathbf{w}-1|)\\
    &\hspace{40pt}.\Gamma(1-s+\frac{3\mathbf{w}}{2}+|2a^{\eta}+b^{\eta}-3\mathbf{w}+1|).
\end{align*}
Define the following quantities that are invariant under different embeddings:
\begin{align*}
    a_3(\mu) &= -\frac{\alpha^{\eta}_i+\beta^{\eta}_i}{2}= \frac{3\mathbf{w}}{2},\\
    \ell_3(\mu) &= \min_{\eta:F\to \mathbb{C}} \{ |a^{\eta}+2b^{\eta}-3\mathbf{w}-1|,|2a^{\eta}+b^{\eta}-3\mathbf{w}+1|\}.
\end{align*}

The set of critical points of $L(s,\sigma\otimes \omega_{\sigma})$ satisfy :
\[
   1+a_{3}(\mu)-\ell_{3}(\mu) \leq s_0 \leq a_3(\mu)+\ell_3(\mu).
\]
We want $-3\frac{3}{2}$ and $1-3\frac{3}{2}$ be critical points for this $L$-function, we have this property if and only if
\[
    -\frac{7}{2}-\ell_3(\mu)\leq a_3(\mu) \leq -\frac{11}{2}+\ell_3(\mu),
\]
or equivalently:
\[
    \frac{-7-\ell_3}{3} \leq \mathbf{w} \leq \frac{-11+\ell_3}{3},
\]
the inequality  above is valid only if $\ell_3(\mu)\geq 2$.

The evaluation point $-\frac{3}{2}$ is a critical point of $L(s,\sigma)L(2s,\omega_{\sigma})L(3s,\sigma\otimes\omega_{\sigma})$ and $L(1+s,\sigma)L(1+2s,\omega_{\sigma})L(1+3s,\sigma\otimes\omega_{\sigma})$ if and only if 
\begin{align*}
            -1-\ell_1(\mu)&\leq \mathbf{w} \leq -5+\ell_1(\mu),\\
            -2-\ell_2(\mu) &\leq \mathbf{w} \leq -4+\ell_2(\mu),\\
            -\frac{7}{3}-\frac{\ell_3(\mu)}{3} &\leq \mathbf{w} \leq -\frac{11}{3}+\frac{\ell_3(\mu)}{3},\\
            \ell_1(\mu)\geq 2,&\qquad \ell_2(\mu)\geq 1,\qquad \ell_3(\mu)\geq 2.
\end{align*}

Which is non-empty if all of the inequalities are non-trivial. We can break this in three cases, comparing $\ell_1(\mu)$, $\ell_2(\mu)$, and $\ell_3(\mu)$:
\begin{itemize}
    \item[(I)] $\ell_1(\mu) \leq \min\{\ell_2(\mu), \ell_3(\mu)\} \implies  -1-\ell_1(\mu)\leq \mathbf{w} \leq -5+\ell_1(\mu)$,
    \item[(II)] $\ell_2(\mu) \leq \min\{\ell_1(\mu), \ell_3(\mu)\}\implies -2-\ell_2(\mu) \leq \mathbf{w} \leq -4+\ell_2(\mu)$,
    \item[(III)] $\ell_3(\mu) < \min\{\ell_1(\mu), \ell_2(\mu)\} \implies \-\frac{7}{3}-\frac{\ell_3(\mu)}{3} \leq \mathbf{w} \leq -\frac{11}{3}+\frac{\ell_3(\mu)}{3}$.
\end{itemize}

Following over calculation for $P_{\beta}$, we can calculate the critical regions for $P_{\alpha}$.
\begin{figure}[htp]
    \centering
        \includegraphics[width=.35\textwidth]{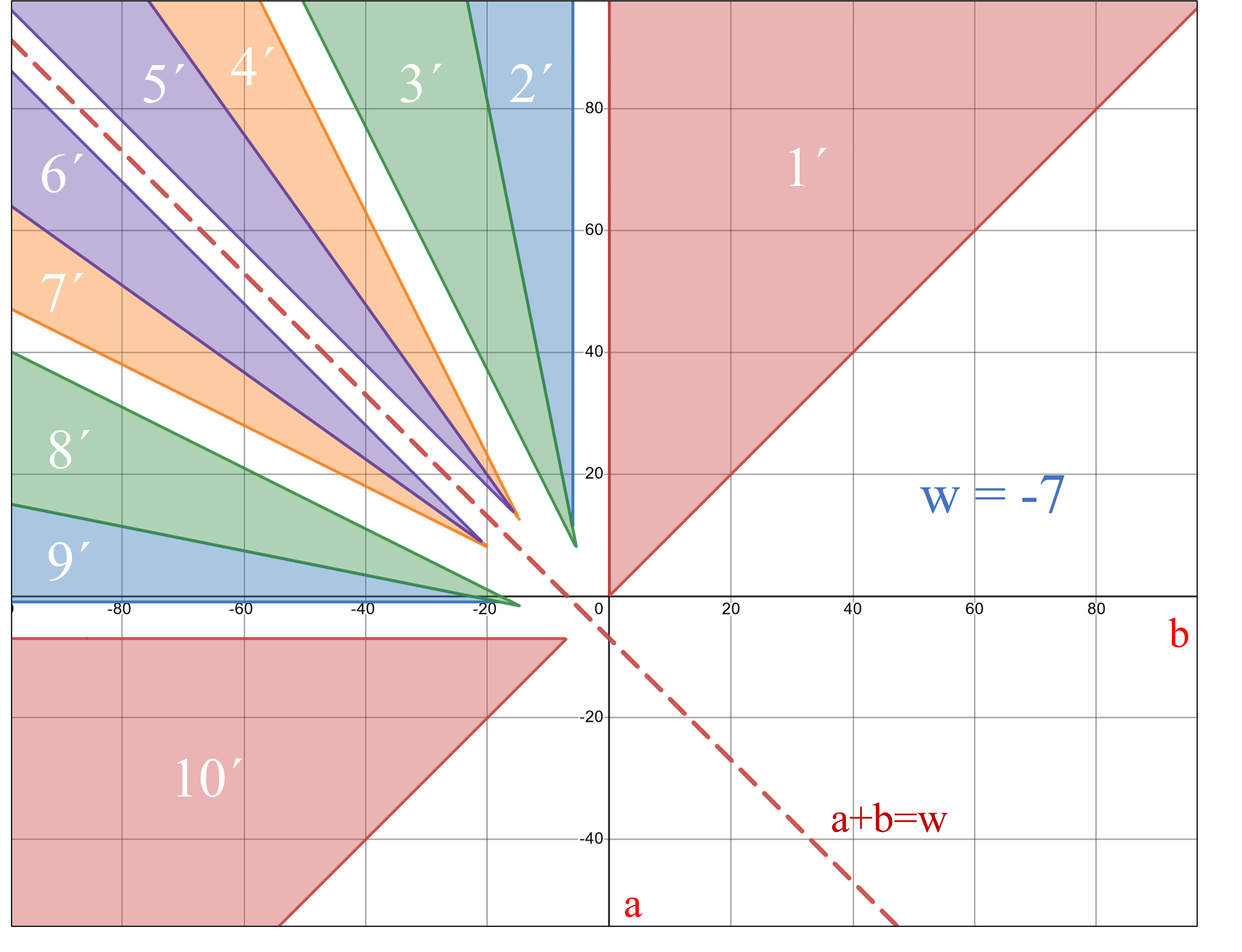}
        \includegraphics[width=.35\textwidth]{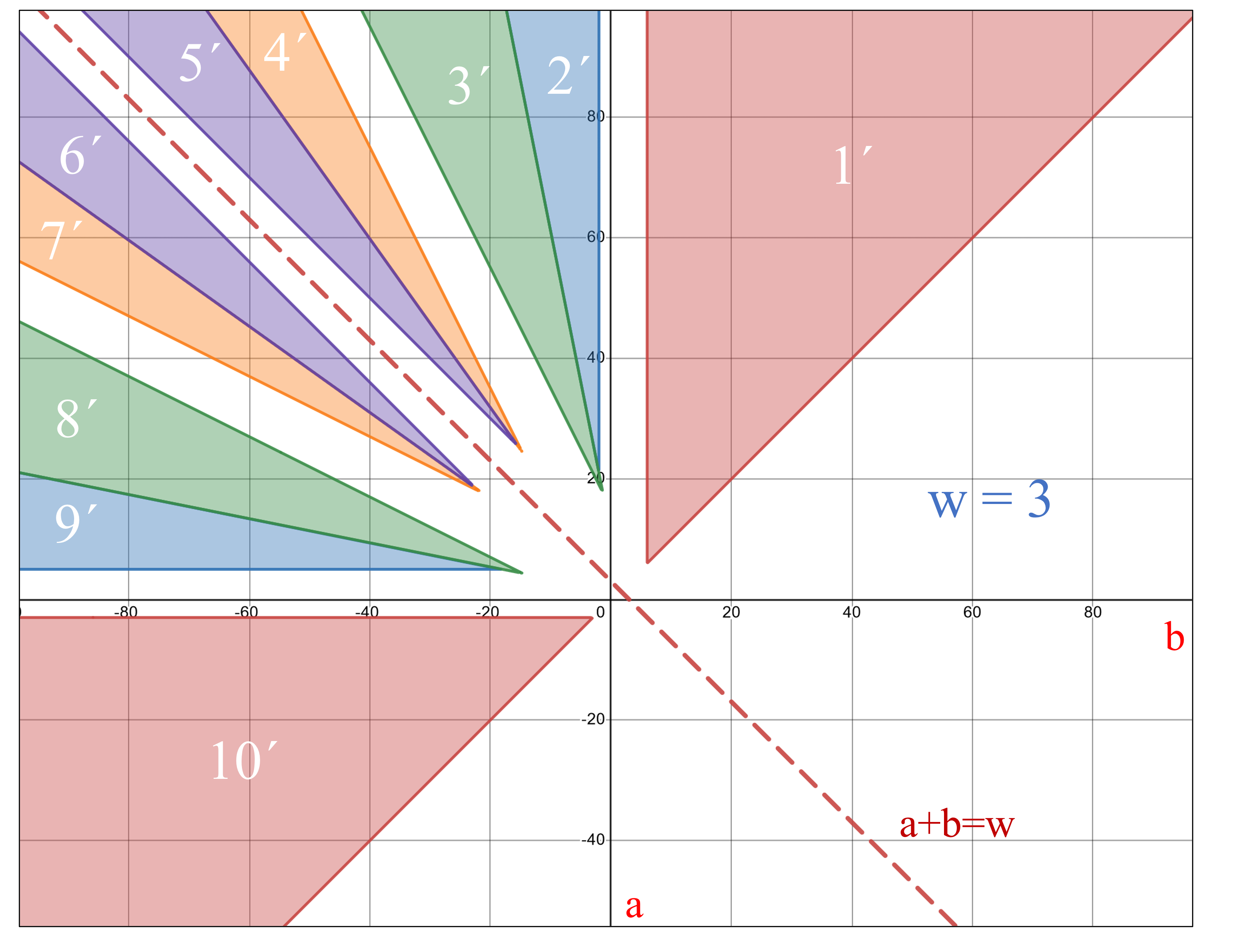}
    \caption{Critical Regions for the case $P_{\alpha}$ at different $\mathbf{w}$.}
    \label{fig:critregalpha}
\end{figure}
\begin{remark}
    The analogue of the condition discussed in Remark \ref{rmk: wrunit}, is $\mathbf{w}\leq -3$ and $\mathbf{w}>-3$ in this case.  Moreover, the critical set in $TR$-case is trivial.
\end{remark}

\section{The Combinatorial lemma: The case $G_2$}\label{sec:comb}
In this section, we state and prove the combinatorial lemmas, Lemma \ref{lem:combbeta} and Lemma \ref{lem:combalpha}. These lemmas generalize the combinatorial lemma introduced by Harder and Raghuram \cite[Lemma 7.14]{HR} to the case of $G_2$.
\subsection{The case $P_{\beta}$}
In this section, we form a combinatorial assumption on weight $\mu = (a^{\eta},b^{\eta})_{\eta:F\to \mathbb{C}}$ of the Levi factor $\mathbf{GL_2} \simeq M_{\beta}$ to find a relation between critical points assigned to weight as previous section and boundary cohomology of the ambient group.

\begin{defn}
        A Kostant representative $w = (w^{\eta})\in W^{P_{\beta}}$ is called balanced if for any embedding $\eta: F\to \mathbb{C}$, we have:
        \[l(w^{\eta})+l(w^{\bar{\eta}}) = dim(U_{P_{\beta}}/F).\]
\end{defn}
\begin{remark} \label{rem:balanced}
         In the case $P_{\beta}$,  being balanced means that the Kostant representative $(w^{\eta},w^{\bar{\eta}})$ mentioned in part (3) of the combinatorial lemma is one of the Kostant representatives $(w_{\alpha\beta\alpha\beta\alpha}^{\eta},1^{\bar{\eta}}),\;$ $(1^{\eta},\; w_{\alpha\beta\alpha\beta\alpha}^{\bar{\eta}}),\; (w_{\alpha\beta\alpha\beta}^{\eta},w_{\alpha}^{\bar{\eta}}), \;(w_{\alpha}^{\eta},w_{\alpha\beta\alpha\beta}^{\bar{\eta}}), \;(w_{\alpha\beta\alpha}^{\eta},w_{\alpha\beta}^{\bar{\eta}}),\;$ or $(w_{\alpha\beta}^{\eta},w_{\alpha\beta\alpha}^{\bar{\eta}})$. It assures that the degree of the cohomology in hand stays in an interval where the inner cohomology groups are non-vanishing.
\end{remark}

\begin{lemma}[Combinatorial Lemma for $P_{\beta}$] \label{lem:combbeta}
    Using the notation introduced in the previous section, the following statements are equivalent:
    \begin{enumerate}
        \item The evaluation point $s=-\frac{5}{2}$ is a critical point for both $L(s,Ad^{3}(\sigma))L(2s,\omega_{\sigma})$ and $L(1+s,Ad^{3}(\sigma))L(1+2s,\omega_{\sigma})$.
        \item The abelian width is bounded in terms of the cuspidal widths as follows:
        \begin{align*}
            7-\ell_1(\mu)&\leq \mathbf{w} \leq 3+\ell_1(\mu),\\
            6-\ell_2(\mu) &\leq \mathbf{w} \leq 4+\ell_2(\mu),\\
            \ell_1(\mu)\geq 2,&\qquad \ell_2(\mu)\geq 1.
        \end{align*}
        
        \item There exists a balanced Kostant representative $w = (w^{\eta})\in W^{P_{\beta}}$ such that $w^{-1}.\mu$ is a dominant weight of $G$.
        \end{enumerate}
    \end{lemma}
        
We have already proven $(1) \iff (2)$ in Section \ref{sec:criticalb}. To address part $(3)$, we need to interpret $\mu = (a^\eta, b^\eta)_{\eta: F \to \mathbb{C}}$ as a weight of $M_\beta \simeq \mathbf{GL_2}$ within $G$, using the parametrization $t_\beta$ of $\mathbf{GL_2}$ defined in Section \ref{sec:Preliminaries}, adjusted for $P_\beta$. For the rest of this section, fix an embedding $\eta$ and the notations $\mu := \mu^{\eta}$, $\;(a,b) = (a^{\eta},b^{\eta})$, $\;\mu^{*}= \mu^{\bar{\eta}}$, $\;(a^* ,b^* )= (a^{\bar{\eta}}, \;a^{\bar{\eta}})$, $\;w^{\eta} = w$, and $w^{\bar{\eta}} = w^{*}$.

Using $t_{\beta}: M_{\beta} \simeq \mathbf{GL_2}$ for a strongly pure weight $(\mu,\mu^{*})$ of $\mathbf{GL_2}$, we have: 
\begin{align*}
    \mu &= a(\alpha+\beta)+b(\alpha) \hspace{57pt}= (2b-a)\gamma_{s}+(a-b)\gamma_{l},\\
    \mu^{*} &= (\mathbf{w}-b)(\alpha+\beta)+(\mathbf{w}-a)(\alpha) = (\mathbf{w}+b-2a)\gamma_{s}+(a-b)\gamma_{l}.
\end{align*}
We can check the integrality and dominance of the same weight as a weight of $G$.

\textbf{Integrality :}
\begin{align*}
    a,b,\mathbf{w} \in \mathbb{Z} \iff (2b-a),(a-b),(\mathbf{w}+b-2a)\in \mathbb{Z},
\end{align*}
which means, $(\mu,\mu^{*})$ is an integral weight of $\mathbf{GL_2} \simeq M_{\beta}$ if and only if it corresponds to an integral weight of $G$.

\textbf{Dominance :}
\begin{itemize}
    \item $(\mu,\mu^{*})$ is a dominant weight of $\mathbf{GL_2} \simeq M_{\beta}$ if 
    \[
        b\leq a.
    \]
    \item $(\mu,\mu^{*})$ corresponds to a dominant weight of $G$ if
    \[
        \begin{cases}
            b\leq a \leq 2b;\\
            b^*\leq a^* \leq 2b^* .
        \end{cases}
    \]
\end{itemize}
Hence, $(\mu,\mu^{*})$ is a dominant weight of $M_{\beta}$ if it corresponds to a dominant weight of $G$.

\begin{table}[ht]
    \caption{The Twisted Action of Kostant Representatives in $W^{P_{\beta}}$}
    \vspace*{6pt}
    \centering
     \begin{tabular}{||c c c c||} 
     \hline\rule{0pt}{4ex}
     $l(w)$ & $w\in W^{P_{\beta}}$ & $w^{-1}.(a,b)$ & $(w^{*})^{-1}.(\mathbf{w}-b,\mathbf{w}-a)$\\
     \hline\hline
     $0$ & $1$ & $(a,b)$ & $(\mathbf{w}-b,\mathbf{w}-a)$\\ 
     \hline\rule{0pt}{4ex}
     $1$ & $w_{\alpha}$ &  $(a,a-b-1)$ & $(\mathbf{w}-b,a-b-1)$\\
     \hline\rule{0pt}{4ex}
     $2$ & $w_{\alpha\beta}$&  $(a-b-2,a+1)$ & $(a-b-2,\mathbf{w}-b+1)$\\ 
     \hline\rule{0pt}{4ex}
     $3$ & $w_{\alpha\beta\alpha}$ & $(a-b-2,-b-4)$ & $(a-b-2,a-\mathbf{w}-4)$\\
     \hline\rule{0pt}{4ex}
     $4$ &  $w_{\alpha\beta\alpha\beta}$ & $(-b-5,a-b-1)$ & $(a-\mathbf{w}-5,a-b-1)$\\ 
     \hline\rule{0pt}{4ex}
     $5$ & $w_{\alpha\beta\alpha\beta\alpha}$ & $(-b-5,-a-5)$ & $(a-\mathbf{w}-5,b-\mathbf{w}-5)$\\
    \hline
    \end{tabular}
    \label{table: twistedbeta}
\end{table}

\vspace{\baselineskip}

The twisted action of Kostant representatives on weights $\mu = a(\alpha+\beta) + b\alpha$, written in the form $(a,b) := a(\alpha+\beta) + b\alpha$, is summarized in Table \ref{table: twistedbeta}. According to part $(3)$ of the combinatorial lemma, there exists a balanced Kostant representative $w$ such that its twisted action on the strongly inner weight $\mu$ produces a dominant weight of $G$. The first step is to classify all strongly inner weights that satisfy condition $(3)$ for each balanced Kostant representative.

\begin{itemize}
    \item[(I)]$(1,w_{\alpha\beta\alpha\beta\alpha})$:
    \[
    \begin{cases}
        b\leq a \leq 2b;\\
        b\leq a\leq 2b- \mathbf{w}-5.
    \end{cases}
    \]
    We can split it into cases by varying $\mathbf{w}$:
    \[
        \begin{cases}
            b\leq a \leq 2b, & \mathbf{w}\leq -5;\\
            b\leq a \leq 2b-\mathbf{w}-5,& \mathbf{w}\geq -5.
        \end{cases}
    \]

    \item[(II)]$(w_{\alpha},w_{\alpha\beta\alpha\beta})$:
    \[
    \begin{cases}
        -1\leq b  \leq a-b-2;\\
        4+\mathbf{w}\leq b \leq a-b+3+\mathbf{w}.
    \end{cases} \implies \begin{cases}
            -1\leq b \leq a-b+3+\mathbf{w},& \mathbf{w}\leq -5;\\
            4+w \leq b \leq a-b-2,& \mathbf{w}\geq -5.
        \end{cases}
    \]

    \item[(III)]$(w_{\alpha\beta},w_{\alpha\beta\alpha})$:
    \[
    \begin{cases}
        -4-a\leq b \leq -3;\\
        2\mathbf{w}-a+6 \leq b \leq 2+\mathbf{w}.
    \end{cases} \implies
        \begin{cases}
            -4-a\leq b \leq 2+\mathbf{w}, &\mathbf{w}\leq -5;\\
            2\mathbf{w}-a+6 \leq b \leq -3, & \mathbf{w}\geq -5.
        \end{cases}
    \]
    \item[(IV)]$(w_{\alpha\beta\alpha},w_{\alpha\beta})$:
     \[
    \begin{cases}
        -2\leq a \leq -b-6;\\
        \mathbf{w}+3 \leq a \leq 2\mathbf{w}-b+4.
    \end{cases} \implies         \begin{cases}
            -2\leq a \leq 2\mathbf{w}-b+4,& \mathbf{w}\leq -5;\\
            \mathbf{w}+3\leq a \leq -b-6,&\mathbf{w}\geq -5.
        \end{cases}
    \]
    \item[(V)]$(w_{\alpha\beta\alpha\beta},w_{\alpha})$:
    \[
    \begin{cases}
        a+b+4\leq b \leq 2a+3;\\
        a+b-\mathbf{w}-1\leq b \leq 2a-\mathbf{w}-2.
    \end{cases} \implies
     \begin{cases}
            a+b-\mathbf{w}-1 \leq b \leq 2a+3,& \mathbf{w}\leq -5;\\
            a+b+4\leq b \leq 2a-\mathbf{w}-2,& \mathbf{w} \geq -5.
        \end{cases}
    \]
    \item[(VI)]$(w_{\alpha\beta\alpha\beta\alpha},1)$:
    \[
    \begin{cases}
        2a+5\leq b \leq a;\\
        2a-\mathbf{w}\leq b \leq a.
    \end{cases} \implies
        \begin{cases}
            2a-\mathbf{w} \leq b \leq a, &\mathbf{w} \leq -5;\\
            2a+5 \leq b \leq a,& \mathbf{w}\geq -5.
        \end{cases}
    \]

\begin{figure}[htp]
    \centering
        \includegraphics[width=.35\textwidth]{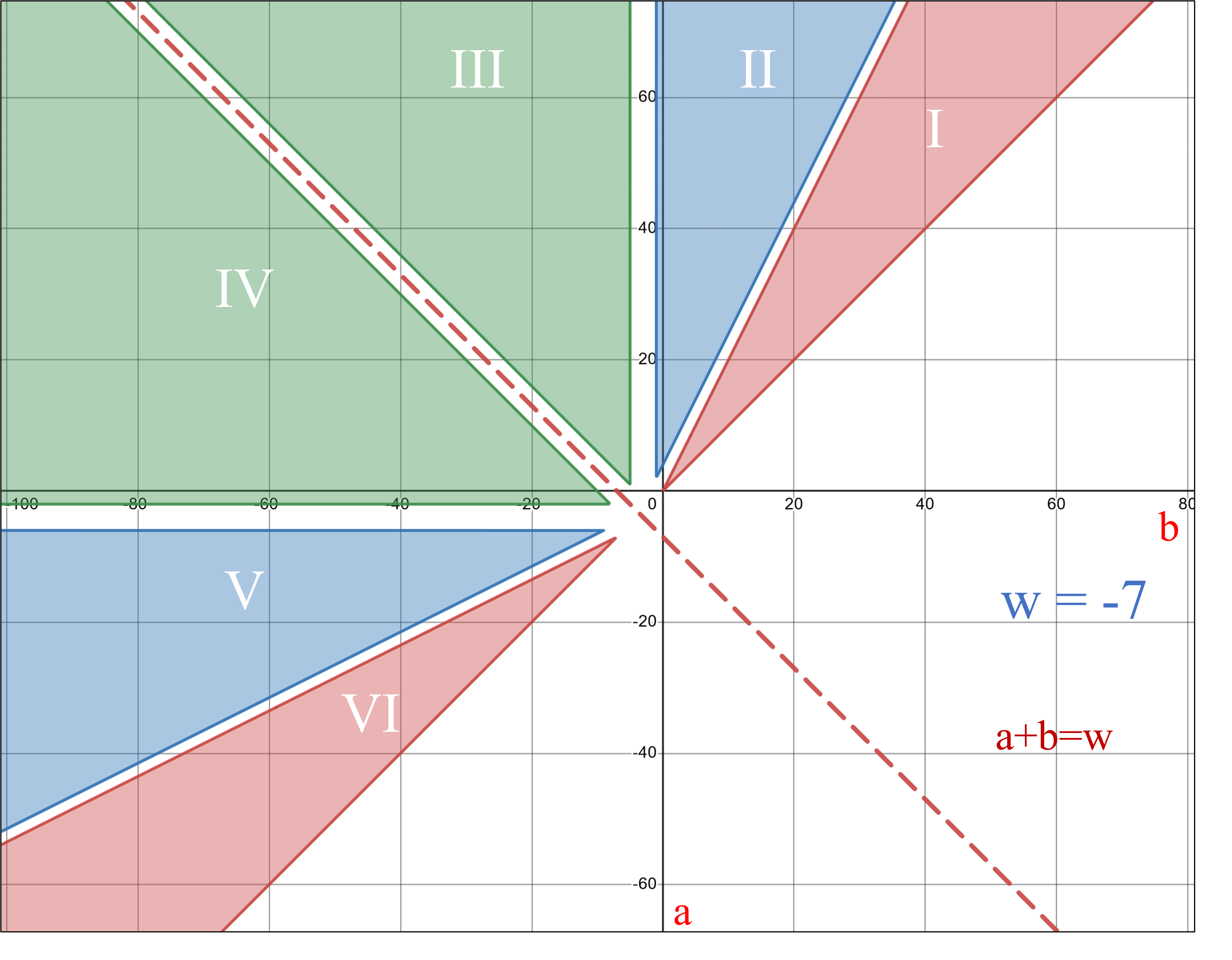}
        \includegraphics[width=.35\textwidth]{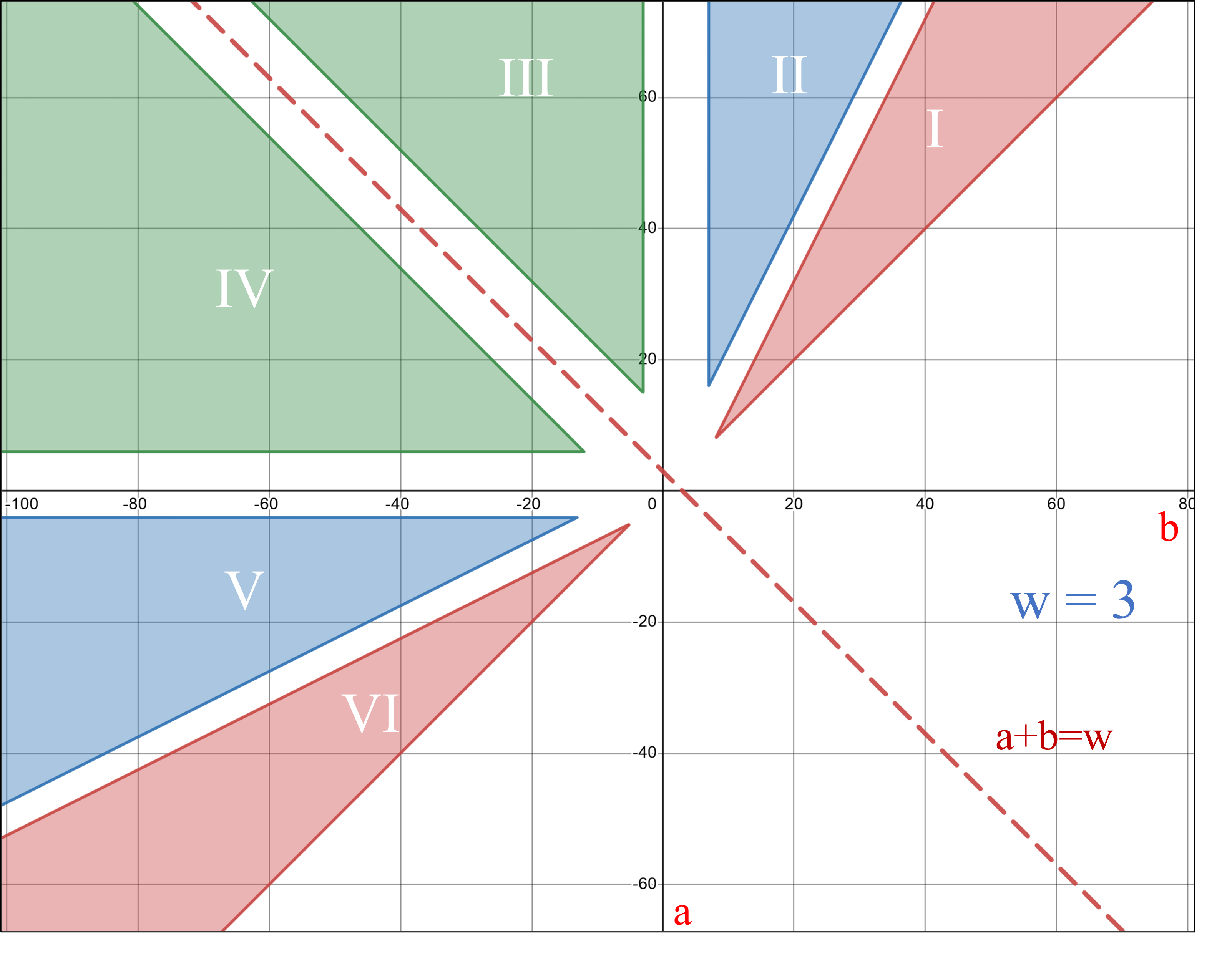}
    \caption{Twisted Action Regions of $P_{\beta}$ at different $\mathbf{w}$.}
    \label{fig:twistedregbeta}
\end{figure}
\end{itemize}

\begin{figure}[htp]
    \centering
        \includegraphics[width=.35\textwidth]{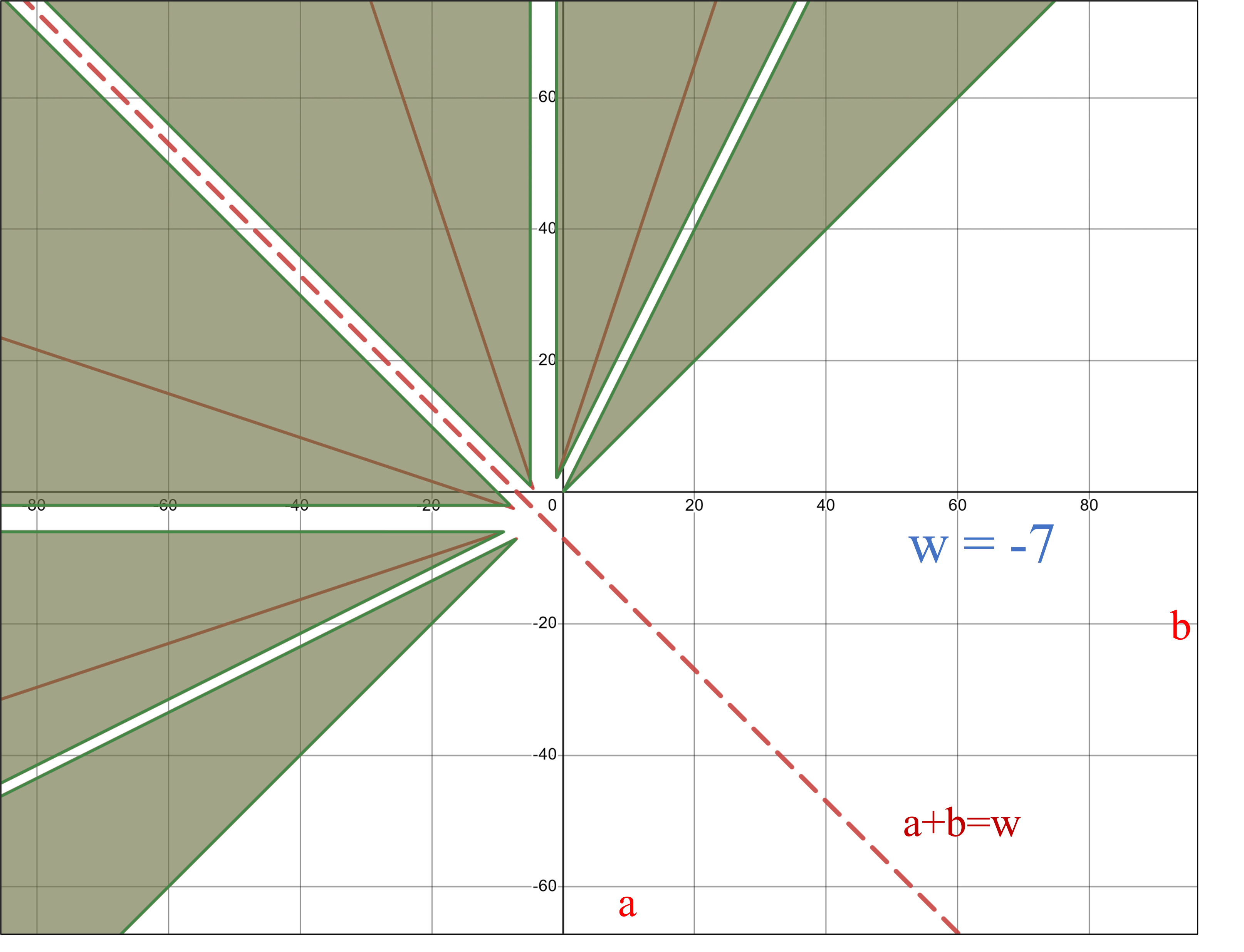}
        \includegraphics[width=.35\textwidth]{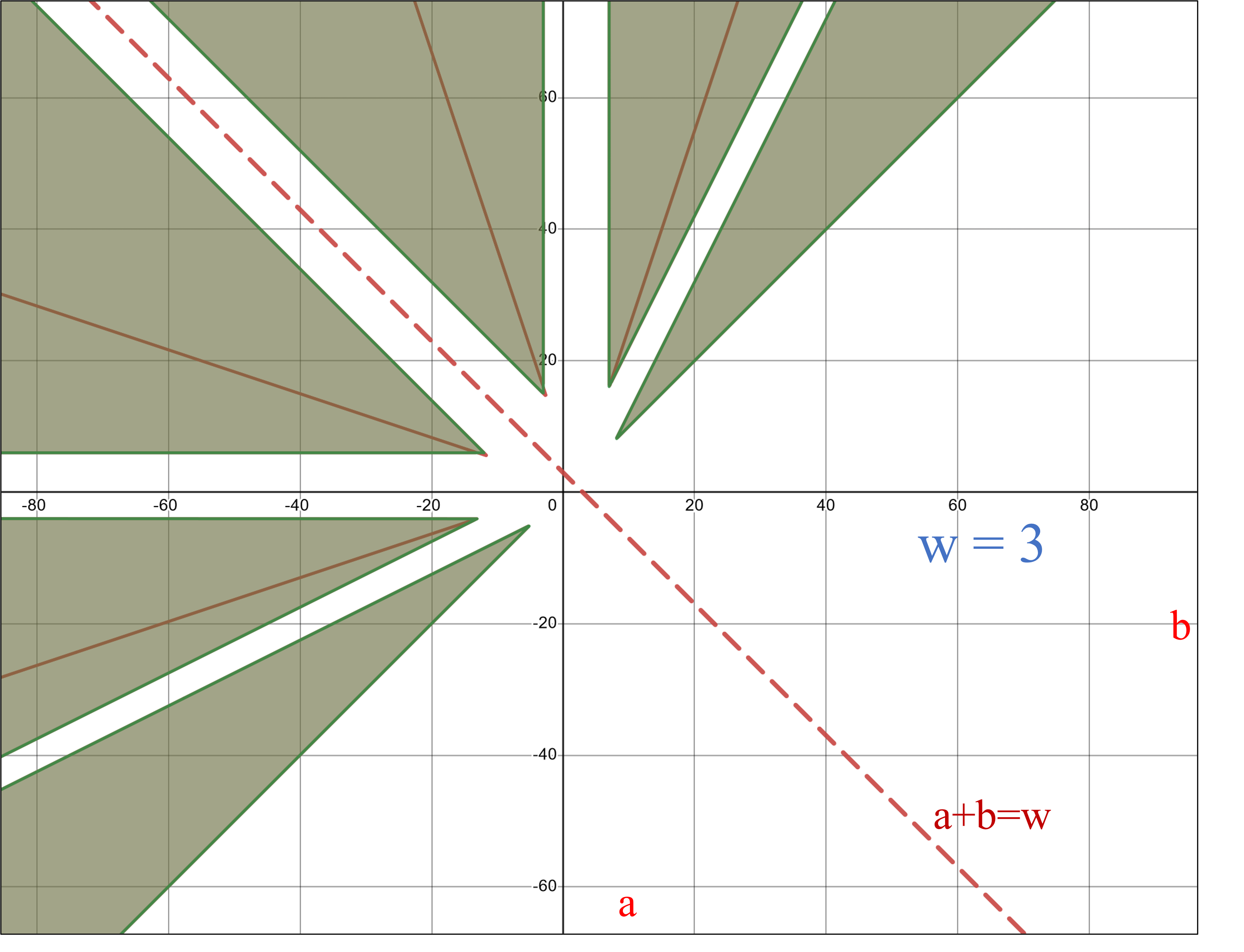}
    \caption{Critical Region in Red and Twisted Action Region in Green}.
    \label{fig:intersectbeta}
\end{figure}

We refer to the regions described above as twisted action regions, as illustrated in Figure \ref{fig:twistedregbeta}. By comparing the outer boundaries and the vertex of each twisted action region, it is evident that each critical region intersects with only one twisted action region, and each twisted action region contains at most two critical regions.

The critical region (5) and (6) are not fully covered with a twisted action region. We claim that this small region does not contain any lattice point of $\mathbb{Z}^{2}$. First, we have:
\[
    v_{(III)} = \begin{cases}
        (2+\mathbf{w},-6-\mathbf{w}), & \mathbf{w}\leq -5,\\
        (-3,2\mathbf{w}+9), &\mathbf{w}\geq -5.
    \end{cases}
\]
For $\mathbf{w}\leq -5$, translate vertices by adding the vector $(-2-\mathbf{w},6+\mathbf{w})$. As the translation vector has integer arguments, If the triangle formed by $v_{(4)}$, $v_{(5)}$, and $v_{(III)}$ contains a lattice point, its translation by this vector will be a lattice point inside of the translated triangle. The translated vertices are:
\[
 \begin{cases}
    v'_{(III)} = (2+\mathbf{w},-6-\mathbf{w})+(-2-\mathbf{w},6+\mathbf{w}) = (0,0),\\
    v'_{(4)} = (2+\mathbf{w},-5-\mathbf{w})+(-2-\mathbf{w},6+\mathbf{w}) = (0,1),\\
    v'_{(5)} = (\frac{5}{2}+\mathbf{w},-\frac{13}{2}-\mathbf{w})+(-2-\mathbf{w},6+\mathbf{w}) = (\frac{1}{2},-\frac{1}{2}).
    \end{cases}
\]
We can also transform this triangle when $\mathbf{w}\geq -5$, by the vector $(3,-2\mathbf{w}-9)$ which transforms the triangle formed by $v_{(4)}$, $v_{(5)}$, and $v_{(III)}$ to the same triangle as the last part. The transformed triangle clearly doesn't have any lattice points. It would finish the proof of the claim for the critical region (5) when $\mathbf{w} \leq -5$, as the transformed vertices are independent of $a,b$, and $\mathbf{w}$.
\begin{figure}[htp]
    \centering
        \includegraphics[width=.30\textwidth]{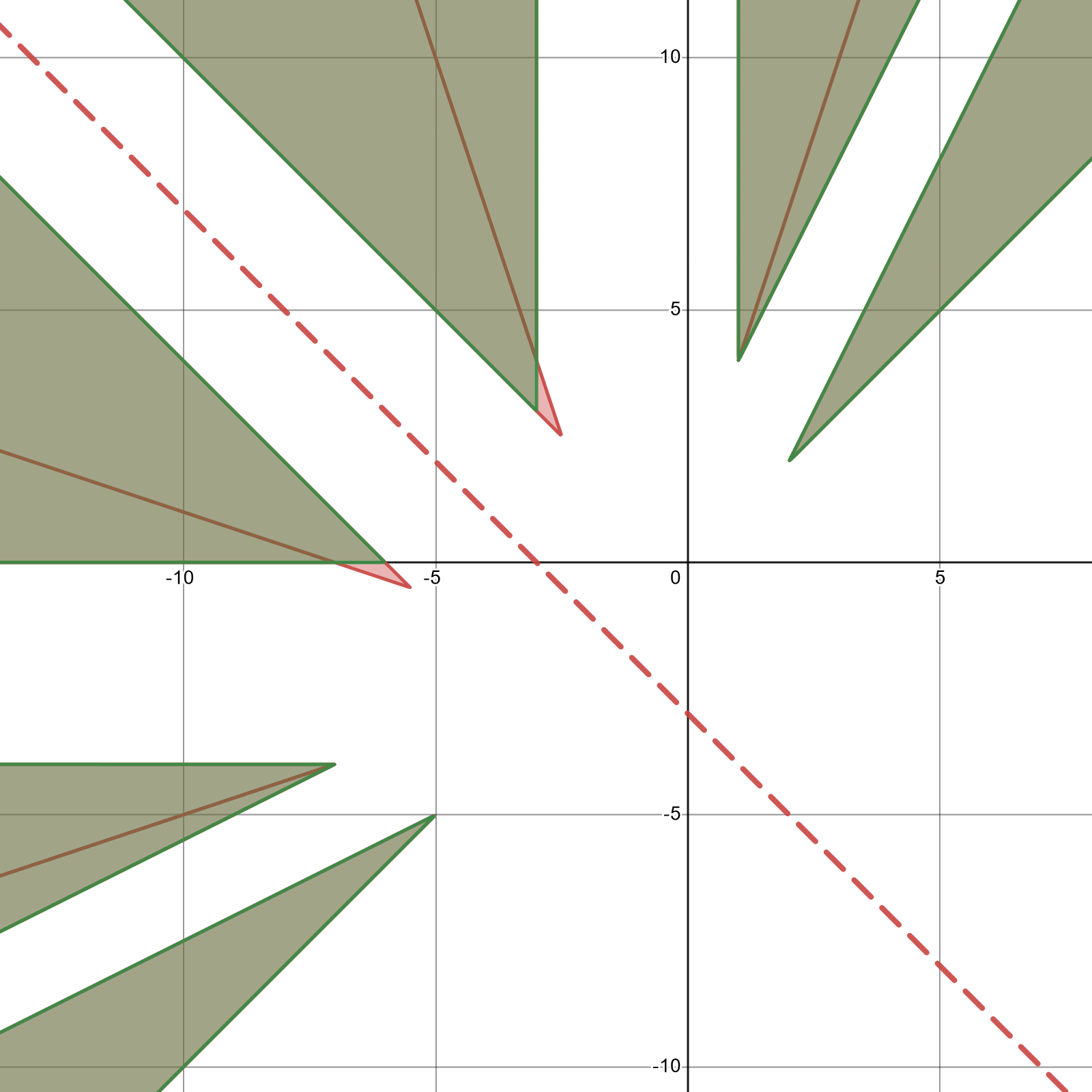}
        \includegraphics[width=.30\textwidth]{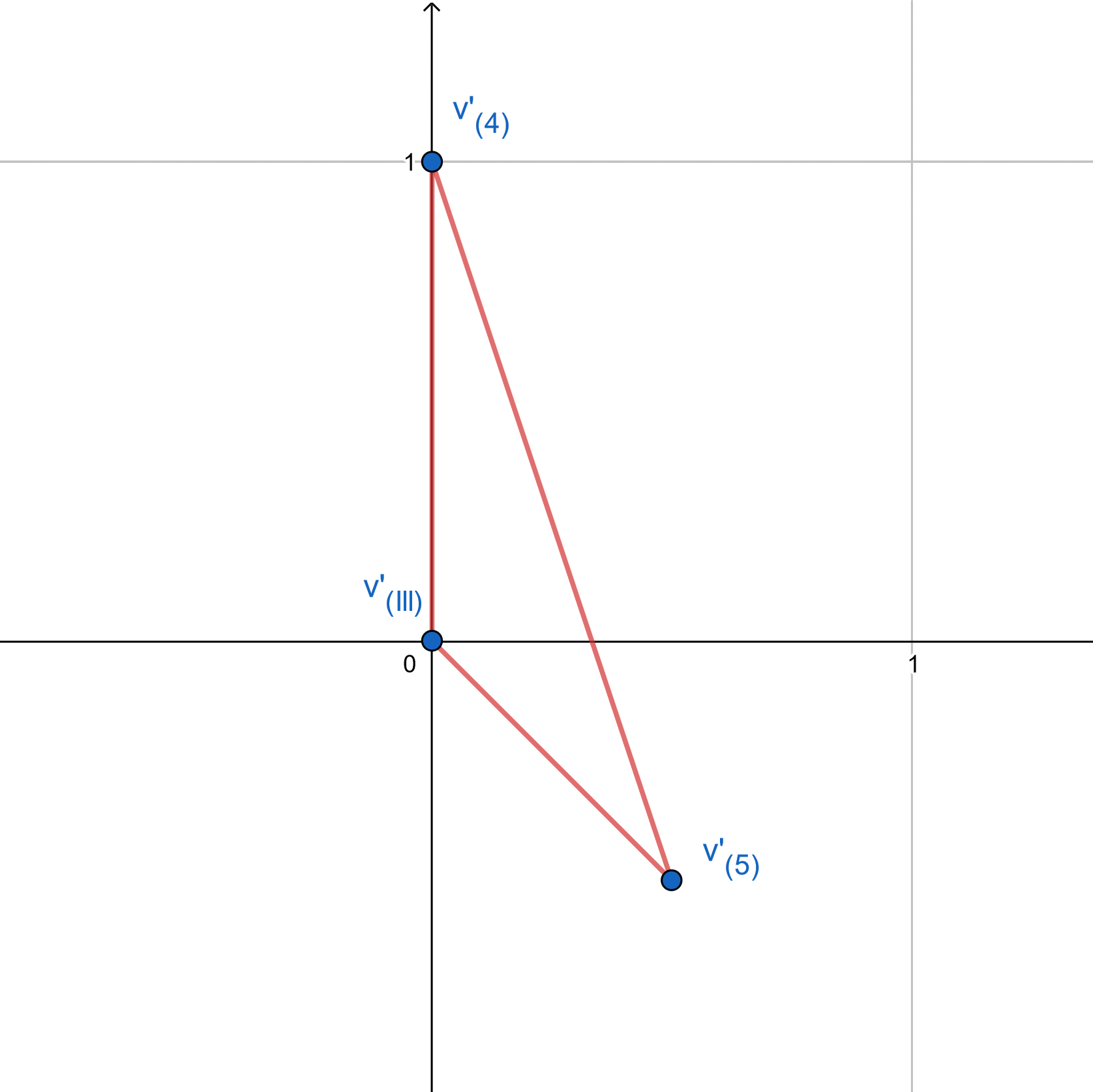}
    \caption{Uncovered Region Doesn't Contain Any Lattice Points}
    \label{fig:mesh1}
\end{figure}

We can repeat this process to show that the uncovered part of the critical region (7) does not contain any lattice point, but it is clear due to the symmetry of the regions. It will complete the proof of combinatorial lemma over imaginary quadratic number fields. 

\begin{remark}
    The proof of the combinatorial lemma for a general totally imaginary field case is reduced to working with pairs of complex embeddings corresponding to each archimedean place. So our proof implies the combinatorial lemma in the general case.
\end{remark}

\subsection{The case $P_{\alpha}$}
In this section, we state the combinatorial assumption for the weight $\mu = (a^{\eta},b^{\eta})_{\eta:F\to \mathbb{C}}$ of the Levi factor $\mathbf{GL_2} \simeq M_{\alpha}$.

\begin{lemma}[Combinatorial Lemma]\label{lem:combalpha}
     Using the notation introduced in the previous section, the following statements are equivalent:
    \begin{enumerate}
        \item The evaluation point $s=-\frac{3}{2}$ is a critical point for both 
        \[\begin{cases}
            L(s,\sigma)L(2s,\omega_{\sigma})L(3s,\sigma\otimes\omega_{\sigma}),\\
            L(1+s,\sigma)L(1+2s,\omega_{\sigma})L(1+3s,\sigma\otimes\omega_{\sigma}).   
        \end{cases}\]
        \item The abelian width is bounded in terms of the cuspidal widths as follows:
        \begin{align*}
            -1-\ell_1(\mu)&\leq \mathbf{w} \leq -5+\ell_1(\mu),\\
            -2-\ell_2(\mu) &\leq \mathbf{w} \leq -4+\ell_2(\mu),\\
            -\frac{7}{3}-\frac{\ell_3(\mu)}{3} &\leq \mathbf{w} \leq -\frac{11}{3}+\frac{\ell_3(\mu)}{3},\\
            \ell_1(\mu)\geq 2,&\qquad \ell_2(\mu)\geq 1,\qquad \ell_3(\mu)\geq 2.
        \end{align*}
        
        \item There exists a balanced Kostant representative $w = (w^{\eta})\in W^{P_{\alpha}}$ such that $w^{-1}.\mu$ is a dominant weight of $G$.
        \end{enumerate}
    \end{lemma}
    \begin{remark}
We defined a balanced Kostant representative $w = (w^{\eta})\in W^{P_{\alpha}}$ in previous section, where it satisfies $l(w^{\eta})+l(w^{\bar{\eta}}) = dim(U_{P_{\alpha}}/F)=10$. Therefore, the Kostant representative in part (3) of the combinatorial lemma is one of the $(w_{\beta\alpha\beta\alpha\beta}^{\eta},1^{\bar{\eta}})$ ,  $(1^{\eta},w_{\beta\alpha\beta\alpha\beta}^{\bar{\eta}})$ , $(w_{\beta\alpha\beta\alpha}^{\eta},w_{\beta}^{\bar{\eta}})$ , $(w_{\beta}^{\eta}, w_{\beta\alpha\beta\alpha}^{\bar{\eta}})$ , $(w_{\beta\alpha\beta}^{\eta},w_{\beta\alpha}^{\bar{\eta}})$ , or $(w_{\beta\alpha}^{\eta},w_{\beta\alpha\beta}^{\bar{\eta}})$.
\end{remark}

\begin{table}[ht]\label{tab:twistalpha}
    \caption{The Twisted Action of Kostant Representatives in $W^{P_{\alpha}}$}
    \vspace*{6pt}
    \centering
     \begin{tabular}{||c c c c||} 
     \hline\rule{0pt}{4ex}
     $l(w)$ & $w\in W^{P_{\alpha}}$ & $w^{-1}.(a,b)$ & $(w^{*})^{-1}.(\mathbf{w}-b,\mathbf{w}-a)$\\
     \hline\hline
     $0$ & $1$ & $(a,b)$ & $(\mathbf{w}-b,\mathbf{w}-a)$\\ 
     \hline\rule{0pt}{4ex}
     $1$ & $w_{\beta}$ &  $(a+b+1,-b-2)$ & $(2\mathbf{w}-b-a+1,-\mathbf{w}+b-1)$\\
     \hline\rule{0pt}{4ex}
     $2$ & $w_{\beta\alpha}$&  $(a-b-1,b-1)$ & $(a-b-1,\mathbf{w}-b-1)$\\ 
     \hline\rule{0pt}{4ex}
     $3$ & $w_{\beta\alpha\beta}$ & $(a,b-a-2)$ & $(\mathbf{w}-b,b-a-2)$\\
     \hline\rule{0pt}{4ex}
     $4$ &  $w_{\beta\alpha\beta\alpha}$ & $(-b-a-5,a+1)$ & $(a+b-2\mathbf{w}-5,\mathbf{w}-b+1)$\\ 
     \hline\rule{0pt}{4ex}
     $5$ & $w_{\beta\alpha\beta\alpha\beta}$ & $(-b-3,-a-3)$ & $(a-\mathbf{w}-3,b-\mathbf{w}-3)$\\
    \hline
    \end{tabular}
\end{table}
\vspace{\baselineskip}
   Following the previous section, assume $F$ is a quadratic imaginary field. Let $(\mu,\mu^{*})$ be a strongly pure weight of $t_{\alpha}: M_{\alpha}\simeq GL_2$, we have: 
\begin{align*}
    \mu &= a(2\alpha+\beta)+b(\alpha+\beta) \hspace{57pt}= (b-a)\gamma_{s}+(b)\gamma_{l},\\
    \mu^{*} &= (\mathbf{w}-b)(\alpha+\beta)+(\mathbf{w}-a)(\alpha) = (b-a)\gamma_{s}+(\mathbf{w}-a)\gamma_{l}.
\end{align*}

\textbf{Integrality :}
\begin{align*}
    a,b,\mathbf{w} \in \mathbb{Z} \iff (b-a),(\mathbf{w}-a)\in \mathbb{Z},
\end{align*}

\textbf{Dominance :}
\begin{itemize}
    \item $(\mu,\mu^{*})$ is a dominant weight of $M_{\alpha}$ if 
    \[
         b\leq a.
    \]
    \item $(\mu,\mu^{*})$ corresponds to a dominant weight of $G$ if
    \[
            0\leq b\leq a .
    \]
\end{itemize}

Table \ref{tab:twistalpha} presents the twisted action of the Kostant representatives for the weights $\mu = a(2\alpha + \beta) + b(\alpha + \beta)$, expressed in the form $(a, b) := a(2\alpha + \beta) + b(\alpha + \beta)$.
Using the same technique in the previous section for the case of $P_{\beta}$, we can show that the critical regions and twisted action regions for $P_{\beta}$ cover the same integral lattice points in the weight plane.

\begin{figure}[htp]
    \centering
        \includegraphics[width=.35\textwidth]{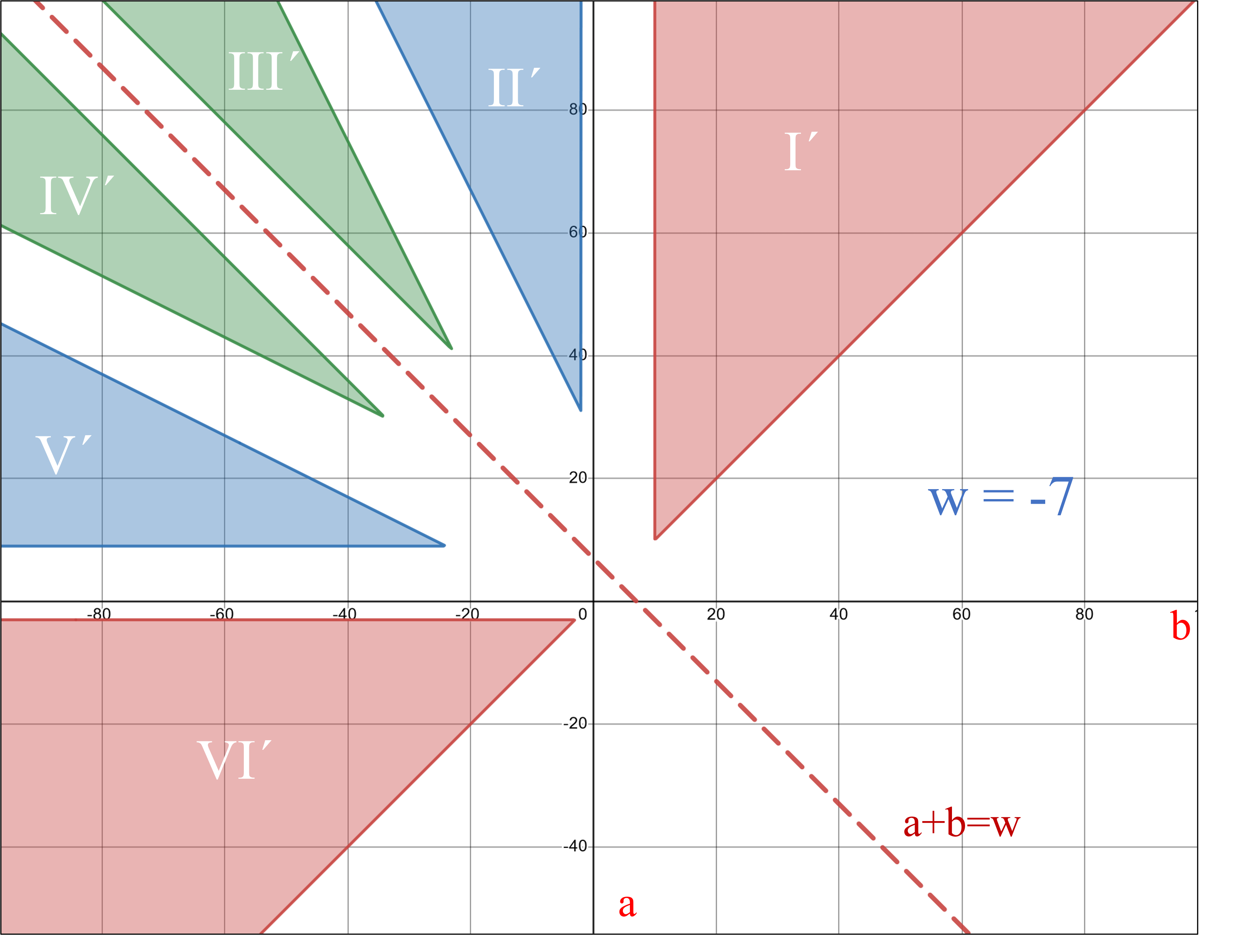}
        \includegraphics[width=.35\textwidth]{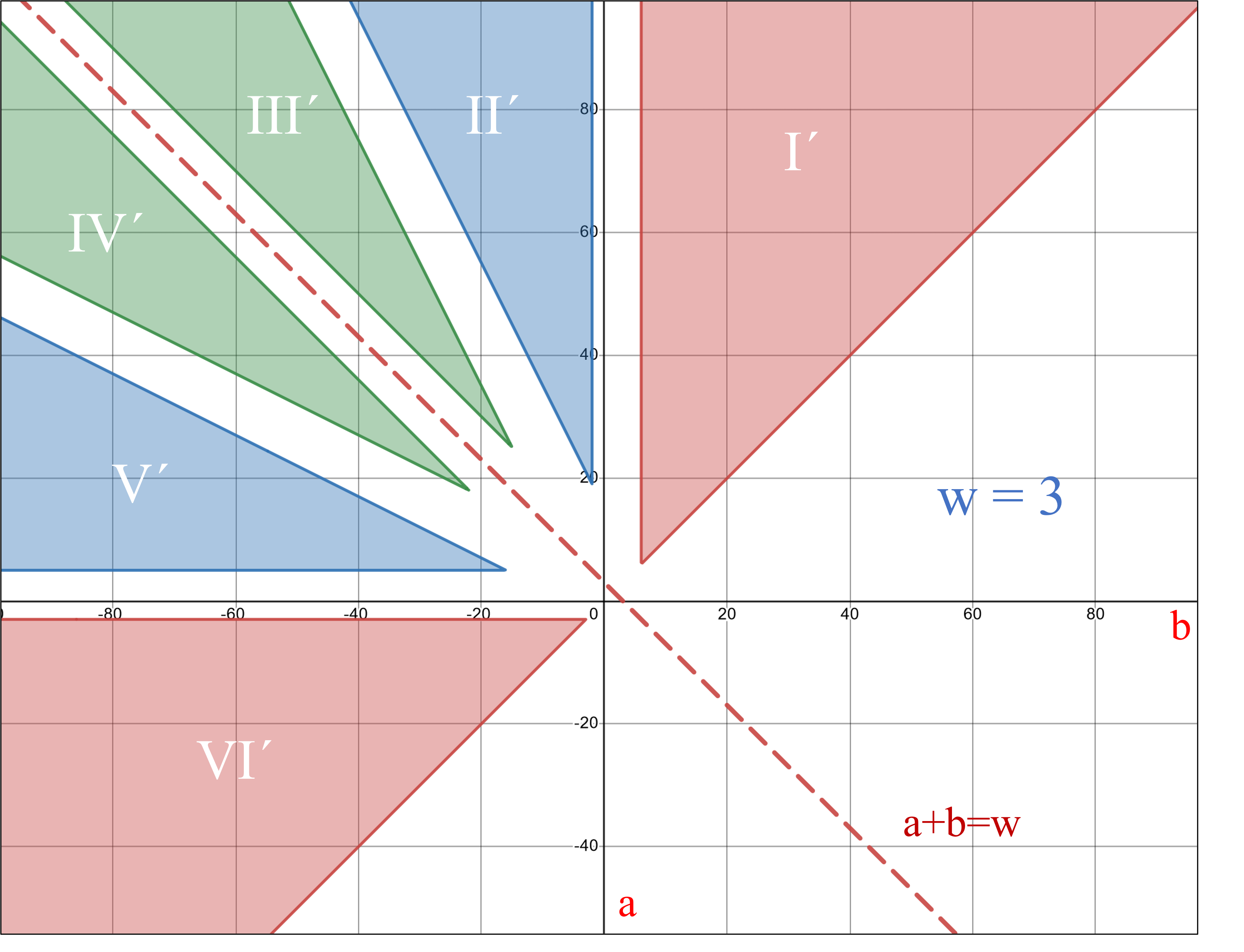}
    \caption{Twisted Action Regions of $P_{\alpha}$ at different $\mathbf{w}$.}
    \label{fig:twistedregalpha}
\end{figure}

\begin{figure}[htp]
    \centering
        \includegraphics[width=.35\textwidth]{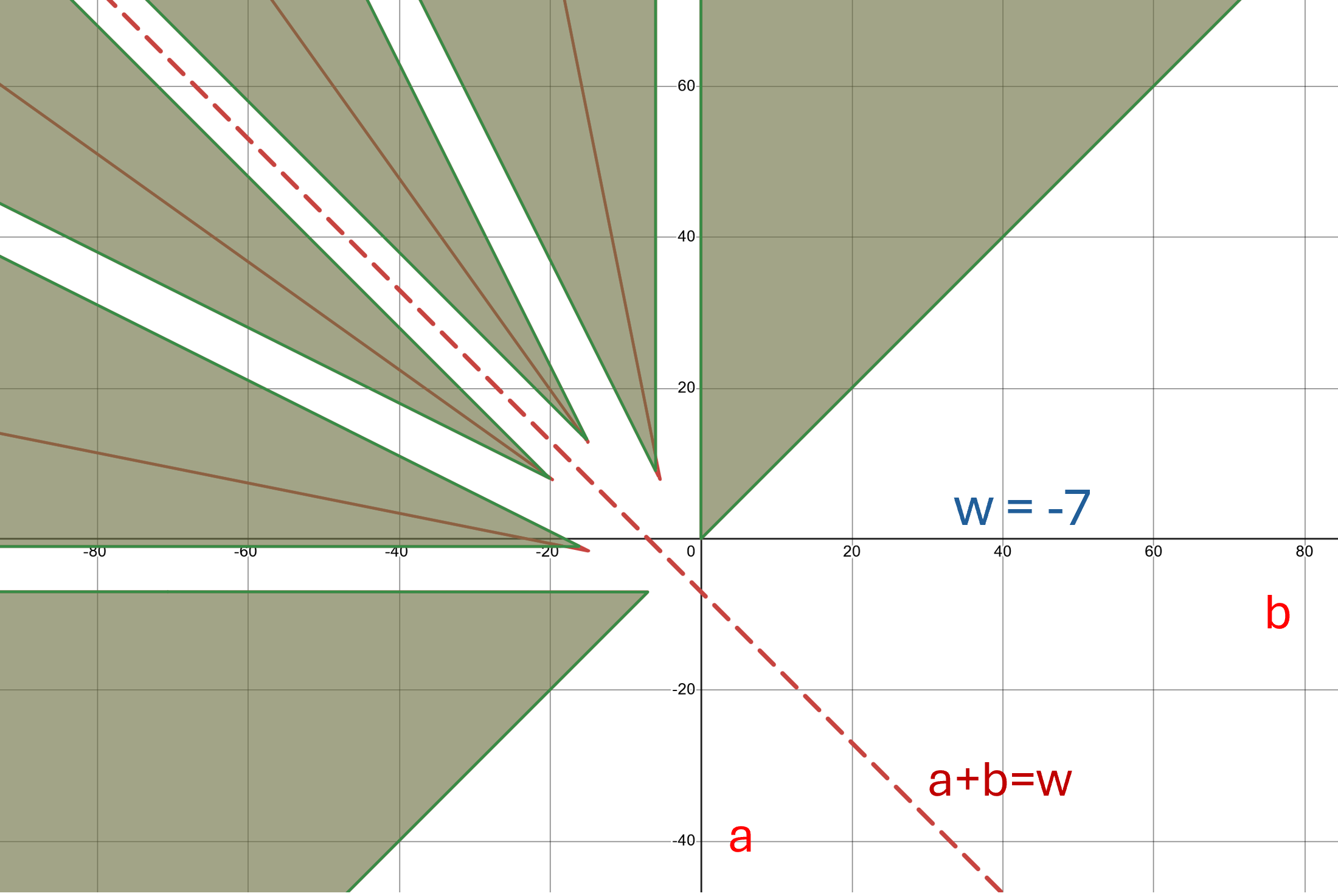}
        \includegraphics[width=.35\textwidth]{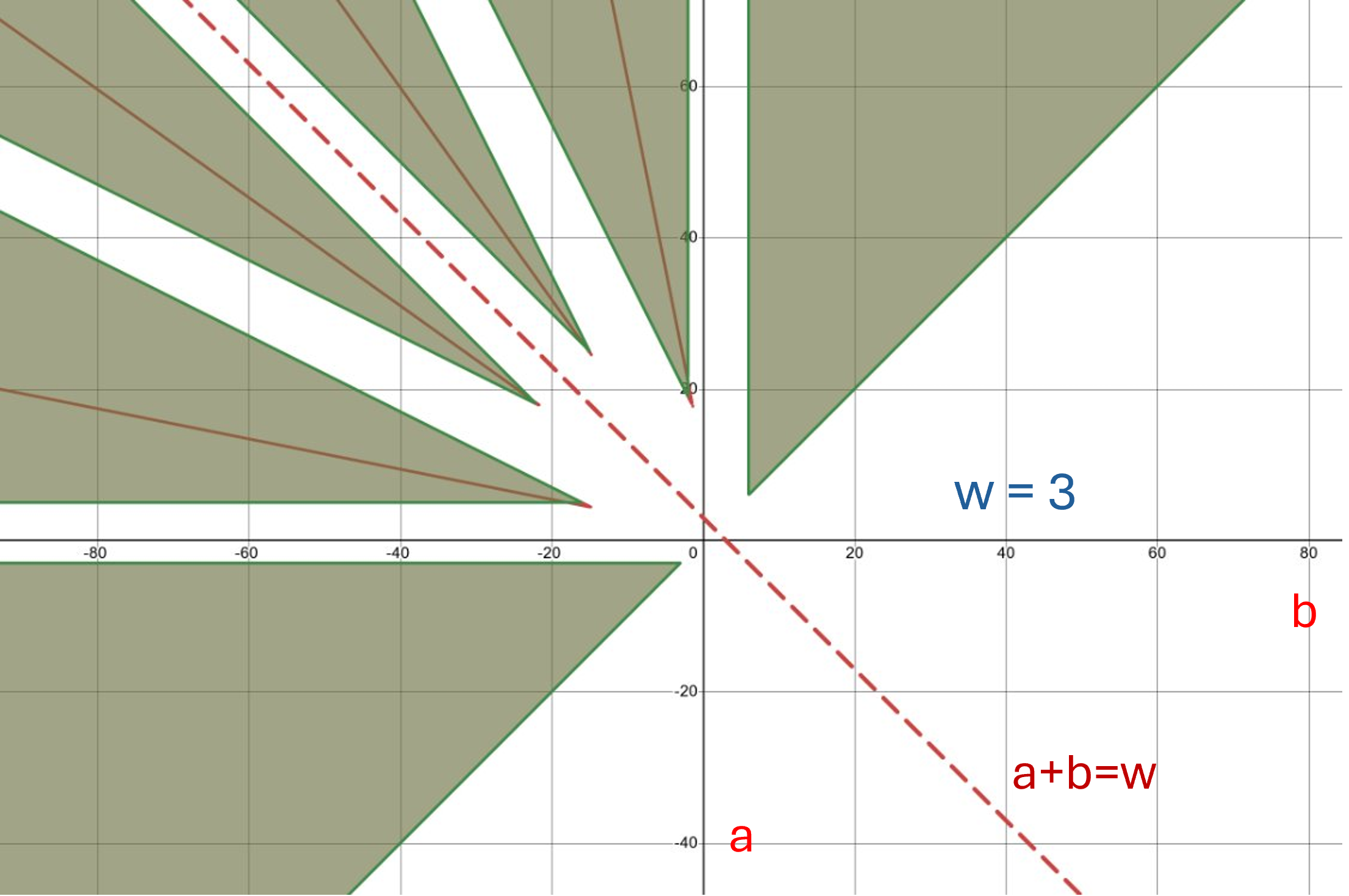}
    \caption{Critical Region in Red and Twisted Action Region in Green}.
    \label{fig:intersectregalpha}
\end{figure}

\begin{figure}[htp]
    \centering
        \includegraphics[width=.45\textwidth]{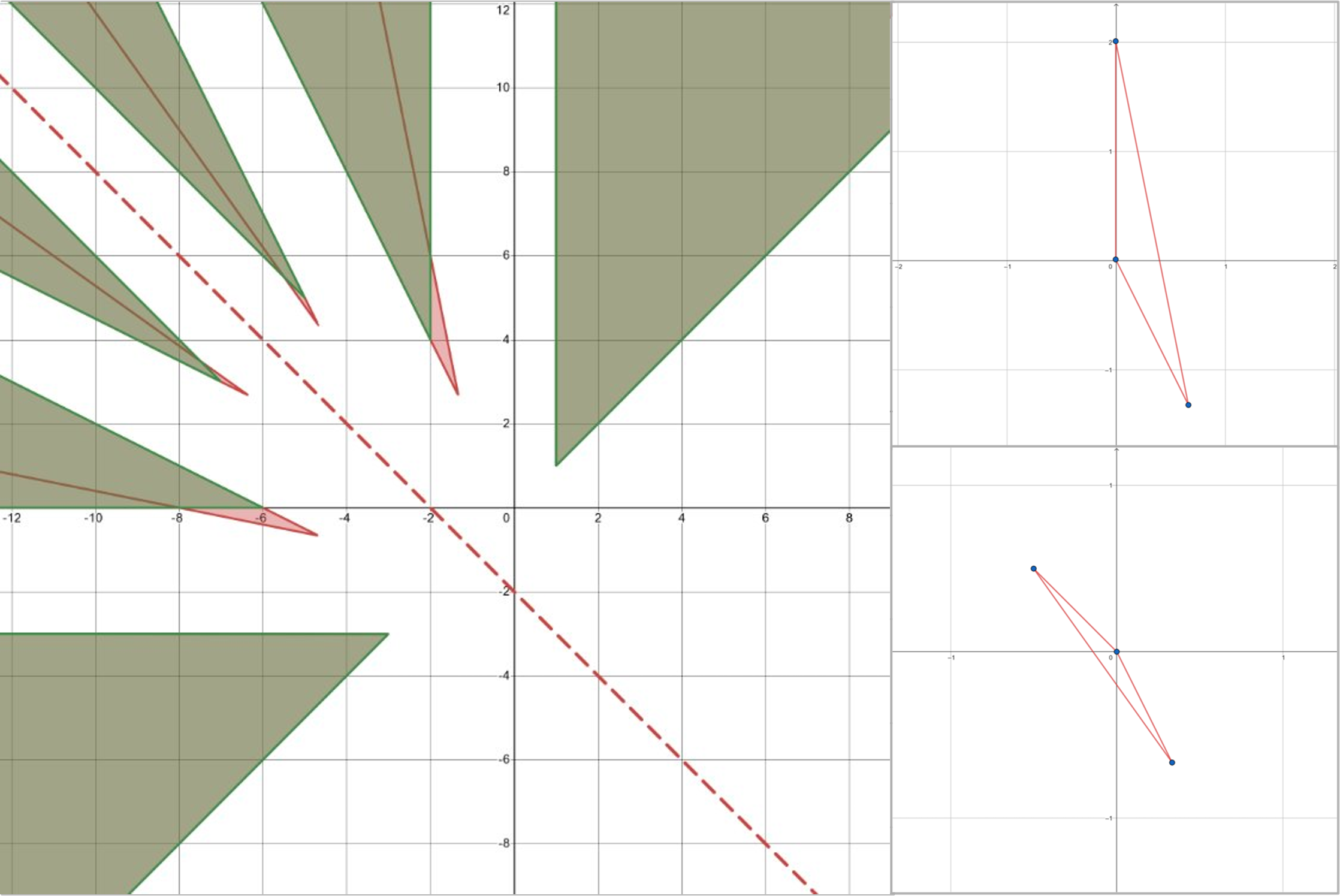}
    \caption{Uncovered Region contains no lattice points}
    \label{fig:mesh2}
\end{figure}

\subsection{The Tate twist revisited}\label{sec:ttr}
In Section \ref{sec:tt}, we introduced the notion of the Tate twist. Now we want to study the effect of Tate twists on the arguments in this section. Consider the notations introduced in the previous section. For $m\in \mathbb{Z}$, the Tate twist of $\mu \in X_{00}^{*}(T_2\times E)$ with purity weight $\mathbf{w}$ is $\mu-m\delta_2 \in X_{00}^{*}(T_2\times E)$ with purity weight $\mathbf{w}-2m$. Now we can calculate the abelian width and the cuspidal width of the twisted weight following the previous section:
\begin{align*}
    a_1(\mu-m\delta_2) &= a_1(\mu)-m,\\
    \ell_1(\mu-m\delta_2) &= \ell_1(\mu),\\
    a_2(\mu-m\delta_2) &= a_2(\mu) -m,\\
    \ell_2(\mu-m\delta_2) &= \ell_2(\mu).
\end{align*}
The bounds obtained in the second part of the combinatorial lemma will restrict the possibility of the Tate twist. For example, in the case of a maximal parabolic subgroup $P_{\beta}$, we can consider the following cases:

\textbf{Case $\ell_1 \leq \ell_2$.} This condition is independent of the Tate twist as the cuspidal width is invariant under the Tate twist. In this case, the combinatorial lemma satisfies all the successive critical points of the adjoint symmetric cubic L-function. Then the bounds implied by the combinatorial lemma will control the possible Tate twists:
\[
    1-\frac{5}{2}-\frac{\ell_1(\mu)}{2}-a_1(\mu)\leq -m \leq -1-\frac{5}{2}+\frac{\ell_1(\mu)}{2}-a_1(\mu),
\]
which means by Tate twist we can transverse over all successive pairs of critical points of the adjoint symmetric cubic L-function.

\textbf{Case $\ell_1 > \ell_2$.} Again, this condition is independent of the Tate twist as the cuspidal width is invariant under the Tate twist. In this case, the combinatorial lemma satisfies all the successive critical points of the Hecke L-function of the central character of the representation of $GL_2$ with the highest weight $\mu-m\delta_2$. In the same way, the possible Tate twists will satisfy:
\[
    1-2\frac{5}{2}-\ell_2(\mu)-a_2(\mu)\leq -m \leq -1-2\frac{5}{2}+\ell_2(\mu)-a_2(\mu),
\]
which makes it possible to run over all successive pairs of critical points of the Hecke L-function.

\begin{remark}\label{nomore}

We now arrive at a key principle central to the Harder-Raghuram approach: it is sufficient to establish the rationality result only at the point of evaluation. From there, we can deduce the rationality results for the ratios of $L$-values of the ``product'' of $L$-functions at all successive critical values within the critical set of the product. \textit{No more, and no less!}
\end{remark}
\begin{remark}
By the discussion above, starting with a representation whose highest weight satisfies one of the stated conditions, we can prove the rationality results for all critical values of one $L$-function, provided we know the rationality for the others.
\end{remark}

\section{Non-archimedean Considerations}\label{sec:Nonint}
This section examines the non-archimedean aspects of the cohomological interpretation of the Langlands-Shahidi method, as developed in \cite{Raghurampadic}, based on \cite{Waldspurger_2003}, under specific hypotheses, addressing the $p$-adic cases in full generality. We will review these hypotheses and verify their validity in our particular case.

\subsection{Being on the right of the unitary axis}\label{runit}
Fix a non-archimedean place $v$ of $F$. For the remainder of this section, all arguments will be restricted to the local field $F_v$. To simplify notation, we will omit $v$ from the expressions. 

Let $\sigma$ be a smooth irreducible admissible representation of $M$, where $P$ is one of the maximal parabolic subgroups of $G$. As $M\simeq GL_2$, $\sigma$ is generic;  also assume $\sigma$ is essentially tempered, i.e. tempered mod the center. There exists $e\in \mathbb{R}$ and a smooth irreducible unitary tempered representation $\sigma^u$ such that 
\[
    \sigma \simeq \sigma^u \otimes \eta^{e}.
\]

The admissible representation $\tilde{r}_j(\sigma) :=  \tilde{r}_j \circ \phi_{\sigma_v}:W_F\times SL_2(\mathbb{C}) \to \;^{L}GL_{d_j}$ is also essentially tempered. Therefore, there is an exponent $\tilde{f}_j\in \mathbb{R}$ which is a function of the $r_j$ and $\mathbf{e}$, that satisfies
\[
    \tilde{r}_j(\sigma) \simeq \tilde{r}_j(\sigma)^u \otimes |det|^{\tilde{f}_j}.
\]

\begin{defn}
    A smooth irreducible admissible representation $\sigma$ of $M$ that is essentially tempered with the exponents $e\in \mathbb{R}$ is on the right side of the unitary axis with respect to the ambient group $G$, if
\[
    k +\tilde{f_1} >0.
\]
where $k$ is the point of evaluation, as defined in Definition \ref{def:poe}.
\end{defn}
By \cite[Lemma 2.4.2]{Raghurampadic}, we have $\tilde{f}_j = j \cdot \tilde{f}_1$. Consequently, if $\sigma$ lies to the right of the unitary axis, all transfers of $\sigma$ that appear in the constant term introduced in Section \ref{ch:LS} also lie to the right of the unitary axis. Moreover, let $\sigma$ be on the right of the unitary axis concerning ambient group G, then local $L$-values 
\begin{align*}
    L(jk,\;\sigma,\;\tilde{r}_j) &= L(jk,\;\tilde{r}_j(\sigma))= L(jk+\tilde{f}_j,\;\tilde{r}_j(\sigma)^{t}),\\
    L(jk+1,\;\sigma,\;\tilde{r}_j) &= L(jk+\tilde{f}_j+1,\;\tilde{r}_j(\sigma)^{t}),
\end{align*}
are finite for each $1\leq j \leq m$ \cite[Corollary 2.4.4.]{Raghurampadic}.

\subsection{Local critical L-values}
In this section, we establish the validity of Hypothesis in \cite[Hypothesis 4.2.1]{Raghurampadic}, which resolves the parity constraint introduced in \textit{loc. cit.} for our case. Fix an embedding $\iota: E \to \mathbb{C}$, and define
\[
    \chi_{\vartheta}^{1/2}(a) := diag(|a|^{1/2},|a|^{1/2}),  \quad a \in F^{*}.
\]
For each $1\leq j \leq m$, there exists $h_j\in \frac{1}{2}+\mathbb{Z}$ such that $\tilde{r}_j(\chi_{\vartheta}^{\epsilon_{\vartheta}/2})=|\; |^{h_j}$. For each $1 \leq j \leq m$, there exists $h_j \in \frac{1}{2} + \mathbb{Z}$ such that $\tilde{r}j(\chi_{\vartheta}^{\epsilon_{\vartheta}/2}) = |\cdot|^{h_j}$. Let $\sigma$ be an irreducible admissible, half-integrally unitary, essentially tempered representation of $M_{\vartheta}$. The condition ``half-integrally unitary'' means there exists an integer $w$ such that ${^{\iota}\sigma} \otimes \eta^{w/2}$ is unitary. Therefore, we have:
\[
    L(s,\;\sigma,\;\tilde{r}_j) = L(s-h_j,\;\sigma\otimes \chi_{\vartheta}^{\epsilon_{\vartheta}/2},\;\tilde{r}_j).
\]
This twist provides us with algebraic data. For an $L$-value to be critical in the sense of Deligne \cite{Deligne1979ValeursDF}, the evaluation point must be an integer. In our context, this implies that $h_j$ must be a half-integer. A parabolic subgroup $P$ is said to be critical for $G$ if the evaluation point $k=-\langle \rho_P,\alpha_P\rangle$ satisfies the condition:
\[
    jk\in h_j +\mathbb{Z}, \quad \forall 1\leq j \leq m.
\]
Moreover, a parabolic subgroup $P$ is called integral if it satisfies:
\[
    \rho_{P_{\vartheta}}|_{A_{\vartheta}} \in X^{*}(A_{\vartheta}.
\]
\begin{prop}
    Both standard maximal parabolic subgroups of $G$ satisfy the critical and integrality conditions described above.
\end{prop}
\begin{proof}
   For $P_{\beta}$, recall that $\rho_{\beta} = \frac{5}{2}\gamma_s$. Considering the parametrization $t_{\beta}: F^{*}\times F^{*} \to T_{\beta}$. Let $t_{\beta}(a,a) \in A_P$, we have:
\[
    \rho_{\beta}(t_{\beta}(a,a)) = a^{5/2}a^{5/2} = a^5, \quad \textit{is integral}.
\]
The point of evaluation, $k= -\frac{5}{2}$, is half-integral and $m=2$. To show that $P_{\beta}$ is criticality for $G$, it is enough to show that $h_1$ is half-integral. As $\tilde{r_1}=r_3^{0}$, we have:
\begin{align*}
    \tilde{r}_1(\textit{diag}(|a|^{1/2},|a|^{1/2})) &= \textit{diag}((|a|^{1/2})^2(|a|^{1/2})^{-1},|a|^{1/2},|a|^{1/2},(|a|^{1/2})^{-1}(|a|^{1/2})^2)\\
    &=\textit{diag}(|a|^{1/2},|a|^{1/2},|a|^{1/2},|a|^{1/2}),\\
    &\Rightarrow h_1 = \frac{1}{2} \quad \implies -\frac{5}{2} \in \frac{1}{2}+\mathbb{Z}.
\end{align*}
Case $j=2$ is obvious as $2k \in \mathbb{Z}$.

For $P_{\alpha}$, consider the parametrization $t_{0}: F^{*}\times F^{*} \to T_{\alpha}$. Let $t(a,a) \in A_P$, we have:
\[
    \rho_{\alpha}(t(a,a)) = a^{3/2}a^{3/2} = a^3, \quad \textit{is integral} .
\]
In this case the point of evaluation, $k= -\frac{3}{2}$, is half integral and $m=3$. The representation $\tilde{r}_1$ is the standard representation of $GL_2$ so $h_1 = 1/2$. Moreover, for $\tilde{r}_3$ the twist of standard representation by the central character, we have $h_3 = \frac{1}{2}$.

\end{proof}

\subsection{Arithmetic of local intertwining operators: non-archimedean case }
Fix $\vartheta \in \Delta$ and $\iota: E \to \mathbb{C}$. As the Levi is isomorphic to $GL_2$, it follows from the discussion in \cite[Section 3.3.1.]{Raghurampadic} that if a smooth, absolutely irreducible representation ${^\iota}\sigma$ of $M_{\vartheta}$ is half-integrally, essentially tempered, and lies to the right of the unitary axis, then these conditions hold for ${^{\iota_1}\sigma}$ for any embedding ${\iota_1}: E \to \mathbb{C}$.

Define the normalized intertwining operator corresponding to $T_{st}(s,\sigma)$ as:
\[
    T_{norm}(s,\sigma) = r(s,\sigma,w_0)^{-1} T_{st}(s,\sigma)
\]
where $r(s,\sigma,w_0)$ is the normalization factor introduced in Section \ref{ch:LS}.
\begin{prop} \label{thm:nonarchint}
    Let $E$ be large enough, and $\sigma$ be as above. then there exists an $E$-linear G-equivariant map:
    \[
        T_{st,E}: ^{a}\!\operatorname{Ind}_{P}^{G}(\sigma) \to ^{a}\!\operatorname{Ind}_{P}^{G}(\tilde{\sigma}\otimes |\; |^{5}),
    \]
    such that for any embedding $\iota: E \to \mathbb{C}$, we have:
    \[
        T_{norm,E}\otimes_{E,\iota} \mathbb{C} = T_{norm,\iota},
    \]
    where $T_{norm,\; \iota}$ is the normalized intertwining operator at the point of evaluation.
\end{prop}
\begin{proof}
First, the normalized intertwining operator for the unitary representations $\sigma_v \otimes |\cdot|^{w/2}$ is holomorphic and non-vanishing for $\operatorname{Re}(s) > 0$, as established in \cite[Proposition 3.1]{KSannals}. In our case, this implies that $\sigma_v$ lies to the right of the unitary axis. Having verified all the hypotheses in \cite{Raghurampadic}, we can therefore invoke the proof for the general case provided in \textit{loc. cit.}, Theorems 3.3.7 and 4.3.1.
\end{proof}

\section{Intertwining Operators at archimedean Places}\label{sec:arch}
In Section \ref{sec:Nonint}, we studied the arithmetic of the intertwining operator at non-archimedean places. In this section, we will study the arithmetic of the intertwining operators and their contribution to the ratio of the local factors at archimedean places.
\subsection{Induced representations at the point of evaluation}
Let $\mu \in X^{*}_{00}(Res_{F/\mathbb{Q}}(T_2)\times \mathbb{C})$ be a strongly pure weight with purity weight $\mathbf{w}$. Fix an archimedean place $v\in S$, let $\sigma_v$ be an irreducible essentially tempered representation of $M_{\beta} \simeq GL_2(\mathbb{C})$ with highest weight $\mu^{v} = ((a,b),(a^*,b^*))$. From the calculation of the cuspidal parameter in Section \ref{sec:critical}, we obtain:
\begin{align*}
    \tilde{r}_{1}(\sigma_v)&\simeq (\tilde{\mathbb{J}}_{Ad^3(\mu^v)} \otimes |\det |^{\frac{\mathbf{w}}{2}})\otimes |\det |^{\frac{-\mathbf{w}}{2}} \simeq \tilde{r}_1(\sigma_v)^{t} \otimes |\det|^{-a_1(\mu^{v})},\\
    \tilde{r}_{2}(\sigma_v) &\simeq (\tilde{\mathbb{J}}_{\det(\mu^v)} \otimes |\det |^{\mathbf{w}})\otimes |\det|^{-\mathbf{w}} \simeq \tilde{r}_2(\sigma_v)^{t} \otimes |\det|^{-a_2(\mu^{v})},
\end{align*}
where $\tilde{r}_j(\sigma_v)^{u}$ is a unitary tempered irreducible representation of $GL_{d_j}$, for $j\in \{1,2\}$, $d_1=4$ and $d_2=1$. Therefore, the weight $\mu^{v}$  lies to the right of the unitary axis with respect to the ambient group $G$, if
\[
\begin{cases}
-\frac{5}{2}-a_{1}(\mu^{v}) &\geq 0\\
-5-a_{2}(\mu^{v}) &\geq 0
\end{cases} \iff \mathbf{w} \leq -5.
\]
Moreover, we have a decomposition $\sigma_{v} = \sigma_{v}^{u}\otimes |.|^{-\frac{\mathbf{w}}{2}}$, where $\sigma_{v}^{u}$ is a unitary cuspidal irreducible representation. Let $w_{0}:=w_{\gamma_s} =: w_{P_{\beta}}$, the standard intertwining operator is:
\[
    T_{st}: I^{G}_{P_{\beta}}(s,\sigma_v)\to I^{G}_{P_{\beta}}(-s,w_{0}(\sigma)_v).
\]
where, $w_{0}(\sigma_v) = \sigma_{v}\otimes \omega_{\sigma_v}^{-1} = \tilde{\sigma}_{v}$ as shown in \cite{SH89}. Therefore, the highest weight of $\tilde{\sigma}_{v}$ will be $w_0(\mu^v ) = \mu^v -\det(\mu^v )= ((-b,-a),(-b^*,-a^*))$. Evaluating at point of evaluation $k = -\frac{5}{2}$,
\begin{align*}
    I_{P_{\beta}}^{G}(-\frac{5}{2},\sigma_{v}) &= ^{a}\!\operatorname{Ind}_{P_{\beta}}^{G}(\sigma_{v}),\\
    I_{P_{\beta}}^{G}(\frac{5}{2},\tilde{\sigma}_{v}) &= ^{a}\!\operatorname{Ind}_{P_{\beta}}^{G}(\tilde{\sigma}_{v}(5)).
\end{align*}
where $\tilde{\sigma}(5)$ is the Tate twist. Moreover, we have :
\begin{align*}
        T_{st}|_{s=-\frac{5}{2}}& : ^{a}\!\operatorname{Ind}_{P_{\beta}}^{G}(\sigma_{v})\to ^{a}\!\operatorname{Ind}_{P_{\beta}}^{G}(\tilde{\sigma}_{v}(5)),\\
        T^{\vee}_{st}|_{s=\frac{5}{2}}& : ^{a}\!\operatorname{Ind}_{P_{\beta}}^{G}(\tilde{\sigma}_{v}(5))\to ^{a}\!\operatorname{Ind}_{P_{\beta}}^{G}(\sigma_{v}).
\end{align*}

Following Shahidi \cite{ShBook}, let $\psi_{F_v}$ be a nontrivial additive character of $F_v$. We can form a generic character $\psi_v:U\to \mathbb{C}$, then $\sigma_v$ is a $\psi_v$-generic representation of $M_\beta$. Let $\lambda_{\psi_v}(s,\sigma_v)$ be the standard Whittaker functional on $I_{\beta}(s,\sigma_v)$, which is well defined since the cuspidal representations of $GL_2$ are globally and hence locally generic everywhere. It is well known that $\lambda_{\psi}(s,\sigma_{v})$ extends to an entire function that is non-zero everywhere. By multiplicity one for the Whittaker model there is a complex number, called local coefficient denoted by $C_{\psi_v}(s,\sigma_v,w_{0})$,  which satisfies the functional equation:
\[
    C_{\psi_v}(s,\sigma_v,w_{0}) (\lambda_{\psi}(-s,w_{0}(\sigma_{v}))\circ T_{st}(s,\sigma_v))=\lambda_{\psi}(s,\sigma_{v}).
\]
The local coefficient in this case, by \cite{SH89}, is:
\[
    C_{\psi_v}(s,\sigma_v,w_{0}) \sim \dfrac{L(1-s,\tilde{\sigma}_v,Ad^3) L(1-2s,\omega_{\tilde{\sigma}_v})}{L(s,\sigma_v,Ad^3) L(2s,\omega_{\sigma_v})},
\] 
Having the main results in mind, assume $k_{\beta}$ and $k_{\beta}+1$ are critical. Now, we have:
\begin{prop}\label{prop:indirrbeta}
    With the setting introduced in this section, the intertwining operator $T_{st}|_{s=-\frac{5}{2}}$ is an isomorphism between the two irreducible induced representations $I_{\beta}(-\frac{5}{2},\sigma_v)$ and $I_{\beta}(\frac{5}{2},\tilde{\sigma}_{v})$.
\end{prop}
\begin{proof}
First, we prove that $I_{\beta}\left(-\frac{5}{2}, \sigma_v\right)$ and $I_{\beta}\left(\frac{5}{2}, \tilde{\sigma}_v\right)$ are irreducible. If $I\left(-\frac{5}{2}, \sigma_v\right)$ were reducible, then by Casselman-Shahidi's criterion \cite[Proposition 5.3]{CassSha} at $s = -\frac{5}{2}$, we would have
\[
    L\left(1 - s, \tilde{\sigma}_v, \operatorname{Ad}^3\right)^{-1} L\left(1 - 2s, \omega_{\tilde{\sigma}_v}\right)^{-1} = 0,
\]
which is a contradiction. The L-functions above do not have poles or zeros at $s = -\frac{5}{2}$, as we assume that $s = -\frac{5}{2}$ is critical for these L-functions.

Next, we show that the intertwining operator $T_{st}$ at $s = -\frac{5}{2}$, between irreducible induced representations, is finite and non-zero. If $T_{st}\big|_{s = -\frac{5}{2}}$ were zero, then the following product would have a pole:
\[
    L\left(1 - \left(-\frac{5}{2}\right), \tilde{\sigma}_v, \operatorname{Ad}^3\right) L\left(1 - 2\left(-\frac{5}{2}\right), \omega_{\tilde{\sigma}_v}\right),
\]
which contradicts the assumption that $-\frac{5}{2}$ is a critical point of these L-functions as $\lambda_{\psi}\left(-\frac{5}{2}, \sigma_v\right)$ is non-zero. If $T_{st}\big|_{s = -\frac{5}{2}}$ were a pole, then the following product would also have a pole:
\[
    L\left(-\frac{5}{2}, \sigma_v, \operatorname{Ad}^3\right) L\left(2\left(-\frac{5}{2}\right), \omega_{\sigma_v}\right),
\]
which again contradicts the assumption that $-\frac{5}{2}$ is a critical point of these L-functions. Here, $\lambda_{\psi}\left(-\frac{5}{2}, \sigma_v\right)$ must be non-zero.
\end{proof}

The case $P_{\alpha}$ can be easily proved using the same method. We will simply state the corresponding theorem:
\begin{prop}\label{prop:indirralpha}
    Let have assumptions as the previous assumption for the case $P_{\alpha}$. Then, the intertwining operator $T_{st}|_{s=-\frac{3}{2}}$ is an isomorphism between the two irreducible induced representations $I_{\alpha}(-\frac{3}{2},\sigma_v) = ^{a}\!\operatorname{Ind}_{P_{\alpha}}^{G}(\sigma_{v})$ and $I_{P_{\alpha}}^{G}(\frac{3}{2},\tilde{\sigma}_{v}) = ^{a}\!\operatorname{Ind}_{P_{\alpha}}^{G}(\tilde{\sigma}_{v}(3))$.
\end{prop}

\subsection{Arithmetic of the archimedean intertwining operators} \label{sec:aaintop}
Fix an archimedean place $v$ of $F$. Since $F$ is totally imaginary,  we have $F_{v} \simeq \mathbb{C}$. Following Section \ref{sec:Preliminaries}, we can define $X^{*}(T) = \mathbb{Z}(2\alpha + \beta) \oplus \mathbb{Z}(\alpha + \beta)$ and identify $T(\mathbb{C})$ with $\mathbb{C}^{*}\times \mathbb{C}^{*}$ via the representative:
\begin{align*}
    t_0 : T &\to \mathbb{C}^{*}\times \mathbb{C}^{*},\\
    t &\mapsto ((2\alpha+\beta)(t), (\alpha+\beta)(t)). \numberthis
    \label{T}
\end{align*}

Furthermore, we introduced the reparametrizations $t_{\alpha}$ and $t_{\beta}$ of $T(\mathbb{C})$ within the Levi factors $M_{\alpha}$ and $M_{\beta}$, respectively. Recall that $t_{\alpha} = t_0$ as noted in Remark \ref{choice of character}. The relationship between these parametrizations can be described as follows:
 \begin{align*}
     \mathbb{C}^{*}\otimes \mathbb{C}^{*} \xrightarrow{t_{0}^{-1}} T(\mathbb{C}) \rightarrow &M_{\beta} \xrightarrow{t_{\beta}^{-1}} GL_2(\mathbb{C}) \hookleftarrow  \mathbb{C}^{*}\otimes \mathbb{C}^{*},\\
     t_{\beta}^{-1}\circ t_{0}^{-1}(a,b) &= (b, ab^{-1}),\\
     t_{0}\circ t_{\beta}(a,b) &= (ab, a).
 \end{align*} 
 Let $\chi$ and $\psi$ be characters of $T(\mathbb{C})$. There exist characters $\chi_1$, $\chi_2$, $\psi_1$, and $\psi_2$ of $\mathbb{C}^{*}$ such that:
 \begin{align*}
     t_0(\chi) &= \chi_1 \otimes \chi_2,  &t_0(\psi) &= \psi_2 \otimes \psi_1\psi_2^{-1},\\
     t_{\beta}^{-1}(\chi) &= \chi_1\chi_2 \otimes \chi_1, & t_{\beta}^{-1}(\psi) &= \psi_1\otimes \psi_2. \numberthis \label{char}
 \end{align*}
We will represent the characters of $\mathbb{C}^{*}\times \mathbb{C}^{*}$ originating from the reparametrization $t_{\beta}$ as $\psi$, and the one originating from $t_0 = t_{\alpha}$ as $\chi$. 

Let $\mu = ((a,b), (a^{*},b^{*}))$ be a strongly pure weight of $M_{\beta} \simeq GL_{2}(\mathbb{C})$. 
Following Section \ref{sec:cohrep},
\begin{equation}
    \sigma_v = \mathbb{J}_{\mu} = \operatorname{Ind}_{B_2(\mathbb{C})}^{GL_2(\mathbb{C})}(z^{\alpha_1}\bar{z}^{\beta_1}\otimes z^{\alpha_2}\bar{z}^{\beta_2}),
\end{equation}
where $\alpha_i$ and $\beta_i$ are the cuspidal parameters of $\mu$ where:
\begin{align*}
    \alpha_1 &= -b+\frac{1}{2}, &  \alpha_2 = -a-\frac{1}{2},\\
    \beta_1 &= -a^{*}-\frac{1}{2}, &  \beta_2 = -b^{*}+\frac{1}{2}.
\end{align*}

Let $t_{\beta}(\psi) = \psi_1 \otimes \psi_2$ denote a character of $T \subset M_{\beta}$:
\[
    \psi_i= z^{\alpha_i}\bar{z}^{\beta_i}, \qquad i = 1,2.
\]
By (\ref{char}) and the transitivity of the normalized induction, we have:
\begin{align*}
    \operatorname{Ind}_{P_{\beta}(\mathbb{C})}^{G_2(\mathbb{C})}(\sigma_v(-\frac{5}{2}))&= \operatorname{Ind}_{P_{\beta}(\mathbb{C})}^{G_2(\mathbb{C})}(\mathbb{J}_{\mu}(-\frac{5}{2})) =  \operatorname{Ind}_{B(\mathbb{C})}^{G_2(\mathbb{C})} (\psi_2(-\frac{5}{2})\otimes \psi_1\psi_2^{-1}),\\
    \operatorname{Ind}_{P_{\beta}(\mathbb{C})}^{G_2(\mathbb{C})}(\tilde{\sigma}_v(\frac{5}{2}))&= \operatorname{Ind}_{P_{\beta}(\mathbb{C})}^{G_2(\mathbb{C})}(\mathbb{J}_{w_{P}(\mu)}(\frac{5}{2})) = \operatorname{Ind}_{B(\mathbb{C})}^{G_2(\mathbb{C})} (\psi_1^{-1}(\frac{5}{2})\otimes \psi_1\psi_2^{-1}).
\end{align*}
Therefore, $T_{st}$ is also an intertwining operator between principal series representations, denoted by $A_{v}(w_{P})$.
\begin{equation}
\begin{tikzcd}[column sep=4em]
%% first row
\operatorname{Ind}_{P_{\beta}(\mathbb{C})}^{G_2(\mathbb{C})}(\sigma_v(-\frac{5}{2}))\arrow[rrr,"T_{st}"]\ar[equal]{dd}  &&& \operatorname{Ind}_{P_{\beta}(\mathbb{C})}^{G_2(\mathbb{C})}(\tilde{\sigma}_v(\frac{5}{2}))\ar[equal]{dd} \\
&&&\\
\operatorname{Ind}_{B(\mathbb{C})}^{G_2(\mathbb{C})} (\psi_2(-\frac{5}{2})\otimes \psi_1\psi_2^{-1}) \arrow[rrr,"A_{v}(w_P)"] &&& \operatorname{Ind}_{B(\mathbb{C})}^{G_2(\mathbb{C})} (\psi_1^{-1}(\frac{5}{2})\otimes \psi_1\psi_2^{-1})
\end{tikzcd}
\end{equation}

To simplify the notations, let $\chi = \chi_1 \otimes \chi_2$ be a character of $\mathbb{C}^{*}\times \mathbb{C}^{*} \simeq T$ we will use the following convention:
\begin{align*}
    I(\chi_1 \otimes \chi_2)  &:= \operatorname{Ind}_{B(\mathbb{C})}^{G_2(\mathbb{C})} (\chi_1\otimes \chi_2)\\
    &= \operatorname{Ind}_{P_{\beta}(\mathbb{C})}^{G_2(\mathbb{C})}\operatorname{Ind}_{B_2(\mathbb{C})}^{GL_2(\mathbb{C})}(\chi_1\chi_2\otimes \chi_1)\\
    &= \operatorname{Ind}_{P_{\alpha}(\mathbb{C})}^{G_2(\mathbb{C})}\operatorname{Ind}_{B_2(\mathbb{C})}^{GL_2(\mathbb{C})}(\chi_1\otimes \chi_2).
\end{align*}

Moreover, the action of the simple reflections in the Weyl group on the characters of the group are as follows:
\begin{align*}
     w_{\alpha}(t_{0}(\chi)) &= w_{\alpha}(\chi_1\otimes \chi_2) = \chi_2\otimes \chi_1, & w_{\beta}(t_{0}(\chi)) &= w_{\beta}(\chi_1\otimes \chi_2) = \chi_1\chi_2\otimes \chi_2^{-1},\\
     w_{\alpha}(t_{\beta}(\psi)) &= w_{\alpha}(\psi_1\otimes \psi_2) = \psi_1\otimes \psi_1\psi_2^{-1}, & w_{\beta}(t_{\beta}(\psi)) &= w_{\beta}(\psi_1\otimes \psi_2) = \psi_2\otimes \psi_1.
\end{align*}

Consider the minimal factorization $w_{P}=w_{\alpha}w_{\beta}w_{\alpha}w_{\beta}w_{\alpha}$ into simple reflections. We have the following ordered set characterizes the roots in $\mathfrak{u}_{\beta}$:
\begin{align*}
    \{\gamma>0, w_P \gamma<0\} &= \{\alpha,\;w_{\alpha}(\beta),\;w_{\alpha\beta}(\alpha),\;w_{\alpha\beta\alpha}(\beta),\;w_{\alpha\beta\alpha\beta}(\alpha)\}\\
    &= \{\alpha, \;3\alpha+\beta, \;2\alpha+\beta, \;3\alpha+2\beta, \;\alpha+\beta\}.
\end{align*}
For each simple root $\gamma$ one can define the homomorphism $\phi_{\gamma}:SL_2(\mathbb{C}) \to G_2(\mathbb{C})$ such that:
\[
\phi_{\gamma}\begin{pmatrix}
1 & u\\
0 & 1
\end{pmatrix} = e_{\gamma}(u),\quad \phi_{\gamma}\begin{pmatrix}
u & 0\\
0 & u^{-1}
\end{pmatrix} = h_{\gamma}(u), \quad \phi_{\gamma}\begin{pmatrix}
0 & 1\\
-1 & 0
\end{pmatrix} = w_{\gamma}. 
\]

Let $\bar{U}(w_P) = U \cap w_P \bar{U} w_P^{-1}$, where $\bar{U}$ is the opposite unipotent radical. Elements of $\bar{U}(w)$ can be expressed as $\prod_{\gamma>0, w_P\gamma<0} e_{-\gamma}(u_{\gamma})$ for $u_{\gamma} \in \mathbb{C}$. Using the ordered set above we have:
\[
    \bar{U}(w_P) = U_{-(\alpha+\beta)} . U_{-(3\alpha+2\beta)} . U_{-(2\alpha+\beta)} . U_{-(3\alpha+\beta)} . U_{-(\alpha)}
\]

We can  decompose the intertwining operator $A_{v}(w_P)$ to rank one operators \cite{ShBook}:
\begin{multline*}
    A_{v}(w_P,\chi) f(g) = \int_{U(w_P)} f(w_P^{-1}ng)dn\\
                         = A_{v}(w_{\alpha},w_{\beta\alpha\beta\alpha}\chi)\circ A_{v}(w_{\beta},w_{\alpha\beta\alpha}\chi) \circ A_{v}(w_{\alpha},w_{\beta\alpha}\chi)\circ A_{v}(w_{\beta},w_{\alpha}\chi) \circ A_{v}(w_{\alpha},\chi)f(g)\\
                         = A_{v}(w_{\alpha},\chi_1)\circ A_{v}(w_{\beta},\chi_2) \circ A_{v}(w_{\alpha},\chi_3)\circ A_{v}(w_{\beta},\chi_4) \circ A_{v}(w_{\alpha},\chi_5)f(g).    
\end{multline*}

The measure $du$ at each step will be a Lebesgue measure, and $dn$ is the product of these five Lebesgue measures. Each operator in the decomposition above is an intertwining operator on $GL_2(\mathbb{C})$ extended to an intertwining operator of $G_2(\mathbb{C})$. For short root $\gamma$, we can extend the map $\phi_{\gamma}$ defined as above to $\phi_{\gamma}: GL_2(\mathbb{C}) \simeq M_{\alpha}(\mathbb{C}) \hookrightarrow G_2(\mathbb{C})$. On the other hand for long root $\gamma$, the map $\phi_{\gamma}$ can be extended to $\phi_{\gamma}: GL_2(\mathbb{C}) \simeq M_{\beta}(\mathbb{C}) \hookrightarrow G_2(\mathbb{C})$. Considering the simple reflection we have to study the following cases separately:
\begin{align}
    A_{v}(w_{\alpha},\chi) &= \operatorname{Ind}_{P_{\alpha}(\mathbb{C})}^{G_2(\mathbb{C})} A_{v}(w_{\alpha},\chi_1 \otimes \chi_2), \label{Aalpha}\\
    A_{v}(w_{\beta},\chi) &= \operatorname{Ind}_{P_{\beta}(\mathbb{C})}^{G_2(\mathbb{C})} A_{v}(w_{\beta},\chi_1\chi_2 \otimes \chi_1).\label{Abeta}\
\end{align}

With abuse of notation, we also use $A_v$ for the intertwining operators for $GL_2$. It is worth noting that \ref{Aalpha} is an intertwining operator for $GL_2(\mathbb{C}) \simeq M_{\alpha}$ while \ref{Abeta} is an intertwining operator for $GL_2(\mathbb{C}) \simeq M_{\beta}$, that's why the character for \ref{Abeta} is changed. Moreover, for simple root $\alpha\in \Delta$ (or $\beta$) we can form the maximal compact subgroup of the Levi as $K_{\alpha} = M_{\alpha}(\mathbb{C}) \cap K_{\infty} \simeq SU(2)$ ( or $K_{\beta} = M_{\beta}(\mathbb{C}) \cap K_{\infty} \simeq SU(2)$). In each case for $\vartheta \in \Delta$, the maximal compact torus of $K_{\infty}$ is isomorphic to that of $K_{\vartheta}$. Furthermore, we have $K_{2} = B(\mathbb{C})\cap K_{\infty}$, and by the Iwasawa decomposition, $G_2(\mathbb{C}) = B(\mathbb{C})K_{\infty}$. For any character $\chi \simeq \chi_1\otimes \chi_2$ of $B(\mathbb{C})$, we can identify the principal series representation $\operatorname{Ind}_{B(\mathbb{C})}^{GL_2(\mathbb{C})}(\chi)$ as a $K_{\infty}$-module with the induced representation $\operatorname{Ind}_{K_{2}}^{K_{\infty}}(\chi)$. Again, for the simple root $\gamma\in \Delta$ the restriction $\operatorname{Ind}_{K_{2}}^{K_{\infty}}(\chi)\big|_{K_{\vartheta}}$ is a representation of the $K_{\vartheta}$-module of the representation $\operatorname{Ind}_{B(\mathbb{C})}^{G_2(\mathbb{C})}(\chi)\big|_{GL_2}$ where it reduces the calculation to the case $GL_2$ intertwining operators.

In the previous section, we showed that $I((\psi_2(-\frac{5}{2})\otimes \psi_1\psi_2^{-1}))$ and $I((\psi_1^{-1}(\frac{5}{2})\otimes \psi_1\psi_2^{-1}))$ are irreducible. By \cite{Speh1980ReducibilityOG}, each standard module appearing in the intermediate states is also irreducible. Following \cite{RagTI}, by evaluating the values of intertwining operators at the lowest $K_\alpha$-types (or $K_{\beta}$-types), we can interpret them as a scaling by a factor.  Let us denote $Ind_{B_2(\mathbb{C})}^{GL_2(\mathbb{C})} (\chi_1\otimes \chi_2)$ by $\chi_1 \times \chi_2$. For the case $w_{\alpha}$:
\begin{equation}
\begin{tikzcd}[column sep=4em]
%% first row
\operatorname{Ind}_{B(\mathbb{C})}^{G_2(\mathbb{C})} (\chi_1\otimes \chi_2)\arrow[rrr,"A_{v}(w_{\alpha})"]  &&& \operatorname{Ind}_{B(\mathbb{C})}^{G_2(\mathbb{C})} (\chi_2\otimes \chi_1) \\
&&&\\
\chi_1\times \chi_2 \arrow[uu,"\operatorname{Ind}_{P_{\alpha}(\mathbb{C})}^{G_2(\mathbb{C})}"] \arrow[rrr,"A_{v}(w_\alpha)"] &&& \chi_2 \times \chi_1\arrow[uu,"\operatorname{Ind}_{P_{\alpha}(\mathbb{C})}^{G_2(\mathbb{C})}"]
\end{tikzcd}
\end{equation}
We can compute this scaling factor up to a non-zero rational, same as \cite[Proposition 4.16.]{RagTI},
\begin{equation}
    \frac{L_{\infty}(0,\chi_1\chi_2^{-1})}{L_{\infty}(1,\chi_1\chi_2^{-1})}.
\end{equation}

For the case $w_{\beta}$, we have:
\begin{equation}    
\begin{tikzcd}[column sep=4em]
%% first row
\operatorname{Ind}_{B(\mathbb{C})}^{G_2(\mathbb{C})} (\chi_1\otimes \chi_2)\arrow[rrr,"A_{v}(w_{\beta})"]  &&& \operatorname{Ind}_{B(\mathbb{C})}^{G_2(\mathbb{C})} (\chi_1\chi_2\otimes \chi_2^{-1}) \\
&&&\\
\chi_1\chi_2\times \chi_1 \arrow[uu,"\operatorname{Ind}_{P_{\beta}(\mathbb{C})}^{G_2(\mathbb{C})}"] \arrow[rrr,"A_{v}(w_\beta)"] &&& \chi_1\times \chi_1\chi_2\arrow[uu,"\operatorname{Ind}_{P_{\beta}(\mathbb{C})}^{G_2(\mathbb{C})}"]
\end{tikzcd}
\end{equation}
Again, by appealing \cite{RagTI}, the scaling factor up to a non-zero rational is equal to:
\begin{equation}
    \frac{L_{\infty}(0,\chi_2)}{L_{\infty}(1,\chi_2)}.
\end{equation}

The following diagram demonstrates these intermediate intertwining operators in each step. The order of the operators follows the decomposition above, and to simplify the notation we dropped the character part in the notation and let $A_{v}(w_{\alpha},\chi_{i}) =: A_v(w_{\alpha})$ for $i=1,3,5$, and $A_{v}(w_{\beta}, \chi_{j}) =: A_v(w_{\beta})$ for $j=2,4$.
\begin{figure}[ht]
    \centering
    \[
\begin{tikzcd}[column sep=4em]
%% first row
I_{\beta}(-\frac{5}{2},\sigma_v) = \operatorname{Ind}_{P_{\beta}(\mathbb{C})}^{G_2(\mathbb{C})}(\sigma_v(-\frac{5}{2}))\ar[equal]{rrr}\arrow[dddddddddd,"T_{st}"]  &&& I(\psi_2(-\frac{5}{2})\otimes \psi_1\psi_2^{-1})\arrow[dd,"A_{v}(w_{\alpha})"]\\
&&&\\
&&& I(\psi_1\psi_2^{-1}\otimes \psi_2(-\frac{5}{2})) \arrow[dd,"A_{v}(w_{\beta})"]\\
&&&\\
&&& I(\psi_1(-\frac{5}{2}) \otimes \psi_2^{-1}(\frac{5}{2})) \arrow[dd,"A_{v}(w_{\alpha})"]\\
&&&\\
&&& I(\psi_2^{-1}(\frac{5}{2}) \otimes \psi_1(-\frac{5}{2})) \arrow[dd,"A_{v}(w_{\beta})"]\\
&&&\\
&&& I(\psi_1\psi_2^{-1} \times \psi_1^{-1}(\frac{5}{2})) \arrow[dd,"A_{v}(w_{\alpha})"]\\
&&&\\
I_{\beta}(\frac{5}{2},\tilde{\sigma}_v) = \operatorname{Ind}_{P_{\beta}(\mathbb{C})}^{G_2(\mathbb{C})}(\tilde{\sigma}_v(\frac{5}{2}))\ar[equal]{rrr} &&& I(\psi_1^{-1}(\frac{5}{2})\otimes \psi_1\psi_2^{-1})
\end{tikzcd}
\]\label{fig:ASIOB}
    \caption{archimedean Standard Intertwining Operator, case $P_{\beta}$.}
\end{figure}%G2 coroot Lattice

By irreducibility results from the previous section, the fact that $T_{st}$ is an isomorphism, we can interpret $T_{st}$ the scaling by a factor which can be computed using diagrams above:
\[
\frac{L_{\infty}(0,\psi_1^{-1}\psi_2^{2}(-\frac{5}{2}))}{L_{\infty}(1,\psi_1^{-1}\psi_2^{2}(-\frac{5}{2}))}\frac{L_{\infty}(0,\psi_2(-\frac{5}{2}))}{L_{\infty}(1,\psi_2(-\frac{5}{2}))}\frac{L_{\infty}(0,\psi_1\psi_2(2(-\frac{5}{2})))}{L_{\infty}(1,\psi_1\psi_2(2(-\frac{5}{2})))}\frac{L_{\infty}(0,\psi_1(-\frac{5}{2}))}{L_{\infty}(1,\psi_1(-\frac{5}{2}))}\frac{L_{\infty}(0,\psi_1^{2}\psi_2^{-1}(-\frac{5}{2}))}{L_{\infty}(1,\psi_1^{2}\psi_2^{-1}(-\frac{5}{2}))}.
\]
We can simplify it as:
\begin{equation}
    \frac{L_{\infty}(0,\operatorname{Ad}^{3}(\psi_1(-\frac{5}{2})\times\psi_2(-\frac{5}{2})))}{L_{\infty}(1,\operatorname{Ad}^{3}(\psi_1(-\frac{5}{2})\times\psi_2(-\frac{5}{2})))}\frac{L_{\infty}(0,\omega_{\psi_1(-\frac{5}{2})\times\psi_2(-\frac{5}{2})})}{L_{\infty}(1,\omega_{\psi_1(-\frac{5}{2})\times\psi_2(-\frac{5}{2})})}.
\end{equation}
or equivalently:
\begin{equation}\label{eq:archLb}
    \frac{L_{\infty}(-\frac{5}{2},\operatorname{Ad}^{3}(\psi_1\times\psi_2))}{L_{\infty}(1-\frac{5}{2},\operatorname{Ad}^{3}(\psi_1\times\psi_2))}\frac{L_{\infty}(2(-\frac{5}{2}),\omega_{\psi_1\times\psi_2})}{L_{\infty}(1+2(-\frac{5}{2}),\omega_{\psi_1\times\psi_2})}.
\end{equation}

We can carry the same argument for the case $P_{\alpha}$. In this case, we need to consider the minimal factorization $\omega_P = \omega_{\beta}\omega_{\alpha}\omega_{\beta}\omega_{\alpha}\omega_{\beta}$. The following diagram demonstrates the intermediate intertwining operators in each step for this case. In the same way as in the case $P_{\beta}$, we can form:
\begin{multline}
  \frac{L_{\infty}(0,\psi_2(-\frac{3}{2}))}{L_{\infty}(1,\psi_2(-\frac{3}{2}))}
\frac{L_{\infty}(0,\psi_1\psi_2^{2}(3(-\frac{3}{2})))}{L_{\infty}(1,\psi_1\psi_2^{2}(3(-\frac{3}{2})))}\\
\frac{L_{\infty}(0,\psi_1\psi_2(2(-\frac{3}{2})))}{L_{\infty}(1,\psi_1\psi_2(2(-\frac{3}{2})))}\frac{L_{\infty}(0,\psi_1^{2}\psi_2(3(-\frac{3}{2})))}{L_{\infty}(1,\psi_1^{2}\psi_2(3(-\frac{3}{2})))}\frac{L_{\infty}(0,\psi_1(-\frac{3}{2}))}{L_{\infty}(1,\psi_2(-\frac{3}{2}))}. 
\end{multline}

Simplifying it will give us our desired ratio of $L$-functions:
\begin{equation}\label{eq:archLa}
\frac{L_{\infty}(-\frac{3}{2}, \psi_1\times\psi_2)}{L_{\infty}(1-\frac{3}{2}, \psi_1\times\psi_2)}\frac{L_{\infty}(-2\frac{3}{2}, \omega_{\psi_1\times\psi_2})}{L_{\infty}(1-2\frac{3}{2},\omega_{ \psi_1\times\psi_2})}
\frac{L_{\infty}(-3\frac{3}{2}, (\psi_1\times\psi_2)\otimes \omega_{\psi_1\times\psi_2})}{L_{\infty}(1-3\frac{3}{2}, (\psi_1\times\psi_2)\otimes \omega_{\psi_1\times\psi_2})}.
\end{equation}

\subsection{Cohomological intertwining operators}
Let $P=P_{\beta}$. Fix $\mu \in X_{00}^{*}(Res_{F/\mathbb{Q}}(T_2) \times E)$ satisfying the conditions of the combinatorial lemma, then there is a balanced Kostant representative $w_{cl}\in W^{P_{\beta}}$ such that $\lambda := w_{cl}^{-1}.\mu$ is an integral dominant weight of $G$. Fix $\iota: E \to \mathbb{C}$ and a place $v \in S_{\infty}$, by the Künneth theorem, we can simplify our analysis of the relative cohomology group by concentrating on the local case at the chosen place $v$. Consequently, we will omit $\iota$ and the subscript $v$ from our notation and follow the notation introduced in the previous section.

Let $\sigma_f\in Coh_{!!}(Res_{F/\mathbb{Q}}(GL_2),\mu)$, by Delorme's lemma, for example \cite[Theorem III.3.3.]{BoWal}, at the bottom degree $q_b := 1 + 5 =6$, we have:
\begin{align}\label{delorme}
 H^{6}(\mathfrak{g}_2(\mathbb{C}),K_{\infty}; ^{a}\!\operatorname{Ind}_{P(\mathbb{C})}^{G_2(\mathbb{C})}(\mathbb{J}_{\mu})\otimes \mathcal{M}_{\lambda}) &\simeq H^{1}(\mathfrak{m}(\mathbb{C}),K_{\infty}^{M}; \mathbb{J}_{\mu}\otimes \mathcal{M}_{\mu}).
\end{align}
 We can once again apply Delorme's lemma for $\mathfrak{m} \simeq \mathfrak{gl}_2$, $K_{\infty}^{M}\simeq SU(2)$, $\mu = ((a,b),(a^{*},b^{*}))$ and $\mathbb{J}_{\mu} = \operatorname{Ind}_{B_2(\mathbb{C})} ^{GL_{2}(\mathbb{C})}(z^{-b+1/2}\bar{z}^{-a^{*}-1/2}\otimes z^{-a-1/2}\bar{z}^{-b^{*}-1/2})$. Building upon the calculations presented in \cite[Section 4.1.]{RagTI} for $GL_2$, we can express  Equation \ref{delorme}, as follows:
\begin{align}
 H^{0}(\mathfrak{gl_1}(\mathbb{C}), SU(1); z^{-b+1}\bar{z}^{-a^{*}} \otimes \mathcal{M}_{b-1,a^{*}})&\otimes H^{0}(\mathfrak{gl_1}(\mathbb{C}), SU(1); z^{-a-1}\bar{z}^{-b^{*}} \otimes \mathcal{M}_{a+1,b^{*}}).
\end{align} 
where $\mathcal{M}_{x,x^{*}}$ is the algebraic representation $z \mapsto z^x \bar{z}^{x^{*}}$ of $GL_1(\mathbb{C})$. 
For simplicity, denote:
\begin{align*}
    H^{0}_{(x,x^{*})} &:= H^{0}(\mathfrak{gl_1}(\mathbb{C}), SU(1); z^{-x}\bar{z}^{-x^{*}} \otimes \mathcal{M}_{x,x^{*}});\\
    \mathcal{J} &:= ^{a}\!\operatorname{Ind}_{P(\mathbb{C})}^{G_2(\mathbb{C})}(\mathbb{J}_{\mu}).
\end{align*}
Since the relative lie algebra cohomology group $H^6(\mathfrak{g}_2(\mathbb{C}),K_{\infty}; \mathcal{J} \otimes \mathcal{M}_{\lambda})$ is one-dimensional, we can fix a basis $[\mathcal{J}]$ so we have:
\[
H^6(\mathfrak{g}_2(\mathbb{C}),K_{\infty}; \mathcal{J} \otimes \mathcal{M}_{\lambda}) =\mathbb{C}[\mathcal{J}].
\]
We can present the basis as an element of the complex,
\[
    Hom_{K_{\infty}}(\wedge^{6}(\mathfrak{g}_2/\mathfrak{k}),\mathcal{J}\otimes M_{\lambda}) \simeq Hom_{K_{\infty}}(\mathbf{1},\wedge^{6}(\mathfrak{g}_2/\mathfrak{k})^{*}\otimes\mathcal{J}\otimes M_{\lambda}).
\]
Let $X_{i}^{*} = X_{i_1}^{*}\wedge \dots \wedge X_{i_{6}}^{*}$ be a basis of $\wedge^{6}(\mathfrak{g_2}/\mathfrak{k})^{*}$, $\{m_{\alpha}\}_{\alpha}$ be the basis of $\mathcal{M}_{\lambda}$ and $\phi_{i,\alpha} \in \mathcal{J}$.
\begin{align} 
    [\mathcal{J}] = \sum_{i,\alpha} X^{*}_i\otimes \phi_{i,\alpha} \otimes m_{\alpha}.
\end{align}
On the other hand, we can fix a basis $[\omega_{(x,x^{*})}]_0$ for $H^{0}_{(x,x^{*})}$ which is a rational class corresponding to the cohomology of the trivial representation of $\mathbb{C}^{*}$. Now, we can conclude our argument by forming the following isomorphism:
\begin{align*}
    \gamma_1: H^6(\mathfrak{g}_2(\mathbb{C}),K_{\infty}; \mathcal{J} \otimes \mathcal{M}_{\lambda}) &\simeq H^{0}_{(b-1,a^{*})}\otimes H^{0}_{(a+1,b^{*})}\\
    [\mathcal{J}]_0\qquad &\mapsto [\omega_{(b-1,a^{*})}]_0 \otimes [\omega_{(a+1,b^{*})}]_0, \numberthis \label{rationalarchinter}
\end{align*}
where $[\mathcal{J}]_0 := \gamma^{-1}([\omega_{(b-1,a^{*})}]_0 \otimes [\omega_{(a+1,b^{*})}]_0)$ is a scale of $[\mathcal{J}]$ with highest weight vector $f_0$ with the lowest $K$-type that takes rational values. 
In the same way, we can define a rational class in the relative lie group cohomology in the co-domain of the standard intertwining operator in cohomology groups. Let $\tilde{\mu} = w_{P}(\mu) = ((-b,-a),(-b^{*},-a^{*}))$ which is a strongly pure weight. Moreover, $\tilde{\mu}$ satisfies the combinatorial lemma condition. There exists a balanced Kostant representative $w_{cl}^{'}$, by Lemma \ref{lem:combbeta}, where $\lambda = w_{cl}^{'}.\tilde{\mu}$. Also, denote:
\[
    \tilde{\mathcal{J}} := ^{a}\!\operatorname{Ind}_{P(\mathbb{C})}^{G_2(\mathbb{C})}(\tilde{\mathbb{J}}_{\mu}).
\] 
Now applying Delorme's method and calculations as above, we have the following isomorphism:
\begin{align*}
    \gamma_{w_{P}}: H^6(\mathfrak{g}_2,K_{\infty}; \tilde{\mathcal{J}} \otimes \mathcal{M}_{\lambda}) &\simeq H^{0}_{(-a-1,-b^{*})}\otimes H^{0}_{(-b+1,-a^{*})}\\
    [\tilde{\mathcal{J}}]_0 \qquad &\mapsto [\omega_{(-a-1,-b^{*})}] \otimes [\omega_{(-b+1,-a^{*})}] \numberthis
\end{align*}
where we can present the basis element on the left hand-side as:
\[
    [\tilde{\mathcal{J}}] = \sum_{i,\alpha} X^{*}_i\otimes T_{st}(\phi_{i,\alpha}) \otimes m_{\alpha},
\]
again with the same calculations in \cite[Section 4.1.]{RagTI}, we can locate the $\mathbf{f}_{0} := \phi_{i,\alpha}$ for some $i$ and $\alpha$ which is the highest weight vector of the lowest $K$-type $\mathcal{J}$, and up to a non-zero $\mathbb{Q}^{\times}$ the highest weight vector $\tilde{\mathbf{f}}_{0} = T_{st}(\mathbf{f}_{0})$ of the lowest $K$-type $\tilde{\mathcal{J}}$ which showed up in the calculations of the  intertwining operators in previous section.

In the previous section, we calculated the cocycle decomposition of the standard intertwining operator. We will define $\mathcal{J}_{\omega}$ where $\omega \in W^P$, to be the irreducible principal series representation of the intermediate steps shown in Figure \ref{fig:ASIOB}. Furthermore, we can induce the intertwining operators $A_v(\alpha)$ and $A_v(\beta)$ between the cohomology groups which are basically the induction of $GL_2$ intertwining operators. Again, we can use Delorme's lemma to form isomorphisms between the one-dimensional cohomology groups, as follows:
\[
    \gamma_{w} : H^{q_b}(\mathfrak{g}_2,K_{\infty}, \mathcal{J}_{w}\otimes \mathcal{M}_{\lambda}) \simeq H^{0}_{(x,x^{*})} \otimes H^{0}_{(y,y^{*})},
\]
where $x$ and $y$ are integers calculated by restriction to $GL_2$ case. For any $w\in W^P$, we can fix a rational basis in right hand-side which is simply $[\omega_{(x,x^{*})}]\otimes [\omega_{(y,y^{*})}]$ and define a basis element $[\mathcal{J}_{w}]_0 = \gamma_{w}^{-1}([\omega_{(x,x^{*})}]\otimes [\omega_{(y,y^{*})}])$ for $H^{q_b}(\mathfrak{g}_2,K_{\infty}, \mathcal{J}_{w}\otimes \mathcal{M}_{\lambda})$. Now  using the calculations from the previous section, we will have:

\begin{prop}\label{archinterthm}
    For $v\in S_{\infty}$ and a strongly pure weight $\mu$ satisfying the conditions of the combinatorial lemma, the archimedean intertwining operator $T_{st}$ induces a linear isomorphism between the one-dimensional cohomology spaces $H^6(\mathfrak{g}_2(\mathbb{C}),K_{\infty}; \mathcal{J} \otimes \mathcal{M}_{\lambda})$ and $H^6(\mathfrak{g}_2,K_{\infty}; \tilde{\mathcal{J}} \otimes \mathcal{M}_{\lambda})$ such that:
    \[
        T^{q_b}_{st}([\mathcal{J}]_0) \approx_{\mathbb{Q}^{\times}} \frac{L_{\infty}(-\frac{5}{2},\operatorname{Ad}^{3}(\psi_1\times\psi_2))}{L_{\infty}(1-\frac{5}{2},\operatorname{Ad}^{3}(\psi_1\times\psi_2))}\frac{L_{\infty}(2(-\frac{5}{2}),\omega_{\psi_1\times\psi_2})}{L_{\infty}(1+2(-\frac{5}{2}),\omega_{\psi_1\times\psi_2})} [\tilde{\mathcal{J}}]_0.
    \]
\end{prop}

Following the same line of proof, we can deduce the following results for the case $P_{\alpha}$:
\begin{prop}\label{archinterthma}
    For $v\in S_{\infty}$ and a strongly pure weight $\mu$ satisfying the conditions of the combinatorial lemma, the archimedean intertwining operator $T_{st}$ induces a linear isomorphism between the one-dimensional cohomology spaces $H^6(\mathfrak{g}_2(\mathbb{C}),K_{\infty}; \mathcal{J} \otimes \mathcal{M}_{\lambda})$ and $H^6(\mathfrak{g}_2,K_{\infty}; \tilde{\mathcal{J}} \otimes \mathcal{M}_{\lambda})$ such that:
    \[
        T^{q_b}_{st}([\mathcal{J}]_0) \approx_{\mathbb{Q}^{\times}} \frac{L_{\infty}(-\frac{3}{2}, \psi_1\times\psi_2)}{L_{\infty}(1-\frac{3}{2}, \psi_1\times\psi_2)}\frac{L_{\infty}(-2\frac{3}{2}, \omega_{\psi_1\times\psi_2})}{L_{\infty}(1-2\frac{3}{2},\omega_{ \psi_1\times\psi_2})}
\frac{L_{\infty}(-3\frac{3}{2}, (\psi_1\times\psi_2)\otimes \omega_{\psi_1\times\psi_2})}{L_{\infty}(1-3\frac{3}{2}, (\psi_1\times\psi_2)\otimes \omega_{\psi_1\times\psi_2})} [\tilde{\mathcal{J}}]_0.
    \]
\end{prop}

\section{Cohomology of Induced Representations and Manin-Drinfeld Principle}\label{sec:indbs}

\subsection{The Kostant representatives revisited}
Let $P = P_{\beta}$. Fix the following elements of the Weyl group:
\begin{align*}
    w_G &= w_{\alpha\beta\alpha\beta\alpha\beta} = w_{\beta\alpha\beta\alpha\beta\alpha}, &\hfill \textit{longest element of }W;\\
    w_{0} &:= w_{P} = w_{\alpha\beta\alpha\beta\alpha}, &\hfill \textit{longest element of }W^{P};\\
    w_{M} &= w_{\beta}, &\hfill \textit{longest element of }W_{M}.
\end{align*}
Now we can form two self-bijections of $W^{P_{\beta}}$:
\begin{align*}
    w\mapsto w':= w_0 w &: W^{P} \to W^{P},\\ 
    w\mapsto w^{\vee}:= w_{M} w w_{G}&: W^{P} \to W^{P}.
\end{align*}
It is straightforward to verify that these maps are bijections and that they coincide :
\begin{align*}
    w^{\vee}\; =\; w_{M} w w_{G}\;=\;w_{M} w_{G} w \;= \; w_{0}w = w'.
\end{align*}
where second equation follows from the fact that $w_{G}$ is in the center of $W$. For $w^{\tau}\in W_{P^{\tau}_{\beta}}$ the length $l(w'^{\tau}) = 5 - l(w^{\tau})$, which implies :
\[
   l(w^{\tau}) + l(w'^{\tau}) = dim(U_{P_{0,\beta}}).
\]
Moreover, :
\begin{align*}
    l(w'^{\tau})+l(w'^{\bar{\tau}}) &= 5-l(w^{\tau})+5-l(w^{\bar{\tau}})\\
    &= 10 - (l(w^{\tau})+l(w^{\bar{\tau}}))\\
    &= dim(U_{P_{0,\beta}}).
\end{align*}
Hence, if $w$ is balanced, then $w'$ is also balanced. So, we have: 
\begin{lemma}
We have the following:
\begin{enumerate}
    \item The map $w \mapsto w^{\vee}=w':= w_{0}w$ gives a bijection $W^{P_{\beta}}\to W^{P_{\beta}}$.
    \item For each embeddings $\tau:F\to E$, we have $l(w^{\tau})+l(w'^{\tau}) = dim(U_{P_{0}^{\tau}})$.
    \item $w$ is balanced if and only if  $w'$ is balanced.
\end{enumerate}
\end{lemma}
Let $\mu \in X_{00}^{*}(Res_{F/\mathbb{Q}}(T_2) \times E)$ satisfying the conditions of the combinatorial lemma, and $w_{cl}\in W^{P_{\beta}}$ be as defined in previous section, with $\lambda := w_{cl}^{-1}.\mu$ is a dominant weight of $G$. Then:
    \begin{align*}
        w'.\lambda  &= (w_0 w).(w^{-1}.\mu)\\
                    &= (w_0 w).(w^{-1}(\mu+\rho)-\rho)\\
                    &= (w_0 w)(w^{-1}(\mu+\rho))-\rho\\
                    &= w_0.\mu\\
                    &= \tilde{\mu}-5\delta_{2}.
    \end{align*}
Moreover, $\lambda^{\vee} = -w_{G}\lambda = \lambda$ as $w_G$ acts as multiplication by $-1$. Consequently, 
\[
    \mathcal{M}_{\lambda}^{\vee} \simeq \mathcal{M}_{\lambda^\vee} \simeq \mathcal{M}_{\lambda},
\]  
and we have:
\begin{align*}
    w'.\lambda^{\vee} = w'.\lambda = \tilde{\mu}-5\delta_2  .  
\end{align*}
Note that, for any $w\in W^{P_{\beta}}$, we have $w'\neq w$ as they have different lengths. Also, $(w')'= w$ which gave us:
\[
    (w')'.\lambda^{\vee} = w.\lambda = \mu.
\]

\begin{remark}
    The proof above only uses the fact that the Levi factor of $P_{\beta}$ is isomorphic to $GL_2$, and the dimension of the unipotent subgroup is $5$.  All of these properties are also satisfied in case $P_{\alpha}$. Therefore, we can deduce the same results for $P_{\alpha}$ by changing the notation accordingly.  
\end{remark}

\subsection{Induced representations in the boundary cohomology}\label{IndBnd}
In this section,  we will capture the appearance of various induced modules in boundary cohomology in the bottom and top degrees in Propositions \ref{prop:inducedbeta} and \ref{prop:inducedalpha}. 

Let $P:=P_{\beta}$ and follow the notation of Section \ref{sec:comb}. Denote the bottom degree and top degree of the $H^{\bullet}(S^{M},\mathbb{J}_{\lambda}\otimes\mathcal{M}_{\lambda,\mathbb{C}})$ by $b_2^{\Box}$ and $t_2^{\Box}$, respectively. We have the following data:
\begin{align*}
    b_2^{\mathbb{C}} &= 2(2-1)/2 = 1,\\
    t_2^{\mathbb{C}} &= (2^2-1)-1 = 2, \\
    b_2^{F} &= \mathsf{r} b_2^{\mathbb{C}} = \mathsf{r},\\
    t_2^{F} &= dim_{\mathbb{R}}(S^{M})-b_2^{F} = 3\mathsf{r}-1.
\end{align*}

Using the decomposition of the cohomology of parabolic stratum $\partial_{P} S^G$ in Section \ref{boundcomp}:
\begin{align}
            H^{q}(\partial_P S^{G}, \mathcal{M}_{\lambda})=
    \bigoplus_{w\in W^P} {}^{a}\!\operatorname{Ind}_{P(\mathbb{A}_f)}^{ G(\mathbb{A}_f)}(H^{q-l(w)}(S^{M},H^{q}(\mathfrak{u}_P,\mathcal{M}_{w.\lambda})),
\end{align}
and considering the summand indexed by the balanced Kostant representative $w_{cl}\in W^{P}$ obtained from combinatorial lemma:  
\[
{}^{a}\!\operatorname{Ind}_{P(\mathbb{A}_f)}^{ G(\mathbb{A}_f)}(H^{q-l(w_{cl})}(S^{M},H^{q}(\mathfrak{u}_P,\mathcal{M}_{w_{cl}.\lambda})) \subset H^{q}(\partial_P S^{G}, \mathcal{M}_{\lambda}).
\]
First, let us review some facts about $w_{cl}$:
\begin{align*}
    \mu &= w_{cl}.\lambda, \\
    l(w_{cl}) &= \frac{1}{2}\operatorname{dim}(U_{P}) = 5\mathsf{r},\\
    \sigma_f &\in Coh_{!!}(M,\mu).
\end{align*}
By \cite[Proposition 2.5.]{RagTI}, non-vanishing of this summand is equivalent to:
\begin{align*}
    b_2^{F} \leq q-l(w_{cl}) \leq t_2^{F} &\iff
     b_2^{F}+\frac{1}{2}\operatorname{dim}(U_{P})\leq q \leq  t_2^{F}+\frac{1}{2}\operatorname{dim}(U_{P})\\
     &\iff  q_b : = 6\mathsf{r} \leq q \leq 8\mathsf{r}-1 =:q_t.
\end{align*}
\begin{prop}\label{prop:inducedbeta}
    Consider the notation from the previous section. Then,
    \begin{enumerate}
        \item The module $^{a}\!\operatorname{Ind}_{P(\mathbb{A}_f)}^{G(\mathbb{A}_f)}(\sigma_f)$ appears in $H^{q_b}(\partial_{P}S^{G},\mathcal{M}_{\lambda})$.
        \item $^{a}\!\operatorname{Ind}_{P(\mathbb{A}_f)}^{G(\mathbb{A}_f)}(\tilde{\sigma_f}(5))$ appears in $H^{q_b}(\partial_{P}S^{G},\mathcal{M}_{\lambda}).$
    \end{enumerate}  
    The contragredients of these modules appear in cohomology in the top-degree:
    \begin{enumerate}
        \item[3.] $\widetilde{^{a}\!\operatorname{Ind}_{P(\mathbb{A}_f)}^{G(\mathbb{A}_f)}(\sigma_f)}= ^{a}\!\!\operatorname{Ind}_{P(\mathbb{A}_f)}^{G(\mathbb{A}_f)}(\tilde{\sigma}_f(5))$ appears in $H^{q_t}(\partial_{P}S^{G},\mathcal{M}_{\lambda})$.
        \item[4.]  $\widetilde{^{a}\!\operatorname{Ind}_{P(\mathbb{A}_f)}^{G(\mathbb{A}_f)}(\tilde{\sigma}_f(5))} = ^{a}\!\!\operatorname{Ind}_{P(\mathbb{A}_f)}^{G(\mathbb{A}_f)}(\sigma_f)$ appears in $H^{q_t}(\partial_{P}S^{G},\mathcal{M}_{\lambda})$ . 
    \end{enumerate}
\end{prop}
\begin{proof}
   We can locate $^{a}\!\operatorname{Ind}_{P(\mathbb{A}_f)}^{G(\mathbb{A}_f)}(\sigma_f)$ in the decomposition of $H^{q_b}(\partial_{P}S^{G},\mathcal{M}_{\lambda})$ by considering the balanced Kostant representative coming from the combinatorial lemma. The induced representation is assigned to $\sigma_f$ with the highest weight $\mu = w_{cl}.\lambda$ appears as:
   \[
    {}^{a}\!\operatorname{Ind}_{P(\mathbb{A}_f)}^{ G(\mathbb{A}_f)}(H^{\mathsf{r}}(S^{M},H^{q}(\mathfrak{u}_P,\mathcal{M}_{w_{cl}.\lambda}))\subset H^{q_b}(\partial_{P}S^{G},\mathcal{M}_{\lambda}).
   \]
   For part(2), the induced representation of $\tilde{\sigma}_f(5)$ with highest weight $\tilde{\mu}-5\delta_2 = w_{cl}'.\lambda$ give us:
   \[
        {}^{a}\!\operatorname{Ind}_{P(\mathbb{A}_f)}^{ G(\mathbb{A}_f)}(H^{\mathsf{r}}(S^{M},H^{q}(\mathfrak{u}_P,\mathcal{M}_{w_{cl}'.\lambda}))\subset H^{q_b}(\partial_{P}S^{G},\mathcal{M}_{\lambda}).
   \]
   To go from bottom-degree to top-degree, we need to look at the contravariants of these modules. Then, $w'_{cl}$ and $(w'_{cl})' = w_{cl}$ will take care of the cases (3) and (4), respectively.
\end{proof}

We can translate Proposition \ref{prop:inducedbeta} to the case $P_{\alpha}$, where we get the same numerology with the same line of proof.
\begin{prop}\label{prop:inducedalpha}
    Consider the notation from the previous section and $P=P_{\alpha}$. We have for cohomology bottom degree-degree:
    \begin{enumerate}
        \item The module $^{a}\!\operatorname{Ind}_{P(\mathbb{A}_f)}^{G(\mathbb{A}_f)}(\sigma_f)$ appears in $H^{q_b}(\partial_{P}S^{G},\mathcal{M}_{\lambda})$.
        \item $^{a}\!\operatorname{Ind}_{P(\mathbb{A}_f)}^{G(\mathbb{A}_f)}(\tilde{\sigma_f}(3))$ appears in $H^{q_b}(\partial_{P}S^{G},\mathcal{M}_{\lambda})$.
    \end{enumerate}  
    The contragredients of these modules appear in cohomology in the top-degree:
    \begin{enumerate}
        \item[3.] $\widetilde{^{a}\!\operatorname{Ind}_{P(\mathbb{A}_f)}^{G(\mathbb{A}_f)}(\sigma_f)}= ^{a}\!\!\operatorname{Ind}_{P(\mathbb{A}_f)}^{G(\mathbb{A}_f)}(\tilde{\sigma}_f(3))$ appears in $H^{q_t}(\partial_{P}S^{G},\mathcal{M}_{\lambda})$.
        \item[4.]  $\widetilde{^{a}\!\operatorname{Ind}_{P(\mathbb{A}_f)}^{G(\mathbb{A}_f)}(\tilde{\sigma}_f(3))} = ^{a}\!\!\operatorname{Ind}_{P(\mathbb{A}_f)}^{G(\mathbb{A}_f)}(\sigma_f)$ appears in $H^{q_t}(\partial_{P}S^{G},\mathcal{M}_{\lambda})$ .
    \end{enumerate}
\end{prop}

\subsection{Cohomology of induced representations}\label{sec:cohinbound}
Let $\sigma_f\in Coh_{!!}(M_{\beta},\mu)$ with strongly pure weight $\mu$ such that there is $w_{cl}\in W^{P_{\beta}}$ obtained from combinatorial lemma. We can define the dominant weight $\lambda = w_{cl}^{-1}.\mu$. As we stated in Section \ref{sec:Nonint}, for any unramified finite place $v$ there is a nontrivial intertwining operator 
\[
    T_{st,E}: ^{a}\!\operatorname{Ind}_{P_{\beta}(F_{v})}^{G(F_{v})}(\sigma_{f,v}) \to ^{a}\!\operatorname{Ind}_{P_{\beta}(F_{v})}^{G(F_{v})}(\widetilde{\sigma}_{f,v}(5)),
\]
which implies $^{a}\!\operatorname{Ind}_{P_{\beta}(\mathbb{A}_{f})}^{G(\mathbb{A}_{f})}(\sigma_{f})$ and $^{a}\!\operatorname{Ind}_{P_{\beta}(\mathbb{A}_{f})}^{G(\mathbb{A}_{f})}(\widetilde{\sigma}_{f}(5)) $ are equivalent almost everywhere. We want to show that this is the only nontrivial intertwining operator involving these two representations. The following proposition is the generalization of the proof presented in \cite[Proposition 2.3.1.]{mundy}
from group $G_2$ over a totally real number field to the case of a totally imaginary number field. 

\begin{prop}\label{9}
    Let $\sigma_{\alpha,f}$ and $\sigma_{\beta,f}$ be a unitary, tempered, cuspidal automorphic representations of $M_{\alpha}(\mathbb{A}_f)$ and $M_{\beta}(\mathbb{A}_f)$. Also, let $\eta$ be a quasi-character of $T(\mathbb{Q})\backslash T(\mathbb{A})$, and let $s_{\alpha} , s_{\beta} \in \mathbb{C}$. Consider irreducible sub-quotients $\Pi_{0}, \Pi_{\alpha}$ and $\Pi_{\beta}$ of $I_{B(\mathbb{A}_f)}^{G(\mathbb{A}_f)}(\eta)$, $I_{P_{\alpha}(\mathbb{A}_f)}^{G(\mathbb{A}_f)}(s_{\alpha},\sigma_{\alpha})$ and $I_{P_{\beta}(\mathbb{A}_f)}^{G(\mathbb{A}_f)}(s_{\beta}, \sigma_{\beta})$, respectively. Then, none of $\Pi_{0}$, $\Pi_{\alpha}$ and $\Pi_{\beta}$ are equivalent almost everywhere to each other.
\end{prop}
\begin{proof}
    We may assume $Re(s_{\alpha})$ and $Re(s_{\beta})$ are non-negative, which is not restrictive as we can consider the inverse intertwining map instead. Let $\Pi$ be an automorphic representation of $G_2(\mathbb{A})$ and $Re(s)\geq 0$. Also, let $R_{7}$ be the standard representation of $G_2$, and $\chi$ be a finite order character of $\mathbb{A}^{\times}$. Let $v\notin S$ and $\vartheta_{v} \in G_2(\mathbb{C})$ be the Satake parameter of $\Pi_{v}$, we can define the local L-factor at $v$:
    \[
        L_{v}(s,R_{7}(\Pi)\times \chi) = \det(1-\chi_{v}(p_v)q_{v}^{-s}R_{7}(\vartheta_{v}))^{-1},
    \]
    and the global L-function
    \[
        L^{S}(s,R_{7}(\Pi)\times \chi) = \prod_{v\notin S} L_{v}(s,R_{7}(\Pi)\times \chi).
    \]
    For large enough $S$,
    \begin{align*}
        L^{S}(s,R_{7}(\Pi_{\alpha}) \times \chi) =& L^{S}(s+s_{\alpha}, \sigma_{\alpha}\otimes \chi)L^{S}(s-s_{\alpha}, \tilde{\sigma}_{\alpha}\otimes \chi)\\& .L^{S}(s+s_{\alpha},\omega_{\alpha}\chi)L^{S}(s-s_{\alpha},\omega_{\alpha}^{-1}\chi)L^{S}(s,\chi), \numberthis \label{alpha}
    \end{align*}
   where $\omega_{\alpha}$ is the central character of $\sigma_{\alpha}$. Moreover,
    \[
         L^{S}(s,R_{7}(\Pi_{\beta}) \times \chi) = L^{S}(s+s_{\beta}, \sigma_{\beta}\otimes \chi)L^{S}(s-s_{\beta}, \tilde{\sigma}_{\beta}\otimes \chi)L^{S}(s,Ad^2(\sigma_{\beta})\times \chi).\numberthis \label{beta}
    \]

    In \eqref{alpha}, the possible poles come from the Hecke L-functions on the right hand side. First, if $\chi = 1$, then $L^{S}(s,\chi)$ has a pole at $s=1$. Also, if $\chi=\omega_\alpha$, $L^{S}(s-s_{\alpha},\omega_{\alpha}^{-1}\chi)$ has a pole at $s=1+s_{\alpha}$. Finally, the other L-functions are holomorphic for $Re(s)>1$, and all of the L-functions in \eqref{alpha} are non-zero for $Re(s)\geq 1$. Therefore, $L^{S}(s,R_{7}(\Pi_{\alpha}) \times \chi)$ has a simple pole at $1+s_{\alpha}$, if $s_{\alpha} \neq 0$ and has a double pole at $s = 1$ otherwise.
    
    In \eqref{beta}, since $\sigma_{\beta}$ is cuspidal and unitary, by Gelbart-Jacquet $L^{S}(s,Ad^{2}(\sigma_{\beta})\times \chi)$ has at most a simple pole at $s=1$. Moreover, the other two L-functions on the right hand side of \eqref{beta} are both holomorphic for $Re(s)>0$ and are non-zero for $s=1$. Therefore, $L^{S}(s, R_{7}(\Pi_{\beta})\times\chi)$ at most has a simple pole at $s=1$. 
    
    The discrepancy between the poles of $L^S(s,R_7(\Pi_\alpha)\times\chi)$ and $L^S(s,R_7(\Pi_\beta)\times\chi)$ implies that $\Pi_\alpha$ and $\Pi_\beta$ are not equivalent almost everywhere.

    To deal with representations induced from Borel subgroup, consider the degree 14 L-function $L^{S}(s,\Pi\times \sigma, R_{7}\otimes \rho_2)$, where $\sigma$ is a representation of $GL_2(\mathbb{A}_f)$ and $\rho_2$ is the standard representation of $GL_2$. We can form
    \begin{align*}
        L^{S}(s,\Pi_{\beta}\times \sigma_{\beta}, R_{7}\otimes \rho_2) =& L^{S}(s+s_{\beta},\sigma_{\beta}\times \sigma_{\beta}).L^{S}(s-s_{\beta},\sigma_{\beta}^{\vee} \times \sigma_{\beta})\\&.L^{S}(s,\sigma_{\beta}, Ad^3)L^{S}(s,\sigma_{\beta}).\numberthis \label{beta-borel}
    \end{align*}
    First, $L^{S}(s-s_{\beta},\sigma_{\beta}\times \tilde{\sigma}_{\beta})$ has a pole at $s = 1+s_{\beta}$ and all the other L-functions are non-vanishing at $s=1+s_{\beta}$. If $s_{\beta} = 0$, $L^{S}(s,\sigma_{\beta}, Ad^3)$ can have a simple pole at $s=1$ if $\sigma_{\beta}$ is monomial but it is still non-vanishing at $s=1$ \cite{KSannals}. Therefore,  $L^{S}(s,\Pi_{\beta}\times \sigma_{\beta}, R_{7}\otimes \rho_2)$ at least has one simple pole at $s=1+s_{\beta}$.
    
    However,  $L^{S}(s,\Pi_{0}\times \sigma_{\beta}, R_{7}\otimes \rho_2)$ is just the product of seven L-functions of various character twists of $\sigma_{\beta}$ which are all entire as $\sigma_{\beta}$ is cuspidal. This concludes that $\Pi_0$ and $\Pi_{\beta}$ are not equivalent almost everywhere.

    In the same way, by considering  $L^{S}(s,\Pi_{\alpha}\times \sigma_{\alpha}, R_{7}\otimes \rho_2)$ we can show that $\Pi_0$ is not equivalent almost everywhere to $\Pi_{\alpha}$.
\end{proof}

Let $\sigma_f$ and $\sigma'_f$ be tempered unitary representations of $M_{\beta}(\mathbb{A}_f)$, and $s,s'>0$. For any irreducible sub-quotients $\Pi$ and $\Pi'$ of $I_{P_{\beta}(\mathbb{A}_f)}^{G(\mathbb{A}_f)}(s,\sigma)$ and $I_{P_{\beta}(\mathbb{A}_f)}^{G(\mathbb{A}_f)}(s',\sigma')$. By \cite[Proposition 2.3.3.]{mundy}, if $\Pi$ is equivalent to $\Pi'$ almost everywhere then $s=s'$ and $\sigma_f \simeq \sigma'_f$. Using the same reasoning, we can derive similar results for $M_{\alpha}$ as well.

\subsection{Manin-Drinfeld principle}\label{sec:MD}
Let $P=P_{\beta}$, and $S$ be a finite set of places as usual. let $\mathcal{H}^{G,S} = \otimes_{p\notin S} \mathcal{H}^G$, same as \cite[Section 2.3.5.]{HR}. Assume $^{a}\!\operatorname{Ind}_{P(\mathbb{A}_f)}^{G(\mathbb{A}_f)}(\sigma_f)$ has nonzero $K_f$-fixed vectors with dimension $\mathsf{k}$, denoted by:
\[
    I_{b}^{S}(\sigma_f)_{w_{cl}} := ^{a}\!\operatorname{Ind}_{P(\mathbb{A}_f)}^{G(\mathbb{A}_f)}(H^{b_2^{F}}(\underline{S}^{M}, \mathcal{M}_{w_{cl}.\lambda}(\sigma_f))^{K_f} .
\]
Similarly,
\[
    I_{b}^{S}(\tilde{\sigma_f}(5))_{w'_{cl}} := ^{a}\!\operatorname{Ind}_{P(\mathbb{A}_f)}^{G(\mathbb{A}_f)}(H^{b_2^{F}}(\underline{S}^{M}, \mathcal{M}_{w'_{cl}.\lambda}(\tilde{\sigma_f}(5)))^{K_f}.
\]
We can define this $K_f$-fixed spaces for contragredient modules in top-degree:
\[
    I_{t}^{S}(\sigma_f)_{w'_{cl}} := ^{a}\!\operatorname{Ind}_{P(\mathbb{A}_f)}^{G(\mathbb{A}_f)}(H^{t_2^{F}}(\underline{S}^{M}, \mathcal{M}_{w'_{cl}.\lambda}(\sigma_f))^{K_f},
\]
and,
\[
    I_{t}^{S}((\tilde{\sigma_f}(5)))_{w_{cl}} := ^{a}\!\operatorname{Ind}_{P(\mathbb{A}_f)}^{G(\mathbb{A}_f)}(H^{t_2^{F}}(\underline{S}^{M}, \mathcal{M}_{w_{cl}.\lambda}(\tilde{\sigma_f}(5)))^{K_f}.
\]
Now, we proceed to state and prove a strong form of the Manin-Drinfield principle for our case:

\begin{theorem}[Mannin-Drinfeld Principle for the case $P_{\beta}$]\label{thm:mnnb}
Using the notation introduced above.
\begin{enumerate}
    \item The $\mathcal{H}^{G,S}$-modules $I_{b}^{S}(\sigma_f)_{w_{cl}}$, $I_{b}^{S}(\tilde{\sigma_f}(5))_{w'_{cl}}$, $I_{t}^{S}(\sigma_f)_{w'_{cl}}$, and $I_{t}^{S}((\tilde{\sigma_f}(5)))_{w_{cl}}$ are all finite dimensional $E$-vector spaces of the same dimension which is denoted by $\mathsf{k}$.
    
    \item The direct sum
    \[
        I_{b}^{S}(\sigma_f)_{w_{cl}}\oplus I_{b}^{S}(\tilde{\sigma_f}(5))_{w'_{cl}}
    \]
    is a $2\mathsf{k}$-dimensional $E$-vector space that is isotypic in $H^{q_b}(\partial \underline{S}^{G},\mathcal{M}_{\lambda})^{K_f}$. There is also a $\mathcal{H}^{G,S}$-equivariant projection:
    \[
        \mathfrak{R}^b_{\sigma_f}: H^{q_b}(\partial S^{G},\mathcal{M}_{\lambda}) \to I_{b}^{S}(\sigma_f)_{w_{cl}}\oplus I_{b}^{S}(\tilde{\sigma_f}(5))_{w'_{cl}}
    \]
    
    \item The sum
    \[
        I_{t}^{S}(\tilde{\sigma_f}(5))_{w'_{cl}}\oplus I_{t}^{S}(\sigma_f)_{w_{cl}}
    \]
    is a $2\mathsf{k}$-dimensional $E$-vector space that is isotypic in $H^{q_t}(\partial \underline{S}^{G},\mathcal{M}_{\lambda})^{K_f}$. There is also a $\mathcal{H}^{G,S}$-equivariant projection:
    \[
        \mathfrak{R}^t_{\sigma_f}: H^{q_t}(\partial S^{G},\mathcal{M}_{\lambda}) \to I_{t}^{S}(\sigma_f)_{w_{cl}}\oplus I_{t}^{S}(\tilde{\sigma_f}(5))_{w'_{cl}}
    \]
\end{enumerate}
\end{theorem}

\begin{proof}To prove the first part, for any place $v$ we need to show:
    \[
        dim(^{a}\operatorname{Ind}^{G_v}_{P_v}(\sigma_{v})^{K_v}) = dim(^{a}\operatorname{Ind}^{G_v}_{P_v}(\tilde{\sigma}_{v}(5))^{K_v}).
    \]
    Using Iwahori decomposition, $K_{v}\cap P_v = (K_v \cap M_{P,v}).(K\cap U_{P,v})$. Moreover, $\sigma_v$ and $\tilde{\sigma}_v(5)$ are representations of $M_{P,v}$ which we algebraically induced to $P_v$. Also  $\sigma_v \simeq \tilde{\sigma}_v(5)$ and we have:
    \begin{equation}
        \dim (V_{\sigma_v}^{K_v\cap P_v}) = \dim (V_{\sigma_v}^{K_v\cap M_{P,v}}) = \dim (V_{\tilde{\sigma}_v(5)}^{K_v\cap M_{P,v}}) = \dim (V_{\tilde{\sigma}_v(5)}^{K_v\cap P_v}).
    \end{equation}
    Now by using Frobenius reciprocity and Mackey Theory, for $\tau$ is either $\sigma_v$ or $\tilde{\sigma}_v(5)$ we have:
    \begin{equation}
        Hom_{G}(\operatorname{Ind}_{K_v}^{G_v}(\mathbf{1}), ^{a}\operatorname{Ind}_{P_v}^{G_v}(\tau)) = \bigoplus_{x\in P_v \backslash G_v/K_v} Hom_{K_v\cap x^{-1}P_{v} x} (\mathbf{1}, \tau^{x}),
    \end{equation}
    The notation $\operatorname{Ind}_{K_v}^{G}(\mathbf{1})$ stands for the compact induction of the trivial representation of $K_v$ to $G_v$. By conjugating the right hand side by $x$, we will have :
    \begin{multline*}
                Hom_{G}(\operatorname{Ind}_{K_v}^{G_v}(\mathbf{1}), ^{a}\operatorname{Ind}_{P_v}^{G_v}(\tau)) = \bigoplus_{x\in P_v \backslash G_v/K_v} Hom_{x K_{v} x^{-1} \cap P_v} (\mathbf{1}, \tau)\\ = \bigoplus_{x\in P_v \backslash G_v/K_v} V_{\tau}^{x K_{v} x^{-1} \cap P_v}.
    \end{multline*}

    By Iwasawa decomposition $G_v = P_v.G_v(\mathcal{O}_v)$, for any representative $x\in P_v \backslash G_v/K_v$ from $G(\mathcal{O}_v)$ we will have:
    \begin{equation}
        \dim(^{a}\operatorname{Ind}^{G_v}_{P_v}(\tau)^{K}) = |P_v \backslash G_v/K_v| \dim (V_{\tau}^{K_v\cap P_v}).
    \end{equation}

    Part one of the theorem implies that  $I_{b}^{S}(\sigma_f)_{w_{cl}}\oplus I_{b}^{S}(\tilde{\sigma_f}(5))_{w'_{cl}}$ is a $2\mathsf{k}$-dimensional $E$-vector space. Also, by Proposition \ref{prop:inducedbeta} and \ref{9} we can see that $I_{b}^{S}(\sigma_f)_{w_{cl}}\oplus I_{b}^{S}(\tilde{\sigma_f}(5))_{w'_{cl}}$ is isotypic in  $H^{q_b}(\partial \underline{S}^{G},\mathcal{M}_{\lambda})^{K_f}$, which implies part (2). The third part can be shown similarly to the second part using Proposition \ref{prop:inducedbeta}.
\end{proof}

\begin{remark}
    Again, all of the results above can be easily translated to the case $P_{\alpha}$. The statement of the Manin-Drinfeld principle for the case $P_{\alpha}$, is the same as the case $P_{\beta}$ changing the group notation to the one coming from $P_{\alpha}$ and changing the Tate twist from $5$ to $3$.
\end{remark}

\section{Eisenstein cohomology revisited}
\subsection{Poincare duality}
Following \cite[Section 6.1.]{HR}, we have the following Poincare duality pairing for sheaf cohomology on $S^{G}$:
 \[
    H^{\bullet}(S^{G},\mathcal{M}_{\lambda}) \times H_{c}^{\mathsf{d}-\bullet}(S^{G},\mathcal{M}_{\lambda^{\vee},E}) \to E.
 \]
 The dimension of the boundary is $\operatorname{dim}_{\mathbb{R}}(\partial S^{G}) = \mathsf{d}-1 = q_b+q_t$. We can form the Poincare duality for the boundary as well:
 \[
    H^{\bullet}(\partial S^{G},\mathcal{M}_{\lambda}) \times H^{\mathsf{d}-1-\bullet}(\partial S^{G},\mathcal{M}_{\lambda^{\vee},E}) \to E.
 \]
 Using the fundamental exact sequence \ref{les}, we have the following maps:
 \begin{align*}
     H^{\bullet}(S^{G},\mathcal{M}_{\lambda}) \xrightarrow[]{\mathfrak{r}^{*}} H^{\bullet}(\partial S^{G},\mathcal{M}_{\lambda}),\\
     H^{\mathsf{d}-1-\bullet}(\partial S^{G},\mathcal{M}_{\lambda}) \xrightarrow[]{\mathfrak{d}^{*}} H_{c}^{\mathsf{d}-\bullet}(S^{G},\mathcal{M}_{\lambda}) .
 \end{align*}
Therefore, for any $\xi \in  H^{\bullet}(S^{G},\mathcal{M}_{\lambda^{\vee},E})$ and $\varsigma\in H^{\mathsf{d}-1-\bullet}(\partial S^{G}_{K_f},\mathcal{M}_{\lambda^{\vee},E})$:
\[
    (\mathfrak{r}^{*}(\xi),\varsigma) = (\xi,\mathfrak{d}^{*}(\varsigma)).
\]
For any $\mathfrak{r}^{*}(\xi)\in H^{\bullet}_{Eis}(\partial S^{G}, \mathcal{M}_{\lambda, E})$ and any $\mathfrak{r}^{*}(\varsigma)\in H^{\mathsf{d}-1-\bullet}_{Eis}(\partial S^{G}, \mathcal{M}_{\lambda, E})$, using Poincare duality pairing we have:
\[
    (\mathfrak{r}^{*}(\xi),\mathfrak{r}^{*}(\varsigma))=(\xi,\mathfrak{d}^{*}\mathfrak{r}^{*}(\varsigma)) = (\xi,0)=0.
\]
By the non-degeneracy of the Poincare pairing $\xi \in Ker(\mathfrak{d}^{*})=Im(\mathfrak{r}^{*})$ and $\xi \in H^{\bullet}_{Eis}(\partial S^{G}, \mathcal{M}_{\lambda, E})$, we have:
\[
    H^q_{Eis}(\partial S^{G}, \mathcal{M}_{\lambda, E}) = H^{\mathsf{d}-1-q}_{Eis}(\partial S^{G}, \mathcal{M}_{\lambda, E})^{\perp}.
\]
So under the Poincare duality pairing, the Eisenstein cohomology is a maximal isotropic subspace of the boundary cohomology.

\subsection{Rank-one Eisensten cohomology}
Now we can project the bottom or top-degree cohomology of the limit of the symmetric space into an isotypic component that corresponds to them:
\begin{align*}
\begin{tikzcd}
    H^{q_b}(S^{G},\mathcal{M}_{\lambda}) \arrow[d,"\mathfrak{r}^{*}"]\\
    H^{q_b}(\partial_{P}S^{G},\mathcal{M}_{\lambda})\arrow[d,"\mathfrak{R}^b_{\sigma_f}"]\\ 
    I_{b}^{S}(\sigma_f)_{w_{cl}}\oplus I_{b}^{S}(\tilde{\sigma_f}(5))_{w'_{cl}}
\end{tikzcd} &,&
\begin{tikzcd}
    H^{q_t}(S^{G},\mathcal{M}_{\lambda}) \arrow[d,"\mathfrak{r}^{*}"]\\
    H^{q_t}(\partial_{P}S^{G},\mathcal{M}_{\lambda})\arrow[d,"\mathfrak{R}^t_{\sigma_f}"]\\ 
    I_{t}^{S}(\tilde{\sigma_f}(5))_{w'_{cl}}\oplus I_{t}^{S}(\sigma_f)_{w_{cl}}
\end{tikzcd} .
\end{align*}
Fix the following notation:
\begin{align}\label{frakj}
    \mathfrak{J}^{b}(\sigma_f) &:= \mathfrak{R}^b_{\sigma_f}(H^{q_b}_{Eis}(\partial S^{G}, \mathcal{M}_{\lambda, E})),\\
    \mathfrak{J}^{t}(\sigma_f)^{\vee} &:= \mathfrak{R}^t_{\sigma_f}(H^{q_t}_{Eis}(\partial S^{G}, \mathcal{M}_{\lambda, E})).
\end{align}
Therefore, we can state the main results on rank-one Eisenstein cohomology in our case:
\begin{theorem}\label{thm:eis}
    Let have the notation as Theorem \ref{thm:mnnb}, then:
    \begin{enumerate}
        \item $\mathfrak{J}^{b}(\sigma_f)$ is an $E$-subspace of $I_{b}^{S}(\sigma_f)_{w_{cl}}\oplus I_{b}^{S}(\tilde{\sigma_f}(5))_{w'_{cl}}$ of  dimension $\mathsf{k}$.
        \item $\mathfrak{J}^{t}(\sigma_f)^{\vee}$ is an $E$-subspace of $I_{t}^{S}(\tilde{\sigma_f}(5))_{w'_{cl}}\oplus I_{t}^{S}(\sigma_f)_{w_{cl}}$ of dimension $\mathsf{k}$.
    \end{enumerate}
\end{theorem}

To prove this theorem, we first need to establish the cohomological interpretation of Langlands' constant theorem in our case. This will allow us to construct the cohomology classes in $\mathfrak{J}^b(\sigma_f)$ and $\mathfrak{J}^t(\sigma_f)^\vee$.

Fix an embedding $\iota:E \to \mathbb{C}$, and once again with the abuse of notation let $\sigma := {^{\iota}}\!\sigma$ be a cuspidal automorphic representation of $M(\mathbb{A}) := M_{\beta}(\mathbb{A})$  and $\sigma_f$ be its finite part. In the previous subsection, we showed that at the evaluation point $I^{G}_{P}(s,\sigma)$ equals to the algebraic induction $^{a}\!\operatorname{Ind}^{G(\mathbb{A})}_{P(\mathbb{A})}(\sigma)$ where its finite part appears in the boundary cohomology. We can repeat the same argument for $I^{G}_{P}(-s,\tilde{\sigma}(5))$ and $^{a}\!\operatorname{Ind}^{G(\mathbb{A})}_{P(\mathbb{A})}(\tilde{\sigma}(5))$ at evaluation point.

By  Langlands constant term theorem, we know that to find the poles of the Eisenstein series we should calculate its constant term along $P$. Then forming the constant term along $P$ for any $\phi \in \mathcal{C}^{\infty}(G(\mathbb{Q})\backslash G(\mathbb{A}))$:
\begin{equation}\label{def:constant}
    \mathcal{F}_{P}(\phi)(\underline{g}) := \int_{U_{P}(\mathbb{Q})\backslash U_{P}(\mathbb{A})} \phi(\underline{u}\underline{g})d\underline{u},
\end{equation}
where underline indicates an element of an adelic group, provides us with a map
\[\mathcal{F}_{P} : \mathcal{C}^{\infty}(G(\mathbb{Q})\backslash G(\mathbb{A})) \to \mathcal{C}^{\infty}(P(\mathbb{Q})\backslash G(\mathbb{A})).\] 

\begin{remark}
    Note that the global measure $d\underline{u}$ on $U(\mathbb{A})=U_{P_0}(\mathbb{A}_F)$ that appeared in the integral form of the intertwining operator and the constant term, is a product over the coordinates of $U_{P_0}$ of the additive measure $d\underline{x}$ on $\mathbb{A}_F$ where $d\underline{x} = \prod_v dx_v$ is a product of local additive measures $dx_v$ on $F_{v}$. As we mentioned for the intertwining case, for finite place $v$ the measure $dx_v$ is normalized by $ vol(\mathcal{O}_v) = 1$. For archimedean place $v$, the natural choice for $dx_v$  is nothing but $2|dx_1 dx_2|$ where $dx_i$ is the ordinary Lebesgue measure on $\mathbb{R}$, but we want to normalize it by $vol(U_{P_0}(F)\backslash U_{P_0}(\mathbb{A}_F))=1$. Then
    \[vol_{d\underline{u}}(U_{P_0}(F)\backslash U_{P_0}(\mathbb{A}_F)) = vol_{d\underline{x}}(F\backslash \mathbb{A}_F)^{dim(U_{P_0})} = vol_{d\underline{x}}(F\backslash \mathbb{A}_F)^{5} = |\delta_{F/\mathbb{Q}}|^{5/2},\]
    where $\delta_{F/\mathbb{Q}}$ is the absolute discriminant of $F$. Moreover, $|\delta_{F/\mathbb{Q}}|^{1/2}\in \mathbb{R}^{\times}/\mathbb{Q}^{\times}$ is independent of the enumeration and the choice of basis. We will define the global measure and by the abuse of notation, we set the measure $d\underline{u}$ as:
\[
     |\delta_{F/\mathbb{Q}}|^{-5/2}d\underline{u}.
\]
\end{remark}

\begin{theorem}[Langlands' Constatnt Term]
    For $f\in I^{G}_P (s,\sigma)$, we have:
    \[
        \mathcal{F}_{P}(Eis_{P}(s,f)) = f+T_{st}(s)f.
    \]
\end{theorem}
For $v\notin S$, the vector $f_v = f_{v}^{0}$ is the normalized spherical vector, then by \cite{ShBook} and remark above, we have a cohomological constant term formula in the case $P=P_{\beta}$ as:
\begin{equation}\label{intertwinformula}
    \mathcal{F}_{P}\circ Eis_P(s,f) = |\delta_{F/\mathbb{Q}}|^{-5/2}\frac{L^{S}(s,Ad^{3}(\sigma))L^{S}(2s,\omega_{\sigma})}{L^{S}(1+s,Ad^{3}(\sigma))L^{S}(1+2s,\omega_{\sigma})} \bigotimes_{v\in S}T_{st,v}(s)(f_v)\otimes\bigotimes_{v\notin S}\tilde{f}^{0}_{v}.
\end{equation}

We need to study the holomorphy of the Eisenstein series at the point of the evaluation. Fix $\iota:E \to \mathbb{C}$. Let $\mu \in X_{00}^{*}(T_2\times E)$ with the purity weight $\mathbf{w}$ and $\sigma_f \in Coh_{!!}(M,\mu)$. The weight $\mu$ is said to be on the right side of the unitary axis if $\mathbf{w} \leq -5$, which implies that:
\[
\begin{cases}
-\frac{5}{2}-\frac{\mathbf{w}}{2} &\geq 0\\
-5-\mathbf{w} &\geq 0
\end{cases} \iff \begin{cases}
-\frac{5}{2}-a_{1}(\mu^{v}) &\geq 0,\\
-5-a_{2}(\mu^{v}) &\geq 0.
\end{cases}
\]

Moreover, $^{\iota}\sigma = ^{\iota}\sigma_{u}\otimes |\;|^{-\mathbf{w}/2}$ for a unitary cuspidal representation $^{\iota}\sigma^u$. Then
\[
    L^{S}(-\frac{5}{2},Ad^{3}(\sigma))L^{S}(-5,\omega_\sigma) = L^{S}(-\frac{5}{2}-a_1(\mu),Ad^{3}(\sigma^u))L^{S}(-5-a_2(\mu),\omega_{\sigma^u}).
\]
Now if $\mu$ is on the right of the unitary axis, we have:
\[
    L^{S}(1-\frac{5}{2},Ad^{3}(\sigma))L^{S}(1-5,\omega_\sigma) = L^{S}(1-\frac{5}{2}-a_1(\mu),Ad^{3}(\sigma^u))L^{S}(1-5-a_2(\mu),\omega_{\sigma^u}) \neq 0,
\]
since $1-\frac{5}{2}-a_1(\mu)\geq 1$ and $1-5-a_2(\mu)\geq 1$, \cite{KSannals}.

\begin{prop}
    Let the conditions in the combinatorial lemma hold, and suppose $\mu$ is on the right side of the unitary axis. Then $Eis_{P}(s,f)$ is holomorphic at $s=-\frac{5}{2}$.
\end{prop}
\begin{proof}
The Eisenstein series may have a possible pole at $s=\frac{1}{2}$. In this case, $L^{S}(1,\omega_{\sigma^{u}})$ is a pole only when the central character is trivial. But this means that $\ell_2(\mu) = 0$ which is not possible as the combinatorial lemma forces $\ell_2(\mu)$ to be greater than or equal to 1.    
\end{proof}
\begin{proof}[Proof of Theorem \ref{thm:eis}]
The proof for the top degree follows the same path as that of the bottom degree, so we only prove part 1. First, assume that the highest weight $\mu$ is on the right side of the unitary axis.  Fix $\iota: E \to \mathbb{C}$, it is enough to prove:
\[
 dim_{\mathbb{C}}(\mathfrak{J}^{b}(^{\iota}\sigma_f)) = \mathsf{k}.
\]
We showed that there is $^{\iota}\sigma_{\infty}$, where $^{\iota}\sigma = ^{\iota}\sigma_{f}\times ^{\iota}\sigma_{\infty}$ is a cuspidal automorphic representation of $GL_2(\mathbb{A})$. Now by tensoring the induced modules appearing in part 2 of the Manin-Drinfeld theorem and their relation to relative lie algebra cohomology, we can form the cohomological counterpart of the intertwining operators by applying the functor $H^{b}(\mathfrak{g}(\mathbb{R}), \overline{K}_{\infty},-\otimes \mathcal{M}_{^{\iota}\lambda})$ as:
\[
    I_{b}^{S}(^{\iota}\sigma_f)_{w_{cl}} \xrightarrow[]{T_{st}(5/2,\;^{\iota}\sigma)^{\bullet}}  I_{b}^{S}(^{\iota}\tilde{\sigma}_f)_{w'_{cl}}.
\]
The arithmetic counterpart of this intertwining operator is called Eisenstein intertwining operator denoted by $T_{Eis}$ where:
\[
    T_{Eis} \otimes_{\iota, E} \mathbb{C} = T_{st,\iota}.
\]
Applying Langlands constant term, we can form the map $\mathcal{F}_{P}^{b}$ and also the Eisenstein series $Eis_{P}^{b}$ attached to this cohomological groups at the bottom degree. The Manin-Drinfeld theorem implies that the image of $\mathcal{F}_{P}^{b}\circ Eis_{P}^{b}$  contains in the $\mathfrak{J}^{b}(^{\iota}\sigma_f)$. So we can rewrite it as:
\[
    \mathfrak{J}^{b}(^{\iota}\sigma_f) \supset \{(\xi, \xi+T_{st,\iota}(-\frac{5}{2},\;^{\iota}\sigma))\; ; \; \xi\in I_{b}^{S}(^{\iota}\sigma_f)_{w_{cl}}\},
\]
or equivalently:
\[
    \mathfrak{J}^{b}(\sigma_f) \supset \{(\xi, \xi+T_{Eis}(-\frac{5}{2},\;\sigma))\; ; \; \xi\in I_{b}^{S}(\sigma_f)_{w_{cl}}\}.
\]
As we showed the intertwining operators at the point of evaluation are non-vanishing in all the places, and the Manin-Drinfeld theorem we will have:
\[
    dim_{E}(\mathfrak{J}^{b}(\sigma_f)) = dim_{\mathbb{C}}(\mathfrak{J}^{b}(^{\iota}\sigma_f)) \geq dim_{\mathbb{C}}(I_{b}^{S}(^{\iota}\sigma_f)_{w_{cl}}) = \mathsf{k} .
\]
If we find ourselves on the left side of the unitary axis, we can examine its contragredient counterpart located on the right side of the unitary axis. Applying the same argument yields equivalent results. Now, applying the same argument as in \cite[6.2.2.2]{HR} through the Poincaré pairing, we conclude the proof by demonstrating that:
\[
dim_{E}(\mathfrak{J}^{b}(\sigma_f)) \leq \mathsf{k}.
\]
\end{proof}

\subsection{Galois equivariancy}
In this section we generalize the Galois equivariancy results in \cite[Section 5.3.2.5.]{RagTI}. We prove for $P=P_{\beta}$ and the case $P_{\alpha}$ following the same line of proof. The field $F$ is $CM$-type, therefore strongly pure weights in $X_{00}^{*}(Res_{F/\mathbb{Q}}(T_2)\times E)$ are the base change of the strongly pure weights of $X_{00}^{*}(Res_{F_1/\mathbb{Q}}(T_2)\times E)$ \cite[Cor.3.14.]{RagTI}. Fix an embedding $\iota: E\to \mathbb{C}$. Let $\gamma \in Gal(\bar{\mathbb{Q}}/\mathbb{Q})$; it induces a map over the cohomology of unipotent subgroups:
\[
    \gamma_{*}: H^{q}(\mathfrak{u}_P,\mathcal{M}_{^{\iota}\lambda,\mathbb{C}})(^{\iota}w) \to H^{q}(\mathfrak{u}_P,\mathcal{M}_{^{\gamma\circ\iota}\lambda,\mathbb{C}})(^{\gamma\circ \iota}w),
\]
where $q$ is the length of the Weyl group's element $l(^{\iota}w) = l(^{\gamma\circ \iota}w)$. As we showed earlier, by Kostant results $H^{q}(\mathfrak{u}_P,\mathcal{M}_{^{\iota}\lambda,\mathbb{C}})(^{\iota}w) \simeq \mathcal{M}_{^{\iota}w.^{\iota}\lambda,\mathbb{C}}$. Now applying Schur's lemma, we can understand the action of the Galois element on the cohomology of the unipotent subgroup by studying its action on the highest weight vector for the irreducible representation $\mathcal{M}_{^{\iota}w.^{\iota}\lambda,\mathbb{C}}$. Fix an ordering for the embeddings of $F$ into  $E$:
\[
    Hom(F,E) = \{\tau_1,\tau_2,\dots,\tau_d\}.
\]
Moreover, fix an ordering for the positive roots of $\Delta(\mathfrak{u}_{P_0})$ as in Section \ref{sec:aaintop}, where $\phi_1 = \alpha$, $\phi_2 = 3\alpha+\beta$, $\phi_3 = 2\alpha+\beta$, $\phi_4 = 3\alpha+2\beta$, and $\phi_5 = \alpha+\beta$. For each $\phi \in \Delta(\mathfrak{u}_{P_0})$, fix a generator $e_{\phi}$ of $X_{\phi}$. Furthermore, define the set of generator $\{e_{\phi}^{*}\}$ for $\mathfrak{u}_{P_0}^{*}$ to be dual of $\{e_{\phi}\}$. For any Kostant representative $w_0$, we have $\Phi_{w_0} = \{\phi>0 \;: \; w_{0}^{-1}\phi<0 \} = \{\phi_1^{w_0},\dots,\phi_l^{w_0}\}$ where $l=l(w_0^{-1})=l(w_0)$. Now define:
\[
    e_{\Phi_{w_0}}^{*} := e_{\phi_{1}^{w_0}}^{*} \wedge \dots \wedge e_{\phi_{l}^{w_0}}^{*} \in \bigwedge^{q_0}(\mathfrak{u}_{P_0}^{*}), \qquad q_0 = l(w_0).
\]
We can define the basis element $e^{\tau}_{\phi}$ by mapping $X_{\phi} \to X^{\tau}_{\phi} = X_{\phi}\otimes_{F,\tau} E$ sending $e_{\phi}$ to $e_{\phi}\otimes 1$. Ordering for the embeddings will give us the ordering $w= \{w^{\tau_1},\dots,w^{\tau_d} \}$ and consequently the following basis:
\[
    e_{\Phi_{w}}^{*} := e_{\Phi_{w^{\tau_1}}}^{*}\wedge \dots \wedge e_{\Phi_{w^{\tau_d}}}^{*} \in  \bigwedge^{q}(\mathfrak{u}_{P}^{*}), \qquad q = l(w).
\]

Similarly for base change to $\mathbb{C}$, fixing $
\iota: E \to \mathbb{C}$:
\[
    e_{\Phi_{^{\iota}w}}^{*} := e_{\Phi_{w^{\iota\circ\tau_1}}}^{*}\wedge \dots \wedge e_{\Phi_{w^{\iota\circ\tau_d}}}^{*} \in  \bigwedge^{q}(\mathfrak{u}_{P}^{*}), \qquad q = l(w).
\]
Back to the irreducible representation $\mathcal{M}_{\lambda^{\tau},E}$ with highest weight $\lambda^{\tau}$. Fix a weight vector $s(\lambda^{\tau}) \in \mathcal{M}_{\lambda^{\tau},E}$, then $s(\lambda) = s(\lambda^{\tau_1})\otimes \dots \otimes s(\lambda^{\tau_d})$ is the highest weight vector for $\mathcal{M}_{\lambda}$. Now we can find the highest weight vector for $H^{q}(\mathfrak{u}_{P},\mathcal{M}_{^{\iota}\lambda,\mathbb{C}})(^{\iota}w)$ as:
\[
    \mathsf{h}(\lambda,w,\iota) := e_{\Phi_{^{\iota}w}}^{*}\otimes s(^{\iota}w\;^{\iota}\lambda).
\]
Following \cite{RagTI}, we define:

\begin{defn}\label{def:sign}
    Fix $\iota:E\to \mathbb{C}$. Let $\gamma \in Gal(\bar{\mathbb{Q}}/\mathbb{Q})$ then
\[
    e^{*}_{\Phi_{^{\gamma\circ\iota}w}} = \epsilon_{\iota,w}(\gamma) e^{*}_{\Phi_{^{\iota}w}},
\]

for a signature $\epsilon_{\iota,w}(\gamma)\in \{\pm1\}$.
\end{defn}

 Fix $\gamma \in Gal(\bar{\mathbb{Q}}/\mathbb{Q})$ and assume $F$ is a totally imaginary field in the \textbf{CM}-case. The Galois action on the induced modules in boundary cohomology and the Eisenstein operator between these modules intertwine up to a sign as defined above \ref{def:sign}. More specifically, we have:
\[\begin{tikzcd}
	{I_{b}^{S}(\sigma)_{w_{cl}} \otimes_{E,\iota} \bar{\mathbb{Q}}} &&&& {I_{b}^{S}(\tilde{\sigma}_f(5))_{w_{cl}} \otimes_{E,\iota} \bar{\mathbb{Q}}} \\
	\\
	{I_{b}^{S}(\sigma)_{w_{cl}} \otimes_{E,\gamma\circ\iota} \bar{\mathbb{Q}}} &&&& {I_{b}^{S}(\tilde{\sigma}_f(5))_{w_{cl}} \otimes_{E,\gamma\circ\iota} \bar{\mathbb{Q}}}
	\arrow["{T_{Eis}(\sigma)\otimes_{E,\iota}\bar{\mathbb{Q}}}", from=1-1, to=1-5]
	\arrow["{1\otimes \gamma}", from=1-1, to=3-1]
	\arrow["{1\otimes \gamma}", from=1-5, to=3-5]
	\arrow["{T_{Eis}(\sigma)\otimes_{E,\gamma\circ\iota} \bar{\mathbb{Q}}}", from=3-1, to=3-5]
\end{tikzcd}\]
The action of the Galois element on the induced module above as a Hecke-submodule of the boundary cohomology will act on their coefficient system $\mathcal{M}_{^{\iota}\lambda,\bar{\mathbb{Q}}}$, as well. As we discussed at the beginning of this section, Galois action introduces a signature as defined in \ref{def:sign}. Therefore, each vertical map in the diagram above introduces a signature and we have:
 \begin{equation}\label{gammaequiv}
     (1\otimes \gamma) \circ (T_{Eis}(\sigma) \otimes_{E,\iota} \bar{\mathbb{Q}}) = \epsilon_{\iota,w_{cl}}(\gamma).\epsilon_{\iota,w{'}_{cl}}. T_{Eis}(\sigma)\otimes_{E,\gamma\circ\iota} \bar{\mathbb{Q}}.
 \end{equation}

\section{The Main Theorem on the Critical Values of Certain $L$-functions}
Now, we can state our main theorems in greater detail, elucidating the relationship between the $L$-functions attached to $GL_2$ as Langlands-Shahidi $L$-functions associated with $G_2$:
\begin{theorem} \label{thm:mainbeta}
     Assume that $F$ contains a $CM$-subfield. Assume $E$ is a large enough finite Galois extension of $\mathbb{Q}$ containing a copy of $F$. Let $G=Res_{F/\mathbb{Q}}(G_2/F)$, and $P_{\beta} = M_{\beta}U_{\beta}$ be the maximal standard parabolic subgroup of $G$ where $\beta \in M_{\beta}$. Let $\mu \in X_{00}^{*}(T\times E)$ and $\sigma_f \in Coh_{!!}(M_{\beta},\mu)$. For an embedding $\iota: E \to \mathbb{C}$, suppose ${^{\iota}\sigma}$ is non-monomial and $m,\; m+1 \in Crit(L(s,Ad^{3}(^{\iota}\sigma))L(2s,\omega_{^{\iota}\sigma}))$ where $m \in \frac{1}{2}+\mathbb{Z}$. Then we have: 
     \begin{enumerate}
         \item If for some $\iota$, $L(m+1, Ad^{3}(^{\iota}\sigma))L(2m+1,\omega_{^{\iota}\sigma})=0$ then in fact $L(m+1, Ad^{3}(^{\iota}\sigma))=0$. Therefore, $m+1-a_1(\mu) = \frac{1}{2}$, and
         \[
         L(m+1,Ad^{3}(^{\iota}\sigma))= L(\frac{1}{2},Ad^{3}(^{\iota}\sigma^u))=0
         \]
         holds for every embedding $\iota$.

         \item Suppose $L(m+1,Ad^{3}(^{\iota}\sigma)) \neq 0$; then
         \[
             |\delta_{F/\mathbb{Q}}|^{5/2} \frac{L(m,Ad^3(^{\iota}\sigma))L(2m,\omega_{^{\iota}\sigma})}{L(m+1,Ad^3(^{\iota}\sigma))L(2m+1,\omega_{^{\iota}\sigma})}\in \iota(E) \subset \bar{\mathbb{Q}}.
            \]

        \item  For any $\gamma \in Gal(\bar{\mathbb{Q}}/\mathbb{Q})$:
   \begin{multline*}
       \gamma\left(|\delta_{F/\mathbb{Q}}|^{5/2} \frac{L(m,Ad^3(^{\iota}\sigma))L(2m,\omega_{^{\iota}\sigma})}{L(m+1,Ad^3(^{\iota}\sigma))L(2m+1,\omega_{^{\iota}\sigma})}\right) = \\ \epsilon(\gamma, ^{\iota}w)\epsilon(\gamma, ^{\iota}w')|\delta_{F/\mathbb{Q}}|^{5/2} \frac{L(m,Ad^3(^{\gamma\circ\iota}\sigma))L(2m,\omega_{^{\gamma\circ\iota}\sigma})}{L(m+1,Ad^3(^{\gamma\circ\iota}\sigma))L(2m+1,\omega_{^{\gamma\circ\iota}\sigma})},
   \end{multline*}
   where $w'\in W^{P}$ is the balanced Kostant representative from Combinatorial lemma, and $\epsilon(\gamma, ^{\iota}w)$ and $\epsilon(\gamma, ^{\iota}w')$ are the signs introduced in the previous section.
     \end{enumerate}
\end{theorem}

\begin{proof}
As we discussed in Remark \ref{nomore}, it is enough to prove the theorem only at the point of evaluation $m=-\frac{5}{2}$. Moreover, we can assume that $^{\iota}\mu$ is on the right of the unitary axis which means $\mathbf{w}\leq -5$, as we can always use the functional equation for our $L$-functions \cite{ShBook} to be in the right side of the unitary axis. Then:
\begin{align*}
    1-\frac{5}{2}-\frac{\mathbf{w}}{2}\geq 1.
\end{align*}
Therefore, as $L(s,Ad^{3}(^{\iota}\sigma^u))$ only vanishes at $s=\frac{1}{2}$:
\begin{align*}
    L(1-\frac{5}{2},Ad^{3}(^{\iota}\sigma))= L(1-\frac{5}{2}-\frac{\mathbf{w}}{2},Ad^{3}(^{\iota}\sigma^u)) \neq 0.
\end{align*}
Given $\gamma \in Gal(\bar{\mathbb{Q}}/\mathbb{Q})$, the action of $\gamma$ preserves the property of being on the right of the unitary axis. It is the consequence of strong purity that $\mathbf{w}$ is invariant under the Galois action. Therefore, we proved the part (1) of the theorem. 

To prove part (2), as $\mu$ is on the right side of the unitary axis. Let $\mathfrak{J}^{b}(\sigma_f)$ be as defined in (\ref{frakj}), which is a $\mathsf{k}$-dimensional  $E$-subspace of the $2\mathsf{k}$-dimensional $E$-vector space $I_{b}^{S}(\sigma_f)_{w_{cl}}\oplus I_{b}^{S}(\tilde{\sigma}_f(5))_{w^{'}_{cl}}$ by Theorem \ref{thm:eis}. There is a $E$-linear intertwining operator $T_{Eis}: I_{b}^{S}(\sigma_f)_{w_{cl}} \to I_{b}^{S}(\tilde{\sigma}_f(5))_{w^{'}_{cl}}$ defined by the proof of \ref{thm:eis} such that:
\[
    \mathfrak{J}^{b}(\sigma_f) = \{(\xi, \xi+T_{Eis}(\xi) \; | \; \xi \in I_{b}^{S}(\sigma_f)_{w_{cl}} \}.
\]
Now going to the transcendental level by $\iota: E \to \mathbb{C}$, we get the map $T_{Eis}\otimes \mathbb{C}$. Now, we want to decompose this intertwining operator to the local factors. For non-archimedean places $v\notin S$, let $T_{st,\iota} = T_{st,E}\otimes_{E,\iota} \mathbb{C}$ be the arithmetic intertwining operator as in \ref{thm:nonarchint}. Appealing to Langlands generalization of Gindikin-Karpelevich's formula \cite{Lang}, define the local intertwining operator at the evaluation point as:
\[
    T_{loc,\iota} = \bigotimes_{v\notin S} (\frac{L(-\frac{5}{2},Ad^{3}(\sigma_v))L(-2\frac{5}{2},\omega_{\sigma_v})}{L(1-\frac{5}{2},Ad^{3}(\sigma_v))L(1-2\frac{5}{2},\omega_{\sigma_v})})^{-1}T_{st,\iota,v}.
\]
By transferring $T_{st,\iota,v}$ to arithmetic level $T_{st,E,v}$, we can define the arithmetic counterpart of the local intertwining operator at $v\notin S$ which is defined over $E$, and denote it by $T_{loc,E}$ we have:
\begin{equation}\label{notininter}
    T^{S}_{st}:= \bigotimes_{v\notin S} T_{st,v} =  \frac{L(-\frac{5}{2},Ad^{3}(\sigma))L(-2\frac{5}{2},\omega_{\sigma})}{L(1-\frac{5}{2},Ad^{3}(\sigma))L(1-2\frac{5}{2},\omega_{\sigma})} (T_{loc,E}\otimes_{E,\iota} \mathbb{C}).
\end{equation}
Theorem \ref{thm:nonarchint} covers non-archimedean places in $S$ as well. We can formally normalize these intertwining operators by multiplying them with the ratio of the $L$-factors obtained via the Langlands-Shahidi method, as defined by Shahidi in \cite{Shahidi1990APO}. Following the notation of \cite{HR}, we have:
\[
    T_{norm,\iota} = \bigotimes_{v\in S\backslash S_{\infty}} (\frac{L(-\frac{5}{2},Ad^{3}(\sigma_v))L(-2\frac{5}{2},\omega_{\sigma_v})}{L(1-\frac{5}{2},Ad^{3}(\sigma_v))L(1-2\frac{5}{2},\omega_{\sigma_v})})^{-1}T_{st,\iota,v}.
\]
Form the arithmetic counterpart of this normalized intertwining operator in the same way as the non-ramified case, denoted by $T_{norm,E}$, such that for $v\in S\backslash S_{\infty}$:
\begin{equation}
    T_{norm,\iota}:= \otimes_{v\in S\backslash S_{\infty}}\frac{L(-\frac{5}{2},Ad^{3}(\sigma))L(-2\frac{5}{2},\omega_{\sigma})}{L(1-\frac{5}{2},Ad^{3}(\sigma))L(1-2\frac{5}{2},\omega_{\sigma})} (T_{norm,E,v}\otimes_{E,\iota} \mathbb{C}).
\end{equation}

For $s\in S_{\infty}$ appealing to the Proposition \ref{archinterthm}, we have the arithmetic counterpart of the intertwining operators at archimedean places with normalizing factor match with the $L$-factors at infinity. Summarizing our argument above plus \ref{intertwinformula}, we can express the intertwining operator $T_{Eis}(\sigma)\otimes_{E,\iota} \mathbb{C}$ as follow:
\begin{equation}\label{arithinter}
    T_{Eis}(\sigma) \otimes \mathbb{C} = \delta_{F/\mathbb{Q}} ^{-5/2} \frac{L^{S}(s,Ad^{3}(\sigma))L^{S}(2s,\omega_{\sigma})}{L^{S}(1+s,Ad^{3}(\sigma))L^{S}(1+2s,\omega_{\sigma})} (T_{norm,E}\otimes T_{loc,E}) \otimes_{E,\iota} \mathbb{C}.
\end{equation}

Since the local intertwining operators \ref{arithinter} are defined over $E$, we deduce that 
\[
\delta_{F/\mathbb{Q}} ^{-5/2} \frac{L^{S}(s,Ad^{3}(\sigma))L^{S}(2s,\omega_{\sigma})}{L^{S}(1+s,Ad^{3}(\sigma))L^{S}(1+2s,\omega_{\sigma})} \in \iota(E).
\]
This concludes the proof of part (2).

To prove part (3), by applying \ref{gammaequiv} to the \ref{arithinter} we will get the reciprocity results in part (3) for the ratio of $L$-factors.
\end{proof}

Finally, we can state the main theorem for the case $P_{\alpha}$ as well. 

\begin{theorem} \label{thm:mainalpha}
     Let $M_{\alpha}$ be the Levi factor of the maximal standard parabolic subgroup $P_{\alpha}$. Consider a strongly pure weight $\mu \in X_{00}^{*}(T\times E)$ and let $\sigma_f \in Coh_{!!}(M_{\alpha},\mu)$. For an embedding $\iota: E \to \mathbb{C}$, suppose $m, m+1 \in Crit(L(s, \;^{\iota}\sigma)L(2s,\omega_{^{\iota}\sigma})L(3s, \;^{\iota}\sigma\otimes \omega_{^{\iota}\sigma}))$ where $m \in \frac{1}{2}+\mathbb{Z}$. 
     \begin{enumerate}
         \item If for some $\iota$, the product $L(m+1, \;^{\iota}\sigma)L(2m+1,\omega_{^{\iota}\sigma})L(3m+1, \;^{\iota}\sigma\otimes \omega_{^{\iota}\sigma})=0$, then in fact $L(m+1,\;^{\iota}\sigma)=0$. Then $m+1-a_1(\mu) = \frac{1}{2}$ and
         \[
              L(m+1,\;^{\iota}\sigma)= L(\frac{1}{2},\;^{\iota}\sigma^u) =0, \qquad \forall \iota:E\to \mathbb{C}.
         \]
         \item Suppose $L(m+1,\;^{\iota}\sigma)\neq0$; then:
         \[
             |\delta_{F/\mathbb{Q}}|^{5/2} \frac{L(m, \;^{\iota}\sigma)L(2m,\omega_{^{\iota}\sigma})L(3m, \;^{\iota}\sigma\otimes \omega_{^{\iota}\sigma})}{L(m+1, \;^{\iota}\sigma)L(2m+1,\omega_{^{\iota}\sigma})L(3m+1, \;^{\iota}\sigma\otimes \omega_{^{\iota}\sigma})}\in \iota(E) \subset \bar{\mathbb{Q}}.
            \]
        \item  For any $\gamma \in Gal(\bar{\mathbb{Q}}/\mathbb{Q})$:
   \begin{multline*}
       \gamma\left(|\delta_{F/\mathbb{Q}}|^{5/2} \frac{L(m, \;^{\iota}\sigma)L(2m,\omega_{^{\iota}\sigma})L(3m, \;^{\iota}\sigma\otimes \omega_{^{\iota}\sigma})}{L(m+1, \;^{\iota}\sigma)L(2m+1,\omega_{^{\iota}\sigma})L(3m+1, \;^{\iota}\sigma\otimes \omega_{^{\iota}\sigma})}\right) = \\ \epsilon(\gamma, ^{\iota}w)\epsilon(\gamma, ^{\iota}w')|\delta_{F/\mathbb{Q}}|^{5/2} \frac{L(m, \;^{\gamma\circ\iota}\sigma)L(2m,\omega_{^{\gamma\circ\iota}\sigma})L(3m, \;^{\gamma\circ\iota}\sigma\otimes \omega_{^{\gamma\circ\iota}\sigma})}{L(m+1, \;^{\gamma\circ\iota}\sigma)L(2m+1,\omega_{^{\gamma\circ\iota}\sigma})L(3m+1, \;^{\gamma\circ\iota}\sigma\otimes \omega_{^{\gamma\circ\iota}\sigma})}.
   \end{multline*}
   where $w'\in W^{P}$ is the balanced Kostant representative obtained by Combinatorial lemma, and $\epsilon(\gamma, \;^{\iota}w)$ and $\epsilon(\gamma, \;^{\iota}w')$ are the signs introduced in the previous section.
     \end{enumerate}
\end{theorem}
\begin{proof}
    We need to make sure that the denominator does not vanish in the critical half-integer $m$. By \cite{JaSh76}, we know that the Jaquet-Langlands $L$-functions are non-vanishing for $Re(s)\geq 1$. Therefore, by the functional equation, zeros only occur in the interval $(0,1)$. Thus, the $L$-function
\[
    L(3s,\;^{\iota}\sigma\otimes \omega_{^{\iota}\sigma}) = L(3s-\frac{3\mathbf{w}}{2}, \;^{\iota}\sigma_{u} \otimes \omega_{^{\iota}\sigma^u}),
\]
only vanishes possibly for 
\[
0<3s-\frac{3\mathbf{w}}{2}<1 \iff \frac{\mathbf{w}}{2}< s < \frac{\mathbf{w}}{2}+\frac{1}{3}.
\]
Since the interval above does not contain any half-integers, this $L$-function does not contribute to the vanishing of the denominator at critical points. The $L$-function 
\[
L(s,\;^{\iota}\sigma) = L(s-\frac{\mathbf{w}}{2}, \;^{\iota}\sigma_{u}),
\]
has a possible zero at the half-integer $s = \frac{\mathbf{w}+1}{2}$. We now follow the proof of Theorem \ref{thm:mainbeta}. It is sufficient to consider the point of evaluation $m=-\frac{3}{2}$, and assume that $^{\iota}\mu$ is on the right of unitary axis; i.e. $\mathbf{w} \leq -3$. Therefore,
\[
    1-\frac{3}{2}-\frac{\mathbf{w}}{2} \geq 1,
\]
where the $L$-function in hand is non-vanishing. The rest of the proof will follow the same steps as the proof of Theorem \ref{thm:mainbeta}.
\end{proof}

\bibliography{main}{}
\bibliographystyle{plain}
\end{document}